\documentclass[12pt,a4paper,reqno]{amsart}
\usepackage{amssymb}
\usepackage{amscd}

\numberwithin{equation}{section}

\theoremstyle{plain}

\newtheorem{theorem}[subsection]{Theorem}
\newtheorem{proposition}[subsection]{Proposition}
\newtheorem{lemma}[subsection]{Lemma}
\newtheorem{corollary}[subsection]{Corollary}

\newtheorem{claim}[subsection]{Claim}

\theoremstyle{definition}

\newtheorem{definition}[subsection]{Definition}

\renewcommand{\leq}{\leqslant}
\renewcommand{\geq}{\geqslant}

\newsavebox{\proofbox}
\savebox{\proofbox}{\begin{picture}(7,7)%
  \put(0,0){\framebox(7,7){}}\end{picture}}

\newcommand{\md}[1]{\ensuremath{(\operatorname{mod}\, #1)}}
\newcommand{\mdlem}[1]{\ensuremath{(\mbox{\textup{mod}}\, #1)}}
\newcommand\E{\mathbb{E}}
\newcommand\Z{\mathbb{Z}}
\newcommand\R{\mathbb{R}}

\newcommand\C{\mathbb{C}}
\newcommand\N{\mathbb{N}}

\newcommand\Q{\mathbb{Q}}

\newcommand\g{\mathfrak{g}}
\newcommand\h{\mathfrak{h}}
\newcommand\X{\mathcal{X}}
\newcommand\id{\operatorname{id}}
\newcommand\mes{\operatorname{mes}}
\newcommand\vol{\operatorname{vol}}
\newcommand\Lip{{\operatorname{Lip}}}
\newcommand\eps{\varepsilon}
\newcommand\Ad{\operatorname{Ad}}
\newcommand\codim{\operatorname{codim}}
\newcommand\dist{\operatorname{dist}}

\newcommand\ab{\operatorname{ab}}

\newcommand\poly{\operatorname{poly}}

\newcommand\Span{\operatorname{Span}}

\newcommand\lin{\operatorname{lin}}

\newcommand\HK{\operatorname{HK}}
\newcommand\nonlin{2}

\def\proof{\textit{Proof. }}
\def\example{\textit{Example. }}
\def\examples{\textit{Examples. }}
\def\remark{\textit{Remark. }}
\def\remarks{\textit{Remarks. }}
\def\endproof{\hfill{\usebox{\proofbox}}}
\def\endexample{\hfill{$\diamond$}}

\begin{document}

\title{The quantitative behaviour of polynomial orbits on nilmanifolds}

\author{Ben Green}
\address{Centre for Mathematical Sciences\\
Wilberforce Road\\
     Cambridge CB3 0WA\\
     England
}
\email{b.j.green@dpmms.cam.ac.uk}

\author{Terence Tao}
\address{UCLA Department of Mathematics, Los Angeles, CA 90095-1596.
}
\email{tao@math.ucla.edu}

\thanks{The first author is a Clay Research Fellow and gratefully acknowledges the support of the Clay Institute. The second author is supported by a grant from the MacArthur Foundation.}


\begin{abstract}
A theorem of Leibman \cite{leibman-single-poly} asserts that a polynomial orbit $(g(n)\Gamma)_{n \in \Z}$ on a nilmanifold $G/\Gamma$ is always equidistributed in a union of closed sub-nilmanifolds of $G/\Gamma$. In this paper we give a quantitative version of Leibman's result, describing the uniform distribution properties of a finite polynomial orbit $(g(n)\Gamma)_{n \in [N]}$ in a nilmanifold. More specifically we show that there is a factorization $g = \varepsilon g' \gamma$, where $\varepsilon(n)$ is ``smooth'', $(\gamma(n)\Gamma)_{n \in \Z}$ is periodic and ``rational'', and $(g'(n)\Gamma)_{n \in P}$ is uniformly distributed (up to a specified error $\delta$) inside some subnilmanifold $G'/\Gamma'$ of $G/\Gamma$ for all sufficiently dense arithmetic progressions $P \subseteq [N]$.

Our bounds are uniform in $N$ and are polynomial in the error tolerance $\delta$. In a companion paper \cite{ukmobius} we shall use this theorem to establish the M\"obius and Nilsequences conjecture from our earlier paper \cite{green-tao-linearprimes}.
\end{abstract}

\maketitle

\section{Introduction}

\textsc{Nilmanifolds.}
In the last few years it has come to be appreciated that \emph{nilmanifolds}, together with orbits on them, play a fundamental r\^ole in combinatorial number theory. Their relevance was certainly apparent in \cite{furstenberg-vonneumann}, and it has been displayed quite dramatically in recent ergodic-theoretic work of Host-Kra \cite{host-kra} and Ziegler \cite{ziegler}. More recently the authors have explored how nilmanifolds arise in additive combinatorics \cite{green-tao-u3inverse} and in the study of linear equations in the primes \cite{green-tao-linearprimes}. The present paper is a part of that programme (and in particular will be used to prove the \emph{M\"obius and Nilsequences conjecture} from \cite{green-tao-linearprimes} in the companion \cite{ukmobius} to this paper) but, since it concerns only the intrinsic properties of nilmanifolds, may be read independently of any of the other work. The reader interested in the background may consult the surveys \cite{green-icm,kra-survey,tao-survey-1} or the paper \cite{green-tao-linearprimes}.

We begin by setting out our notation for nilmanifolds. 

\begin{definition}[Filtrations and Nilmanifolds]\label{nil-def}
Let $G$ be a connected, simply connected Lie group with identity element $\id_G$. For the purposes of this paper we define a \emph{filtration} $G_{\bullet}$ on $G$ to be a sequence of closed connected subgroups 
\[ G = G_0 = G_1 \supseteq G_2 \supseteq \dots \supseteq G_d \supseteq G_{d+1} = \{\id_G\}\] 
which has the property that $[G_i,G_j] \subseteq G_{i+j}$ for all integers $i,j \geq 0$. The least integer $d$ for which $G_{d+1} = \{\id_G\}$ is called the \emph{degree} of the filtration $G_{\bullet}$ and here, as usual, the commutator group $[H,K]$ is the group generated by $\{ [h,k] : h \in H, k \in K \}$, where $[h,k] := h k h^{-1}k^{-1}$ is the commutator of $h$ and $k$. If $G$ possesses a filtration then we say that $G$ is \emph{nilpotent}.
Let $\Gamma \subseteq G$ be a uniform subgroup (i.e. a discrete, cocompact subgroup). Then the quotient $G/\Gamma = \{ g\Gamma: g \in G \}$ is called a \emph{nilmanifold}.  We also write $g \md{\Gamma}$ for $g\Gamma$.
\end{definition}

Throughout the paper we will write $m = \dim G$ and $m_i = \dim G_i$, $i = 1,\dots, d$.

\remark  The assumptions of connectedness and simple-connectedness for $G$ are not completely standard, but are very convenient for us. In any situation in which we apply our theorems, we expect to be able to reduce to this case. If a filtration $G_{\bullet}$ of degree $d$ exists then it is easy to see that the \emph{lower central series filtration\footnote{It is not hard to see that the lower central series filtration \emph{is} a filtration, in that we have $[G_i,G_j] \subseteq G_{i+j}$ for all $i,j$.}} defined by $G = G_0 = G_1$, $G_{i+1} = [G, G_i]$ terminates with $G_{s+1} = \{\id_G\}$ for some integer $s \leq d$. We call the minimal such integer $s$ the \emph{step} of the nilpotent Lie group $G$. In this paper the degree $d$ will play a vastly more important r\^ole than the step $s$, since it will be important to work with filtrations more general than the lower central series. 

\examples   The simplest examples of nilmanifolds arise when $s = 1$ in which case we may, after a linear transformation, take $G = \R^m$ and $\Gamma = \Z^m$. The lower central series filtration is given by $G = G_0 = G_1$ and $G_2 = \{\id_G\}$. The nilmanifold $G/\Gamma$ is then referred to as a \emph{torus}. Note that in this example the group operation is written \emph{additively}, as is conventional for abelian groups. When we are working with non-abelian groups we shall write the group operation multiplicatively.
The simplest non-abelian example is given by the 3-dimensional \emph{Heisenberg nilmanifold}, in which $s = 2$. We will study this object in some detail later on. Here we take 
\begin{equation}\label{heisen}
G = \left(\begin{smallmatrix} 1 & \R & \R \\ 0 & 1 & \R \\  0 & 0 & 1 \end{smallmatrix}\right) \quad \mbox{and} \quad \Gamma = \left(\begin{smallmatrix} 1 & \Z & \Z \\ 0 & 1 & \Z \\  0 & 0 & 1 \end{smallmatrix}\right).
\end{equation}
The lower central series filtration is given by $G = G_0 = G_1$, 
\[ G_2 = \left(\begin{smallmatrix} 1 & 0 & \R \\ 0 & 1 & 0 \\  0 & 0 & 1 \end{smallmatrix}\right)\] and $G_3 = \{\id_G\}$.
Observe that a fundamental domain for the action of $\Gamma$ on $G$ is 
\begin{equation} \left\{ \left(\begin{smallmatrix} 1 & x_1 & x_2 \\ 0 & 1 & x_3 \\  0 & 0 & 1 \end{smallmatrix}\right) : 0 \leq x_1, x_2, x_3 < 1\right\}.\end{equation}
Thus one can view $G/\Gamma$ as a unit cube, with the sides glued together in a twisted fashion.\endexample

This paper will be concerned with the qualitative and quantitative equidistribution of various algebraic sequences on nilmanifolds.  We first set out our notation for equidistribution.  

\begin{definition}[Equidistribution]\label{almost-equidistribution}  Let $G/\Gamma$ be a nilmanifold. Here and in the sequel we endow $G/\Gamma$ with the unique normalised Haar measure, we let $[N] := \{n \in \Z: 1 \leq n \leq N\}$, and we write $\E_{a \in A} f(a) := \frac{1}{|A|} \sum_{a \in A} f(A)$ for the average of $f$ on the set $A$.  
\begin{enumerate}
\item An infinite sequence $(g(n)\Gamma)_{n \in \N}$ in $G/\Gamma$ is said to be \emph{equidistributed} if we have
$$
\lim_{N \to \infty} \E_{n \in [N]} F(g(n) \Gamma) = \int_{G/\Gamma} F
$$
for all continuous functions $F: G/\Gamma \to \C$.  
\item An infinite sequence $(g(n)\Gamma)_{n \in \Z}$ in $G/\Gamma$ is said to be \emph{totally equidistributed} if the sequences $(g(an+r)\Gamma)_{n \in \N}$ are equidistributed for all $a \in \Z \backslash \{0\}$ and $r \in \Z$.
\item 
Given a length $N > 0$ and an error tolerance $\delta > 0$, a finite sequence $(g(n)\Gamma)_{n \in [N]}$ is said to be \emph{$\delta$-equidistributed} if we have
$$ \left|\E_{n \in [N]} F(g(n) \Gamma) - \int_{G/\Gamma} F\right| \leq \delta \|F\|_{\Lip}$$
for all Lipschitz functions $F: G/\Gamma \to \C$, where
\[ \Vert F \Vert_{\Lip} := \Vert F \Vert_{\infty} + \sup_{x,y \in G/\Gamma, x \neq y} \frac{|F(x) - F(y)|}{d_{G/\Gamma}(x,y)}\]
and the metric $d_{G/\Gamma}$ on $G/\Gamma$ will be defined in Definition \ref{metric-def} in the next section (it will involve choosing a \emph{Mal'cev basis} $\X$ for $G/\Gamma$).  
\item A finite sequence $(g(n)\Gamma)_{n \in [N]}$ is said to be \emph{totally $\delta$-equidistributed} if we have
$$ \left|\E_{n \in P} F(g(n) \Gamma) - \int_{G/\Gamma} F\right| \leq \delta \|F\|_{\Lip}$$
for all Lipschitz functions $F: G/\Gamma \to \C$ and all arithmetic progressions $P \subset [N]$ of length at least $\delta N$.
\end{enumerate}
\end{definition}

We will be interested in the \emph{qualitative} question of when a sequence $(g(n)\Gamma)_{n \in \N}$ is equidistributed (or totally equidistributed), as well as the more \emph{quantitative} question of when a finite sequence $(g(n)\Gamma)_{n \in [N]}$ is $\delta$-equidistributed (or totally $\delta$-equidistributed).  Such questions, and corresponding questions in more general settings (for example when $G/\Gamma$ is a homogeneous space of a general, not necessarily nilpotent, Lie group) play a fundamental r\^ole in number theory; see \cite{venkatesh} for a discussion. 
These questions are also closely related to the celebrated theorem of Ratner \cite{ratner} on unipotent flows, although as we are restricting attention to nilmanifolds, we will not need the full force of Ratner's theorem (or quantitative versions thereof) here.

\textsc{Qualitative equidistribution theory of linear sequences.}
To begin the discussion let us first restrict attention to \emph{linear} sequences.

\begin{definition}[Linear sequences]  A \emph{linear sequence} in a group $G$ is any sequence $g: \Z \to G$ of the form $g(n) := a^n x$ for some $a, x \in G$.  A \emph{linear sequence} in a nilmanifold $G/\Gamma$ is a sequence of the form $(g(n)\Gamma)_{n \in \Z}$, where $g: \Z \to G$ is a linear sequence in $G$.
\end{definition}

In the additive case $G = \R^m$, $\Gamma = \Z^m$, a linear sequence takes the form $(an+x\md{\Z^m})_{n \in \Z}$.  In this case one can understand equidistribution satisfactorily using Kronecker's theorem and its variants.  For instance, to answer qualitative questions about equidistribution in this case, we have the following classical result.

\begin{theorem}[Qualitative Kronecker theorem]
Let $m \geq 1$, and let $(g(n)\md{\Z^m})_{n \in \N}$ be a linear sequence in the torus $\R^m/\Z^m$.  Then exactly one of the following statements is true.
\begin{itemize}
\item[(i)] $(g(n)\md{\Z^m})_{n \in \N}$ is equidistributed in $\R^m/\Z^m$. 
\item[(ii)] There exists a non-trivial \emph{character} $\eta: \R^m \to \R/\Z$, i.e. a continuous additive homomorphism which annihilates $\Z^m$ but does not vanish entirely, such that $\eta \circ g$ is constant.  \textup{(}Equivalently, if $g(n) = an+x$, there exists a non-zero $k \in \Z^m$ such that $k \cdot a \in \Z$.\textup{)} 
\end{itemize}
In particular, $(g(n)\md{\Z^m})_{n \in \Z}$ is equidistributed if and only if it is totally equidistributed.
\end{theorem}

\remarks  An equivalent formulation of this theorem is that if the linear sequence \[ (g(n)\md{\Z^m})_{n \in \N}\]is not equidistributed, then this sequence instead takes values in a finite union of proper subtori of $G/\Gamma$.  This can be viewed as an extremely simple special case of the theorems of Ratner \cite{ratner} and Shah \cite{shah}.
More quantitative results can be obtained via Fourier analysis\footnote{In this simple setting one could also use more classical tools such as Minkowski's geometry of numbers, and in the $m=1$ case one could even use continued fractions.  However, these methods do not seem to extend easily to higher steps.}; see Proposition \ref{quant-kron} below.

A remarkable theorem of Leon Green allows one to reduce qualitative questions about the distribution of orbits on nilmanifolds of step $s > 1$ to the abelian case just described. 

\begin{definition}[Horizontal torus]\label{horiz-def} Given a nilmanifold $G/\Gamma$, the \emph{horizontal torus} is defined to be $(G/\Gamma)_{\ab} := G/[G,G]\Gamma$.  We let $\pi: G \to (G/\Gamma)_{\ab}$ be the canonical projection map.  A \emph{horizontal character} is a continuous additive homomorphism $\eta: G \to \R/\Z$ which annihilates $\Gamma$; observe that such characters in fact annihilate $[G,G]\Gamma$ and so can be viewed as characters on the horizontal torus.  We say that a horizontal character is \emph{non-trivial} if it is not identically zero.
\end{definition}

It follows from results of Mal'cev \cite{malcev}, and in particular the existence of so-called \emph{Mal'cev bases}, that $(G/\Gamma)_{\ab}$ really is a torus and in fact is isomorphic to $\R^{m_{\ab}}/\Z^{m_{\ab}}$ where $m_{\ab} := \dim_{\R}(G) - \dim_{\R}([G,G])$. We will not actually need this characterisation, as the properties of horizontal characters $\eta : G \rightarrow \R/\Z$ will be our main focus. Readers may find it useful to keep this in mind, however.

\begin{theorem}[Leon Green's theorem]\label{leon-green} 
Let $(g(n)\Gamma)_{n \in \Z}$ be a linear sequence in a nilmanifold $G/\Gamma$.  Then the orbit $(g(n)\Gamma)_{n \in \N}$ is equidistributed in $G/\Gamma$ if and only if the projected orbit $(\pi(g(n)\Gamma))_{n \in \N}$ is equidistributed in the horizontal torus $(G/\Gamma)_{\ab}$.  \textup{(}In particular, $(g(n)\Gamma)_{n \in \Z}$ is equidistributed if and only if it is totally equidistributed.\textup{)}
\end{theorem}

\begin{proof}
See \cite{auslander-green-hahn,lgreen}.  Leon Green used representation theory to establish his result, but a more elementary proof was subsequently found by Parry \cite{parry}.
\end{proof}

\example Suppose that $G/\Gamma$ is the Heisenberg example \eqref{heisen}. Then 
\[ [G,G] = \left(\begin{smallmatrix} 1 & 0 & \R \\ 0 & 1 & 0 \\  0 & 0 & 1 \end{smallmatrix}\right)\] and $(G/\Gamma)_{\ab}$ may be identified with $\R^2/\Z^2$, the projection $\pi$ being given by
\[ \pi \left[ \left(\begin{smallmatrix} 1 & x_1 & x_2\\ 0 & 1 & x_3 \\  0 & 0 & 1 \end{smallmatrix}\right)\right] := (x_1,x_3).\]
Leon Green's theorem implies that the orbit $(a^n\Gamma)_{n \in \N}$, where 
\[ a =  \left(\begin{smallmatrix} 1 & \alpha_1 & \alpha_2\\ 0 & 1 & \alpha_3 \\  0 & 0 & 1 \end{smallmatrix}\right),\]
is equidistributed in $G/\Gamma$ if and only if $1,\alpha_1$ and $\alpha_3$ are independent over $\Q$. It is already somewhat nontrivial to establish this result directly. \endexample

By Kronecker's theorem, we can then recast Theorem \ref{leon-green} in the following equivalent formulation:

\begin{theorem}[Leon Green's theorem, again]\label{leon-green-again} 
Let $(g(n)\Gamma)_{n \in \Z}$ be a linear sequence in a nilmanifold $G/\Gamma$.   
Then exactly one of the following statements is true:
\begin{itemize}
\item[(i)] $(g(n)\Gamma)_{n \in \N}$ is equidistributed in $G/\Gamma$. 
\item[(ii)] There exists a non-trivial horizontal character $\eta: G \to \R/\Z$ such that $\eta \circ g$ is constant.
\end{itemize}
\end{theorem}

\textsc{Qualitative equidistribution theory of polynomial sequences.}
While our primary applications are concerned with linear sequences, it turns out for various technical reasons that it is important to work in the more general class of \emph{polynomial sequences}.  

\begin{definition}[Polynomial sequences in nilpotent groups]\label{poly-def}
Suppose that $G$ is a nilpotent group with a filtration $G_{\bullet}$. Let $g : \Z \rightarrow G$ be a sequence. If $h \in \Z$ we write $\partial_h g := g(n+h) g(n)^{-1}$. We say that $g$ is a polynomial sequence with coefficients in $G_{\bullet}$, and write $g \in \poly(\Z,G_{\bullet})$, if $\partial_{h_i} \dots \partial_{h_1} g$ takes values in $G_i$ for all positive integers $i$ and for all choices of $h_1,\dots,h_i \in \Z$. In this case we say that $g$ has degree $d$. If $g$ lies in $\poly(G_{\bullet})$ for some filtration $G_{\bullet}$ then we simply say that $g$ is a polynomial sequence.
\end{definition}

This definition is a little abstract. However we will show in \S \ref{poly-sequences-sec} that $g : \Z \rightarrow G$ is a polynomial sequence if and only if $g$ has the form $g(n) = a_1^{p_1(n)} \dots a_k^{p_k(n)}$, where $a_1,\dots,a_k \in G$ and the $p_i : \N \rightarrow \N$ are polynomials. In particular a linear sequence $g(n) = a^n x$ is a polynomial sequence, and in fact since $\partial_{h_1} g(n) = a^{h_1}$ and $\partial_{h_2}\partial_{h_1} g(n) = \id_G$ it is clear that such a sequence has coefficients in the lower central series filtration $G_{\bullet}$. Note carefully that the degree of a linear sequence is equal to the step $s$ of the underlying Lie group $G$, and is not equal to one as the name ``linear'' might suggest.

A remarkable result of Lazard and Leibman \cite{lazard,leibman-group-1,leibman-group-2} asserts that $\poly(\Z,G_{\bullet})$ is a group. We will prove this in \S \ref{poly-sequences-sec}, and it will play a key r\^ole in several of our arguments.

Theorem \ref{leon-green} was extended by Liebman \cite{leibman-single-poly} to the case when $g(n)$ is a polynomial sequence rather than a linear one.  In particular, he showed the following generalisation of Theorem \ref{leon-green-again}.

\begin{theorem}[Leibman's theorem]\label{leibman-thm}\cite{leibman-single-poly}
Suppose that $G/\Gamma$ is a nilmanifold. and that $g : \Z \rightarrow G$ is a polynomial sequence. Then exactly one of the following statements is true:
\begin{itemize}
\item[(i)] $(g(n)\Gamma)_{n \in \N}$ is equidistributed in $G/\Gamma$. 
\item[(ii)] There exists a non-trivial horizontal character $\eta: G \to \R/\Z$ such that $\eta \circ g$ is constant.  
\end{itemize}
\end{theorem}

\remark  This theorem significantly generalizes the classical theorem of Weyl that a polynomial sequence in $\R/\Z$ is equidistributed unless all of its non-constant coefficients are rational.
We will in fact use a quantitative version of Weyl's theorem in our arguments; see Proposition \ref{weyl-dichotomy} below.

We can iterate this theorem to establish a factorization result.  We first need some notation.

\begin{definition}[Rational subgroup]\label{rat-sub-def}  Let $G/\Gamma$ be a nilmanifold.  A \emph{rational subgroup} of $G$ is a closed connected subgroup $G'$ of $G$ such that $G'\Gamma/\Gamma \cong G'/\Gamma' =  G'/(G' \cap \Gamma)$ is a closed submanifold of $G/\Gamma$ (or equivalently, that $\Gamma'$ is a cocompact subgroup of $G'$).  We say that $G'$ is \emph{proper} if $G' \neq G$.
\end{definition}

\example  If $G/\Gamma$ is a nilmanifold (that is to say if there \emph{exists} a uniform subgroup $\Gamma \leq G$) one can show that each member $G_i$ of the lower central series is a rational subgroup; see e.g. \cite{corwin-greenleaf} or \cite{malcev}.  
\endexample

\begin{definition}[Rational sequence] \label{rat-def} Let $G/\Gamma$ be a nilmanifold.  A \emph{rational group element} is any $g \in G$ such that $g^r \in \Gamma$ for some integer $r > 0$.  A \emph{rational point} is any point in $G/\Gamma$ of the form $g\Gamma$ for some rational group element $g$.  A sequence $(g(n)\Gamma)_{n \in \Z}$ is \emph{rational} if every element $g(n)\Gamma$ in the sequence is a rational point.
\end{definition}

\remark 
It is not difficult to show that the rational group elements form a dense subgroup of $G$ that contains $\Gamma$; see Lemma \ref{ratpoint}.  We will show in Lemma \ref{rat-poly-lem} that any polynomial sequence in $G/\Gamma$ which is rational is automatically periodic.

\begin{corollary}[Factorization theorem for polynomial sequences]\label{ratnil-factor}  Let $(g(n)\Gamma)_{n \in \Z}$ be a polynomial sequence in a nilmanifold $G/\Gamma$.   Then there exists a rational subgroup $G'$ of $G$ and a factorization $g = \eps g' \gamma$, where $\eps \in G$ is a constant, $g': \Z \to G'$ is a polynomial sequence such that $(g'(n)\Gamma')_{n \in \N}$ is totally equidistributed in $G'/\Gamma'$ \textup{(}where $\Gamma' := G \cap \Gamma$\textup{)}, and $\gamma: \Z \to G$ is a polynomial sequence such that the sequence $(\gamma(n)\Gamma)_{n \in \N}$ is rational \textup{(}and hence, by Lemma \ref{rat-poly-lem} \textup{(i)}, is periodic\textup{)}.
\end{corollary}

\begin{proof}  We give a sketch of this argument only; we will repeat this argument in more detail when proving Theorem \ref{mainthm-prelim} below.

We induct on the dimension $m$ of $G/\Gamma$, assuming that the claim has already been proven for all nilmanifolds of lesser dimension.  By replacing $g(n)$ with $g(0)^{-1} g(n)$ if necessary (absorbing the $g(0)$ factor into the $\eps$ term) we may normalise so that $g(0)=\id_G$.
If $(g(n)\Gamma)_{n \in \Z}$ is equidistributed on $G/\Gamma$, then it is totally equidistributed by Leibman's theorem, and we are done (with $g'=g$, $G' = G$, and $\eps, \gamma$ trivial).  So we may assume that $(g(n)\Gamma)_{n \in \Z}$ is not equidistributed.  By Leibman's theorem, there exists a non-trivial horizontal character $\eta: G \to \R/\Z$ such that $\eta \circ g$ is constant, in fact by our normalisation $g(0)=\id_G$ we must have $\eta \circ g \equiv 0$, thus $g$ takes values in $\ker(\eta)$.  It is then not difficult to factorise $g = g_0 \gamma_0$, where $\gamma_0$ is a polynomial sequence with $(\gamma_0(n)\Gamma)_{n \in \Z}$ rational and periodic, and $g_0$ is a polynomial sequence taking values in the proper rational subgroup $G' \leq G$, defined to be the connected component of $\ker(\eta)$ which contains the origin.  The claim then follows by applying the induction hypothesis to the sequence $(g_0(n) \Gamma')_{n \in \Z}$ in the nilmanifold $G'/\Gamma'$, which has dimension $m-1$, and using the fact that the product of two rational group elements is again rational, as well as the trivial observation that rational group elements of $G'$ are automatically rational group elements of $G$ also.
\end{proof}

\remark 
In words, this corollary asserts that in the qualitative setting, one can decompose 
\begin{equation*}\mbox{(arbitrary polynomial sequence)}  =  \mbox{(constant)} \times \mbox{(totally equidistributed)} \times \mbox{(periodic)}.
\end{equation*}
An inspection of the proof reveals that one can in fact take the constant $\varepsilon$ to be $g(0)$.

As a corollary we obtain a Ratner-Shah type theorem for polynomial sequences in nilmanifolds, first established by Leibman \cite{leibman-single-poly}:

\begin{corollary}[Leibman's Ratner-Shah type theorem for nilmanifolds]\label{ratnil}  Let $(g(n)\Gamma)_{n \in \Z}$ be a polynomial sequence in a nilmanifold $G/\Gamma$.   Then there exists a rational subgroup $G'$ of $G$, a group element $\eps \in G$, and a rational periodic sequence $(x_n)_{n \in \Z}$ in $G/\Gamma$ with some period $q$ such that for every $r \in \Z$, the sequence $(g(qn+r)\Gamma)_{n \in \Z}$ is totally equidistributed in $\eps G' x_r$.
\end{corollary}

\remark  Shah \cite{shah} obtained a similar result for arbitrary discrete unipotent (but linear) flows on a finite volume homogeneous space; the case of continuous unipotent linear flows was treated earlier by Ratner \cite{ratner} (see \cite{dave-witte-morris} for further discussion).  Leibman's proof of Corollary \ref{ratnil} does not use these results, but instead proceeds in two stages.  Firstly, by iterating Theorem \ref{leon-green} (or more precisely a generalization of this theorem to the case when $G$ is not necessarily connected), a version of Corollary \ref{ratnil} for linear sequences is obtained.  Secondly, by utilising a lifting trick of Furstenberg \cite[p. 31]{furstenberg-book}, the polynomial case is deduced from the linear case.  As we shall discuss shortly, these arguments do not work well in the quantitative case, and one must instead grapple with polynomial sequences directly.

\textsc{Quantitative equidistribution results.} This paper stems from an attempt to establish quantitative versions of the above theorems for \emph{finite} orbits.  Unfortunately, the need for quantitative bounds on all aspects of these results forces us to introduce a substantial amount of new notation.

\begin{definition}[Asymptotic notation] We use $Y = O(X)$ or $Y \ll X$ to denote the estimate $|Y| \leq CX$ some absolute constant $C$.  When we need to indicate dependence of $C$ on various parameters, we shall indicate this by subscripts, thus for instance $O_{d,m}(X)$ denotes a quantity bounded in magnitude by $C_{d,m}X$ for some $C_{d,m}$ depending only on the quantities $d,m$.
\end{definition}

\begin{definition}[Circle norm]  If $x \in \R/\Z$, we use $\|x\|_{\R/\Z} := \dist(x,\Z)$ to denote the distance of $x$ to the origin (thus $\|a\md{\Z}\|_{\R/\Z} = |a|$ whenever $-1/2 < a \leq 1/2$). If $x \in \R$, we write $\|x\|_{\R/\Z}$ for $\|x\md{\Z}\|_{\R/\Z}$.
\end{definition}

Our first main result is the following quantitative version of Theorem \ref{leibman-thm}. Note that some of the terminology in this theorem will not be formally introduced until the next section, but this should not prevent the reader from gaining a rough appreciation of the statement.

\begin{theorem}[Quantitative Leibman theorem]\label{leib-quant}  Let $m, d \geq 0$, $0 < \delta < 1/2$, and $N \geq 1$.
Let $G/\Gamma$ be an $m$-dimensional nilmanifold together with a filtration $G_{\bullet}$ of degree $d$ and a $\frac{1}{\delta}$-rational Mal'cev basis $\X$ adapted to this filtration. Suppose that $g \in \poly(\Z,G_{\bullet})$. 
Then at least one of the following statements is true:
\begin{itemize}
\item[(i)] $(g(n)\Gamma)_{n \in [N]}$ is $\delta$-equidistributed in $G/\Gamma$.
\item[(ii)] There exists a non-trivial horizontal character $\eta: G \to \R/\Z$ with $|\eta| \ll \delta^{-O_{m,d}(1)}$ such that $\| \eta \circ g(n) - \eta \circ g(n-1) \|_{\R/\Z} \ll \delta^{-O_{m,d}(1)}/N$ for all $n \in [N]$.
\end{itemize}
\end{theorem}

\remarks   The notions of a ``$\frac{1}{\delta}$-rational Mal'cev basis adapted to $G_{\bullet}$'', of the modulus $|\eta|$ of a horizontal character and of the metric which is implicit in the notion of $\delta$-equidistribution are technical and will be defined precisely in Definition \ref{ratnil-def}, Definition \ref{horiz-char-modulus}, and Definition \ref{metric-def} respectively.

Theorem \ref{leib-quant} asserts that the sequence $(g(n)\Gamma)_{n \in [N]}$ is either  $\delta$-equidistributed up to time $N$, or else it is very far from being equidistributed up to time $\delta^{O_{m,d}(1)}N$, being concentrated very close to a union of $\delta^{-O_{m,d}(1)}$ subtori.  One should view $N$ as being very large compared to $1/\delta$, otherwise the content of the proposition is trivial.  It is not hard to deduce Theorem \ref{leibman-thm} from Theorem \ref{leib-quant}; we leave this to the reader as an exercise.

For technical reasons it will be convenient later to strengthen the statement (ii) slightly, so as to also control higher ``derivatives'' $\partial^j (\eta \circ g)$; see the next section for more information.

Whereas in the qualitative setting one always works in the limit $N \to \infty$, in the quantitative setting one works with a fixed (but large) $N$.  As $N$ increases, there can be transitions in the behaviour of the finite sequence $(g(n) \Gamma)_{n \in [N]}$, in which the equidistribution (or lack thereof) changes significantly (cf. the ``coalescence of progressions'' phenomenon \cite[Chapter 12]{tao-vu}); these transitions are a new feature of the quantitative setting, which are not readily visible in the qualitative one.  We illustrate this with a simple example:

\example
Consider the (additive) example $G = \R$, $\Gamma = \Z$ and $g(n) = (\frac{1}{2}+\sigma)n$, where $0 < \sigma \leq \frac{\delta}{100}$ is a parameter.  In this case we have $m=d=1$.  If $N$ is much larger than $1/\sigma$, we see that $(g(n)\md{\Z})_{n \in [N]}$ is $\delta$-equidistributed.  On the other hand, if $N$ is much smaller than $1/\sigma$, we see that $(g(n)\md{\Z})_{n \in [N]}$ fails to be $\delta$-equidistributed, indeed it is highly concentrated around $0$ and $1/2$ in this case.  However, if we let $\eta: G \to \R/\Z$ be the non-trivial horizontal character $\eta(x) := 2x \md{\Z}$ we see that $\eta(g(n))$ is slowly varying in the sense of (ii).  The transitional regime when $N$ is comparable to $1/\sigma$ is interesting; there is enough irregularity to prevent $\delta$-equidistribution on the sequence $(g(n)\md{\Z})_{n \in [N]}$, but in order to obtain near-constancy of $\eta(g(n))$ one in fact has to pass to shorter sequences such as $(g(n)\md{\Z})_{n \in [c\delta N]}$.  The need to work on a variety of different scales like this is very much a feature of additive combinatorics, particularly those parts of it that have the flavour of ``quantitative ergodic theory''. The work of Bourgain \cite{bourgain} on Roth's theorem is another example.
\endexample

Of course, by specialising to linear sequences, Theorem \ref{leib-quant} also implies a quantitative version of Leon Green's theorem. The proof of Theorem \ref{leib-quant} could be simplified somewhat in this case. Such a theorem is not especially useful, however. The following example may help to illustrate why, in the quantitative setting, the consideration of linear sequences leads naturally to the ``polynomial'' world.

\example (The skew torus) Let us consider the Heisenberg example \eqref{heisen} once more, taking now
\[ a := \left(\begin{smallmatrix} 1 & 2\alpha & \alpha \\ 0 & 1 & 1 \\ 0 & 0 & 1\end{smallmatrix}\right)\]
where $\alpha := N^{-3/2}$. Set 
\[ g(n) := a^n = \left(\begin{smallmatrix} 1 & 2n\alpha & n^2\alpha \\ 0 & 1 & n \\ 0 & 0 & 1\end{smallmatrix}\right).\]
Translating to the fundamental domain, we obtain 
\[ g(n) \Gamma = \left[\left(\begin{smallmatrix} 1 & \{2n\alpha\} & \{-n^2\alpha\} \\ 0 & 1 & 0 \\ 0 & 0 & 1\end{smallmatrix}\right)\right].\]
(Here, and for the rest of the paper, we define $\{x\} := x - \lfloor x\rfloor$, where $\lfloor x \rfloor$ is the greatest integer less than or equal to $x$.)
The orbit $(g(n) \Gamma)_{n \in [N]}$ is certainly not close to equidistributed in $G/\Gamma$, and indeed the projected orbit $(\pi(g(n)\Gamma))_{n \in [N]}$ stays very close to the trivial subtorus $T \subseteq \R^2/\Z^2$ which consists simply of the point $\{(0,0)\}$. 

Now $\pi^{-1}(T)$ is of course isomorphic to a one-dimensional torus $\R/\Z$. However the orbit $(g(n)\Gamma)_{n \in [N]}$ does not approximate a linear orbit on this torus; rather, it has \emph{quadratic} behaviour. Thus $(g(n)\Gamma)_{n \in [N]}$ is very close to $(g'(n)\Gamma')_{n \in [N]}$ on $G'/\Gamma' \cong \R/\Z$, where 

\[ G' := \left( \begin{smallmatrix} 1 & 0 & \R \\ 0 & 1 & 0 \\ 0 & 0 & 1\end{smallmatrix}\right),\]
\[ \Gamma' := \left( \begin{smallmatrix} 1 & 0 & \Z \\ 0 & 1 & 0 \\ 0 & 0 & 1\end{smallmatrix}\right)\] and
\begin{equation} g'(n) := \left(\begin{smallmatrix} 1 & 0 & -n^2\alpha \\ 0 & 1 & 0 \\ 0 & 0 & 1\end{smallmatrix}\right).\end{equation}
Thus, in order to approximate the linear sequence $(g(n)\Gamma)_{n \in [N]}$ by a lower-dimensional sequence, the latter sequence needs to be polynomial.  Note however that if one had the luxury of passing from $[N]$ to a much shorter progression, e.g. $[N^{1/100}]$, then the lower-dimensional sequence would remain linear.  In the limit $N \to \infty$, $N$ and $N^{1/100}$ both go to infinity, which may help explain why in the qualitative setting one can avoid polynomial sequences entirely and work purely in the category of linear sequences.  Unfortunately, for the quantitative applications we have in mind (in particular, the number-theoretic application in \cite{ukmobius}) we cannot afford to reduce the scale $N$ in such a drastic manner\footnote{This is ultimately because it is known how to obtain non-trivial control on averages of number-theoretic functions such as the M\"obius function $\mu$ on intervals such as $[N, N + N\log^{-A} N]$, but not in intervals such as $[N, N+N^{1/100}]$, even if one assumes strong hypotheses such as GRH.}.
\endexample

In much the same way that Theorem \ref{leibman-thm} could be iterated in order to establish Corollary \ref{ratnil-factor},
we can iterate Theorem \ref{leib-quant} to obtain a quantitative factorization theorem. To state it we need quantitative versions of the ``rationality'' concepts of Definition \ref{rat-def} and also the new notion of \emph{smooth} sequences, which must be introduced in place of constant sequences in the finitary setting.

\begin{definition}[Rational sequences, quantitative definitions]\label{rat-def-quant}
Let $G/\Gamma$ be a nilmanifold and let $Q > 0$ be a parameter. We say that $\gamma \in G$ is \emph{$Q$-rational} if $\gamma^r \in \Gamma$ for some integer $r$, $0 < r \leq Q$. A \emph{$Q$-rational point} is any point in $G/\Gamma$ of the form $\gamma\Gamma$ for some $Q$-rational group element $\gamma$. A sequence $(\gamma(n))_{n \in \Z}$ is \emph{$Q$-rational} if every element $\gamma(n)\Gamma$ in the sequence is a $Q$-rational point. 
\end{definition}

\begin{definition}[Smooth sequences]\label{smooth-seq-def}  Let $G/\Gamma$ be a nilmanifold with a Mal'cev basis $\X$.  Let $(\varepsilon(n))_{n \in \Z}$ be a sequence in $G$, and let $M, N \geq 1$.  We say that $(\varepsilon(n))_{n \in \Z}$ is \emph{$(M,N)$-smooth} if we have $d(\varepsilon(n),\id_G) \leq M$ and $d(\varepsilon(n),\varepsilon(n-1)) \leq M/N$ for all $n \in [N]$, where the metric $d = d_\X$ on $G$ will be defined in Definition \ref{metric-def}.
\end{definition}

Note that the notion of a $(M,N)$-smooth sequence collapses to that of a constant sequence in the limit $N \to \infty$ (holding $M$ fixed).

\begin{theorem}[Factorization theorem]\label{mainthm-prelim}  Let $m,d \geq 0$, and let $M_0, N \geq 1$ and $A > 0$ be real numbers.
Suppose that $G/\Gamma$ is an $m$-dimensional nilmanifold together with a filtration $G_{\bullet}$ of degree $d$. Suppose that $\mathcal{X}$ is an $M_0$-rational Mal'cev basis $\mathcal X$ adapted to $G_{\bullet}$ and that $g \in \poly(\Z,G_{\bullet})$. Then there is an integer $M$ with $M_0 \leq M \ll M_0^{O_{A,m,d}(1)}$, a rational subgroup $G' \subseteq G$, a Mal'cev basis $\mathcal{X}'$ for $G'/\Gamma'$ in which each element is an $M$-rational combination of the elements of $\mathcal{X}$, and a decomposition $g = \varepsilon g' \gamma$ into polynomial sequences $\varepsilon, g', \gamma \in \poly(\Z,G_{\bullet})$ with the following properties:
\begin{enumerate}
\item $\varepsilon : \Z \rightarrow G$ is  $(M,N)$-smooth;
\item $g' : \Z \rightarrow G'$ takes values in $G'$, and the finite sequence $(g'(n)\Gamma')_{n \in [N]}$ is totally $1/M^A$-equidistributed in $G'/\Gamma'$, using the metric $d_{\mathcal{X}'}$ on $G'/\Gamma'$;
\item $\gamma: \Z \rightarrow G$ is $M$-rational, and $(\gamma(n)\Gamma)_{n \in \Z}$ is periodic with period at most $M$.
\end{enumerate}
\end{theorem}

\remark 
In words, this corollary asserts that in the quantitative setting, one can decompose 
\begin{equation*} \mbox{(arbitrary polynomial sequence)}  =  \mbox{(smooth)} \times \mbox{(totally equidistributed)} \times \mbox{(periodic)}.
\end{equation*}
The notion of a subgroup $G'$ being $M$-rational relative to a Mal'cev basis $\X$ will be defined in Definition \ref{rational-subgroup-def}.  This result has some faint resemblance to the Szemer\'edi regularity lemma \cite{szemeredi}, although with the key difference that our bounds here are all polynomial in nature.

The derivation of Theorem \ref{mainthm-prelim} from Theorem \ref{leib-quant} will be performed in \S \ref{sec12}-\ref{factor-sec}.

We will use Theorem \ref{mainthm-prelim} in \cite{ukmobius} in order to establish the \emph{M\"obius and Nilsequences conjecture} $\mbox{MN}(s)$ from \cite{green-tao-linearprimes} for arbitrary step $s$.  For this application, it is important that all bounds here are only polynomial in $M$, and that the equidistribution is established on progressions of length linear in $N$ (as opposed to $N^c$ for some small $c > 0$).  

Just as Corollary \ref{ratnil-factor} implies a Ratner-type theorem, namely Corollary \ref{ratnil}, it is not hard to deduce the following result from Theorem \ref{mainthm-prelim}.

\begin{corollary}[Ratner-type theorem for polynomial nilsequences]\label{ratner-result} Let $m,d \geq 0$, $0 < \delta < 1/2$, and $N \geq 1$.
Suppose that $G/\Gamma$ is an $m$-dimensional nilmanifold, that $G_{\bullet}$ is a filtration of degree $d$ on $G$, and that $\mathcal{X}$ is a $1/\delta$-rational Mal'cev basis adapted to $G_{\bullet}$. Suppose that $g \in \poly(\Z,G_{\bullet})$. Then we may decompose $[N]$ as a union $P_1 \cup \dots \cup P_k$ of arithmetic progressions with length $\gg \delta^{O_{m,d}(1)}N$ and the same common difference $q$, $1 \leq q \ll \delta^{-O_{m,d}(1)}$, such that each orbit $(g(n)\Gamma)_{n \in P_i}$ is within $\delta$ \textup{(}using the metric $d_{\X}$\textup{)} of being equidistributed on $x_iG'y_i\Gamma/\Gamma \subseteq G/\Gamma$, where $x_i \in G$, $y_i \in G$ is $\delta^{-O_{m,d}(1)}$-rational, and $G'$ is a closed subgroup of $G$ which is $\delta^{-O_{m,d}(1)}$-rational relative to $\X$ \textup{(}this notion will be defined in the next section\textup{)}.
\end{corollary}

\remark  The reader may wish to compare this with \cite{e-m-v}, another recent result on quantitative variants of Ratner's theorem.

Let us conclude this introduction by remarking that our main theorem actually applies to \emph{multiparameter} polynomial mappings $g : \Z^t \rightarrow G$. In the infinitary setting such a generalization was obtained by Leibman \cite{leibman-multi-poly}, and his result has subsequently been applied in such papers as \cite{bergelson-host-kra} and \cite{leibman-orb-diag}. We have taken the trouble to derive multiparameter extensions of our main results with analogous finitary applications in mind; see Theorems \ref{main-theorem-multi} and Theorem \ref{mainthm-multi}.

\section{Precise statements of results}\label{sec2}

In this section we define various ``quantitative'' concepts (such as $Q$-rational Mal'cev bases, subgroups which are $Q$-rational relative to such a basis and the metrics $d_\X$ and $d_{G/\Gamma}$) which were needed to properly state the main results from the introduction section.  We also give a more precise version of Theorem \ref{leib-quant}, which we will then spend the next several sections proving.

\textsc{Mal'cev bases and metrics on $G/\Gamma$.} The notion of \emph{Mal'cev coordinates} play a vital r\^ole in the quantitative theory of nilmanifolds. They allow us to put a metric on $G/\Gamma$, which in turn allows us to define the notion of equidistribution; they also quantify the ``rationality'' of various objects associated to the nilmanifold. Mal'cev coordinates were introduced in \cite{malcev}, which contains a nice discussion; they are covered quite extensively in the book \cite{corwin-greenleaf}, particularly Chapters 1 and 5.  We will also need several more quantitative statements about Mal'cev coordinates, which we have placed in Appendix \ref{malcev-app}. We recommend that the reader dip into that appendix as and when required.

We will make use of the Lie algebra $\g$ of $G$ together with the exponential map $\exp : \g \rightarrow G$. When $G$ is a connected, simply-connected nilpotent Lie group the exponential map is a diffeomorphism; see \cite[Theorem 1.2.1]{corwin-greenleaf}. In particular, we have a logarithm map $\log: G \to \g$.  One does not really need to have an understanding of the exponential and logarithm maps beyond some of their formal properties, which we will list as we need them, in order to understand this paper.

\begin{definition}[Mal'cev bases]\label{malcev}
Let $G/\Gamma$ be a $m$-dimensional nilmanifold and let $G_{\bullet}$ be a filtration. A basis $\X = \{X_1,\dots,X_{m}\}$ for the Lie algebra $\g$ over $\R$ is called a \emph{Mal'cev basis} for $G/\Gamma$ adapted to $G_{\bullet}$ if the following four conditions are satisfied: 
\begin{enumerate}
\item For each $j = 0,\dots,m-1$ the subspace $\h_j := \Span(X_{j+1},\dots,X_m)$ is a Lie algebra ideal in $\g$, and hence $H_j := \exp \h_j$ is a normal Lie subgroup of $G$.
\item For every $0 \leq i \leq s$ we have $G_i = H_{m-m_i}$ (recall that $m_i = \dim G_i$);
\item Each $g \in G$ can be written uniquely as $\exp(t_1 X_1)\exp(t_2 X_2)\dots \exp(t_m X_m)$, for $t_i \in \R$;
\item $\Gamma$ consists precisely of those elements which, when written in the above form, have all $t_i \in \Z$.
\end{enumerate}
\end{definition}

\emph{Remarks.} Our main results only make sense if the nilmanifold $G/\Gamma$ is already equipped with a Mal'cev basis $\X$, since they involve quantitative dependencies that can only be described using such a basis. However it is a well-known result of Mal'cev \cite{malcev} that \emph{any} nilmanifold $G/\Gamma$ can be equipped with a Mal'cev basis adapted to the lower central series filtration. Indeed the very existence of a discrete and cocompact subgroup $\Gamma$ guarantees that the lower central series is rational by \cite[Theorem 5.1.8 (a)]{corwin-greenleaf} and \cite[Corollary 5.2.2]{corwin-greenleaf}. One may then apply \cite[Proposition 5.3.2]{corwin-greenleaf} to deduce the existence of a Mal'cev basis adapted to the lower central series. More generally there is a Mal'cev basis adapted to any filtration $G_{\bullet}$ which consists of rational subgroups (cf. Definition \ref{rat-sub-def}).

We refer to the $t_i$ as the \emph{Mal'cev coordinates} of $g$, and we define the \emph{Mal'cev coordinate map} $\psi = \psi_\X: G \to \R^m$ to be the map
\begin{equation}\label{psidef}
 \psi(g) := (t_1,\ldots,t_m),
 \end{equation}
 thus for instance $\Gamma = \psi^{-1}(\Z^m)$. If $\mathcal{X}'$ is another Mal'cev basis (relative to some filtration) then we write $\psi' = \psi_{X'}$. Only very occasionally will we need to use the notation $\psi_{\mathcal{Y}}$ to indicate the coordinate map relative to some further basis $\mathcal{Y}$.

\remarks  In the literature, Mal'cev coordinates are invariably discussed in the context of the lower central series filtration and are referred to as \emph{coordinates of the second kind}. Coordinates of the first kind or \emph{exponential coordinates} are derived by writing $\log g \in \g$ as a linear combination $\log g = s_1 X_1 + \ldots + s_m X_m$ of elements of the basis $\X$, and we write $\psi_{\exp}(g) = \psi_{\X,\exp}(g) := (s_1,\ldots,s_m)$ for the coordinates of $g$ obtained in this fashion.  However, we shall mostly work using coordinates of the second kind.

We can use a Mal'cev basis $\X$ to put a (slightly artificial) metric structure on $G$ and on $G/\Gamma$. 

\begin{definition}[Metrics on $G$ and $G/\Gamma$]\label{metric-def}  Let $G/\Gamma$ be a nilmanifold with Mal'cev basis $\X$. We define $d = d_\X : G \times G \rightarrow \R_{\geq 0}$ to be the largest metric such that $d(x,y) \leq |\psi(xy^{-1})|$ for all $x,y \in G$, where $|\cdot |$ denotes the $\ell^{\infty}$-norm on $\R^m$. More explicitly, we have
\[ d(x,y) = \inf\left\{ \sum_{i=0}^{n-1} \min(|\psi(x_{i-1}x_i^{-1})|, |\psi(x_ix_{i-1}^{-1})|): x_0,\dots,x_n \in G; x_0 = x; x_n = y  \right\}.\]
This descends to a metric on $G/\Gamma$ by setting
\[ d(x\Gamma,y\Gamma) := \inf \{d(x',y') : x',y' \in G; x' \equiv x \md{\Gamma}; y' \equiv y \md{\Gamma}\}.\]
It turns out that this is indeed\footnote{We note that this metric structure is a little more specific than in some of our previous papers, notably that in \cite[\S 8]{green-tao-linearprimes}.  This will not cause any difficulty, as the metrics in that paper are equivalent to the one given here, up to constants depending on $G, \Gamma$ and  $\X$.  Indeed, at small scales $d$ agrees with the distance function given by the unique right-invariant Riemannian metric on $G$ whose value at the origin is equal to that of the Euclidean metric at the origin of $\R^m$, pulled back by $\psi$; see also Lemma \ref{dx-bounds}.} a metric on $G/\Gamma$; this essentially follows from the discreteness of $\Gamma$ in $G$, and we will prove it in Lemma \ref{quotient-metric}.
Since $d$ is right-invariant, we also have
\[ d(x\Gamma, y\Gamma) = \inf_{\gamma \in \Gamma}  d(x,y\gamma).\]
\end{definition}

When the letter $d$ is used for a metric, it will always denote the metric $d_{\mathcal{X}}$ relative to some basis $\X$ that is already under discussion. The symbol $d'$ will be used for the metric defined using some other basis $\X'$. On the very rare occasions (for example in the proof of Lemma \ref{rat-bounds}) where the metric relative to some further basis is under consideration we will indicate this explicitly using subscripts.

\textsc{Quantitative rationality.}
Now we define the concept of rational nilmanifolds and subgroups.

\begin{definition}[Height]  The \emph{height} of a real number $x$ is defined as $\max(|a|,|b|)$ if $x=a/b$ is rational in reduced form, and $\infty$ if $x$ is irrational.
\end{definition}

\begin{definition}[Rationality of a basis]\label{ratnil-def}
Let $G/\Gamma$ be a nilmanifold and $Q > 0$.  We say that a Mal'cev basis $\X$ for $G/\Gamma$ is \emph{$Q$-rational} if all of the structure constants $c_{ijk}$ in the relations
\[ [X_i, X_j] = \sum_k c_{ijk} X_k\] are rational with height at most $Q$.
\end{definition}
 
\begin{definition}[Rational subgroups]\label{rational-subgroup-def}
Suppose that a nilmanifold $G/\Gamma$ is given together with a Mal'cev basis $\X = \{X_1,\dots,X_m\}$, and that $Q > 0$. Suppose that $G' \subseteq G$ is a closed connected subgroup. We say that $G'$ is \emph{$Q$-rational} relative to $\X$ if the Lie algebra $\g'$ has a basis $\X' = \{X'_1,\dots,X'_{m'}\}$ consisting of linear combinations $\sum_{i=1}^m a_iX_i$, where $a_i$ are rational numbers with height at most $Q$ for all $i$.
\end{definition}

\begin{definition}[Modulus of a horizontal character]\label{horiz-char-modulus}
Suppose that $G/\Gamma$ is a nilmanifold with a Mal'cev basis $\mathcal{X}$. Suppose that $\eta : G \rightarrow \R/\Z$ is a horizontal character, that is to say a homomorphism from $G$ to $\R/\Z$ which annihilates $\Gamma$. Then, when written in coordinates relative to $\mathcal{X}$, properties (iii) and (iv) of Proposition \ref{malcev} imply that $\eta(g) = k \cdot \psi(g)$ for some unique $k \in \Z^m$. We write $|\eta| := |k|$.
\end{definition}

\textsc{Smooth polynomial sequences.} 
For technical reasons it will be convenient to quantify the smoothness of sequences, such as the sequence $\varepsilon(n)$ appearing in Theorem \ref{mainthm-prelim}, in a slightly different manner from that used so far.

\begin{definition}[Smoothness norms]
Suppose that $g : \Z \rightarrow \R/\Z$ is a polynomial sequence of degree $d$. Then $g$ may be written uniquely as
\[ g(n) = \alpha_0 + \alpha_1 \binom{n}{1} + \dots + \alpha_d \binom{n}{d}\] where $\alpha_i$ is in fact equal to $\partial^i g(0)$.
For any $N > 0$ we define the \emph{smoothness norm} 
\[ \Vert g \Vert_{C^{\infty}[N]} := \sup_{1 \leq j \leq d} N^{j} \Vert \alpha_{j} \Vert_{\R/\Z}.\]
\end{definition}

The smoothness norm $\Vert \cdot \Vert_{C^{\infty}[N]}$ is designed to capture the notion of a polynomial sequence which is slowly-varying. Indeed, the following lemma is easily verified:

\begin{lemma}[Smooth polynomials vary slowly]\label{smooth-slow-1}
Let $g : \Z \rightarrow \R/\Z$ be a polynomial sequence of degree $d$, and let $N > 0$. 
 Then for any $n \in [N]$ we have
\[ \|g(n) - g(n-1)\|_{\R/\Z} \ll_{d} \frac{1}{N} \Vert g \Vert_{C^{\infty}[N]}.\]
\end{lemma}

In view of this lemma, we see that Theorem \ref{leib-quant} will be an immediate consequence of the following more precise statement. This is in fact the main technical result in our paper and we will use it to derive all our other main results.

\begin{theorem}[Quantitative Leibman theorem]\label{main-theorem} Let $m,d \geq 0$, $0 < \delta < 1/2$ and $N \geq 1$.  Suppose that $G/\Gamma$ is an $m$-dimensional nilmanifold together with a filtration $G_{\bullet}$ and that $\mathcal{X}$ is a $\frac{1}{\delta}$-rational Mal'cev basis adapted to $G_{\bullet}$. Suppose that $g \in \poly(\Z,G_{\bullet})$. If $(g(n)\Gamma)_{n \in [N]}$ is not $\delta$-equidistributed, then there is a horizontal character $\eta$ with $0 < | \eta | \ll \delta^{-O_{m,d}(1)}$ such that
\[ \Vert \eta \circ g \Vert_{C^{\infty}[N]} \ll \delta^{-O_{m,d}(1)}.\]
\end{theorem}

\textsc{Notes on reading the paper.} As with so many papers, some parts of this work are merely technical and other parts represent deeper ideas of greater interest. There are quite a number of computations in this paper in which one has to show, say, that a certain integer is bounded polynomially by another, or that a certain basis is $O(\delta^{-O(1)})$-rational. All such computations are of the technical variety and should certainly be ignored on a first reading. They are all in a sense ``clear''; their proofs proceed by algebra of a type which could hardly be expected to introduce non-polynomial dependencies. It is possible that this could even be encoded in some relatively soft ``proof-theoretic'' language, but we have chosen not to follow such a path.

We begin with several sections containing motivating examples. In \S \ref{sec3} we will discuss linear flows on tori $\R^m/\Z^m$, in \S \ref{sec4} we shall discuss polynomial flows on $\R/\Z$, and in \S \ref{heisenberg-sec} we will look at linear flows on the $2$-step Heisenberg nilmanifold \eqref{heisen}. Some lemmas from these sections will be required in the sequel.

We then begin the study of the general case. In \S \ref{poly-sequences-sec} we study the algebraic properties of polynomial sequences on nilpotent groups following Lazard and Leibman. There is a rich general theory here which is not evident from the study of the abelian and Heisenberg examples.

We then turn to the full proof of Theorem \ref{main-theorem}, the quantitative Leibman theorem. This is the technical heart of the paper and is given in the (rather long) \S \ref{sec6}.

In \S \ref{multi-deduction} use a straightforward iteration argument to bootstrap Theorem \ref{main-theorem} to a multiparameter version of itself, namely Theorem \ref{main-theorem-multi}.
In \S \ref{sec13} we then establish a preliminary multiparameter factorization theorem, Proposition \ref{prop2.17}, which is a fairly short consequence of Theorem \ref{main-theorem-multi}. In \S \ref{factor-sec} we then iterate this proposition, obtaining a multiparameter theorem (Theorem \ref{mainthm-multi}) which then easily implies Theorem \ref{mainthm-prelim} (and hence Corollary \ref{ratner-result}) as special cases.  

The appendix contains basic results on bases and nilmanifolds.

There is unfortunately a large amount of notation in this paper.  In Figure \ref{fig1} the key objects in the argument are briefly described.

\begin{figure}
\begin{tabular}{|l|l|l|}
\hline
$G$ & nilpotent group  & Definition \ref{nil-def} \\
$G_{\bullet} = (G_i)_{i=0}^{\infty}$ & filtration on $G$ & Definition \ref{nil-def} \\
$G/\Gamma$ & nilmanifold & Definition \ref{nil-def} \\
$(G/\Gamma)_{\ab} = G/[G,G]\Gamma$ & horizontal torus &  Definition \ref{horiz-def} \\
$G_d/(\Gamma \cap G_d) \cong \R^{m_d}/\Z^{m_d}$ & vertical torus & Definition \ref{vert-def} \\
\hline
$d \geq 0$ & degree of the filtration $G_{\bullet}$ & Definition \ref{nil-def}\\
$s \geq 0$ & step of $G$ & Definition \ref{nil-def} \\
$m \geq 0$ & dimension of $G$ & Definition \ref{nil-def} \\
$m_i$ & dimension of $G_i$ & Definition \ref{nil-def} \\
$m_{\ab}$ & dimension of horizontal torus & Definition \ref{horiz-def}\\
$m_{\lin}$ & $m - m_2$ & \S \ref{sec6} \\
$m_* = m_{\ab} - m_{\lin}$ & nonlinearity degree of $G_{\bullet}$ & \S \ref{sec6}\\
$\eta: G \to \R/\Z$ & horizontal character & Definition \ref{horiz-def} \\
$\xi: G_d \to \R/\Z $ & vertical character & Definition \ref{vert-char-def}\\
\hline
$\X = (X_i)_{i=1}^m$, $\X' = (X'_i)_{i=1}^m$ & Mal'cev bases & Definition \ref{malcev} \\
$\psi,\psi'$ & coordinate maps relative to $\X,\X'$ & \eqref{psidef} \\
$d,d'$ & metrics defined using $\X,\X'$ & Definition \ref{metric-def} \\
$Q \geq 1$ & rationality bound for $\X$ (usually $Q=1/\delta$) & Definition \ref{ratnil-def} \\
$\pi: G \to (G/\Gamma)_{\ab}$ & projection onto the horizontal torus & Definition \ref{horiz-def} \\
\hline
$F: G/\Gamma \to \C$ & Lipschitz function & Definition \ref{metric-def} \\
$0 < \delta < 1/2$ & level of equidistribution & Definition \ref{almost-equidistribution} \\
$N \geq 1$ & length of sequence & Definition \ref{almost-equidistribution} \\
\hline
$g: \Z \to G$ & a polynomial sequence & Definition \ref{poly-def} \\
$\poly(\Z,G_{\bullet})$ & polynomial sequences with coeffs in $G_{\bullet}$ & Definition \ref{poly-def}\\
\hline
$t \geq 1$ & number of parameters & \S \ref{multi-deduction}\\
\hline
\end{tabular}
\caption{A list of key objects in the paper, together with brief descriptions of these objects, and the location where they are first defined or introduced.}
\label{fig1}
\end{figure}

\section{A quantitative Kronecker theorem}\label{sec3}

In this section we prove Theorem \ref{main-theorem} for linear sequences on the torus $\R^m/\Z^m$, that is to say we establish a quantitative Kronecker theorem. The methods and the result are very standard.

\begin{proposition}[Quantitative Kronecker Theorem]\label{quant-kron}
Let $m \geq 1$, let $0 < \delta < 1/2$, and let $\alpha \in \R^m$.  If the sequence $(\alpha n \md{ \Z^m})_{n \in [N]}$ is not $\delta$-equidistributed in the additive torus $\R^m/\Z^m$, then there exists $k \in \Z^m$ with $0 < |k| \ll \delta^{-O_m(1)}$ such that $\Vert k \cdot \alpha \Vert_{\R/\Z} \ll \delta^{-O_m(1)}/N$.
\end{proposition}

\remark  We leave it to the reader to check that this really \emph{is} the specialization of Theorem \ref{main-theorem} to the case of linear orbits on the torus $\R^m/\Z^m$. This may be found helpful in understanding some of our notation. Note in particular that in this case the horizontal torus is simply $\R^m/\Z^m$, and we may take $\pi$ to be the identity map.

\proof By Definition \ref{almost-equidistribution}, there is a Lipschitz function $F : \R^m/\Z^m \rightarrow \R$ such that
\begin{equation}\label{equi-failure} |\E_{n \in [N]} F(\alpha n \md{\Z^m}) - \int_{\R^m/\Z^m} F\, d\theta| > \delta \Vert F \Vert_{\Lip}.\end{equation}
 At the expense of replacing $\delta$ by $\delta/2$ we may translate $F$, add a constant to it and rescale in such a way that $\int F = 0$ and $\Vert F \Vert_{\Lip} = 1$. By approximating $F$ by smooth functions we may assume that $F$ is smooth (we do this to avoid any technical issues regarding convergence of Fourier series). We now use a standard man{\oe}uvre to approximate $F$ by a function which has finite support in frequency space (cf. \cite[Lemma A.9]{green-tao-u3mobius}).

Consider the Fej\'er kernel $K: \R^m/\Z^m \to \R^+$ defined by
\[ K(\theta) := \frac{1}{\mes(Q)} 1_Q \ast \frac{1}{\mes(Q)} 1_Q(\theta) \]
where $Q := [-\frac{\delta}{16m},\frac{\delta}{16m}]^m \subset \R^m/\Z^m$ is a small cube, and $\ast$ denotes the usual convolution operation on the torus $\R^m/\Z^m$.  It is immediate that $K$ is a non-negative function supported in $Q$ with
\begin{equation}\label{fejer-1} \int_{\R^m/\Z^m} K = 1.\end{equation}
A simple calculation also establishes the estimate
\begin{equation}\label{fejer-3} \sum_{k \in \Z^m: |k| \geq M} |\widehat{K}(k)| \ll_m \delta^{-2m}M^{-1}
\end{equation}
for all $M > 1$, where the Fourier coefficient is defined by
\[ \widehat{K}(k) := \int_{\R^m/\Z^m} K(\theta) e(-\theta \cdot k)\, d\theta\]
and $e(x) := e^{2\pi i x}$ is the standard character on $\R/\Z$.  We also have the crude bound
\begin{equation}\label{fact-4} |\widehat{F}(k)| \leq \|F\|_{\infty} \leq \|F\|_{\Lip} \leq 1\end{equation} 
for all $k \in \Z^m$.

Set $F_1 := F \ast K$.  Since $\|F\|_{\Lip} = 1$, and $K$ is supported in $Q$ and satisfies \eqref{fejer-1}, a standard computation shows that
\[ \Vert F - F_1\Vert_{\infty} \leq \delta/8.\]
Choose $M := C_m\delta^{-2m-1}$ for some suitably large $C_m$, and set
\[ F_2(\theta) := \sum_{k \in \Z^m : 0 < |k| \leq M} \widehat{F}_1(k) e(k \cdot \theta).\]
Noting that $\widehat{F}_1(0) = 0$, facts \eqref{fejer-3}, \eqref{fact-4} and the Fourier inversion formula imply that 
\[ \Vert F_1 - F_2 \Vert_{\infty} \leq \delta/8.\]
It follows that $\Vert F - F_2 \Vert_{\infty} \leq \delta/4$, which means in view of the failure of \eqref{equi-failure} that
\[ |\E_{n \in [N]} F_2(n\alpha\Z^m)| \geq \delta/4.\]
Applying \eqref{fact-4} once more we see that there is some $k$, $0 < |k| \leq M$, such that 
\[ |\E_{n \in [N]} e(n k \cdot \alpha)| \gg_m \delta M^m \gg \delta^{O_m(1)}.\]
The result now follows immediately from the standard estimate
\[ |\E_{n \in [N]} e(nt)| \ll \min\left(1, \frac{1}{N\Vert t \Vert_{\R/\Z}}\right),\]
which follows from summing the geometric progression.\endproof

Let us now record a corollary of the $m = 1$ version of this result which will be used several times in the sequel. This gives stronger information in the case that $(n\alpha\md{\Z})_{n \in [N]}$ is \emph{very} far from being equidistributed. 

\begin{lemma}[Strongly recurrent linear functions are highly non-diophantine]\label{strong-linear}
Let $\alpha \in \R$, $0 < \delta < 1/2$, and $0 < \epsilon \leq \delta/2$, and let $I \subseteq \R/\Z$ be an interval of length $\epsilon$ such that $\alpha n \in I$ for at least $\delta N$ values of $n \in [N]$.  Then there is some $k \in \Z$ with $0 < |k| \ll \delta^{-O(1)}$ such that $\Vert k \alpha \Vert_{\R/\Z} \ll \epsilon \delta^{-O(1)}/N$.
\end{lemma}

\proof Taking $F$ to be a Lipschitz approximation to the interval $I$, we see immediately that our assumption precludes $(\alpha n \md{\Z})_{n \in [N]}$ from being $\delta^{10}$-equidistributed. It follows from the case $m = 1$ of Proposition \ref{quant-kron} that there is some $k \in \Z$, $|k| \ll \delta^{-C}$, such that $\Vert k\alpha\Vert_{\R/\Z} \ll \delta^{-C}/N$, where $C=O(1)$. Write $\beta := \Vert k \alpha \Vert_{\R/\Z}$. Let $n_0 \in \Z$ be arbitrary, and suppose that $n'$ ranges over any interval of integers $J$ of length at most $1/\beta$. The number of $n'$ for which $\alpha (n_0 + kn') \Z \in I$ is then at most $1 + \epsilon/\beta$. Since $[N]$ may be divided into $\leq 2k + \beta N$ progressions of the form $\{ n_0 + kn' : n' \in J\}$ we obtain from our assumption the inequality
\begin{equation}\label{compare} \delta N \leq \# \{n \in [N] : \alpha n \Z \in I\} \leq (1 + \frac{\epsilon}{\beta})(2k + \beta N) \ll k + \frac{\epsilon k}{\beta} + \beta N + \epsilon N .\end{equation}
Now the lemma is trivial if $N \ll \delta^{-10C}$ and follows immediately from Proposition \ref{quant-kron} when $\epsilon \gg \delta^{10C}$, so suppose that neither of these is the case. Then all of the terms except the second on the right-hand side of \eqref{compare} are negligible, and we deduce that
\[ \delta N \ll k \epsilon/\beta.\]
This immediately implies the result.\endproof

The main idea in the proof of Proposition \ref{quant-kron}, of course, was that the space of Lipschitz functions is essentially spanned by the space of pure phase functions $e(k \cdot \theta)$. Thus we were able to assert that if the condition \eqref{equi-failure} fails for some $F$, then it also fails (albeit with a smaller value of $\delta$) for a pure phase function with not-too-large frequency.  

A similar observation turns out to be essential in the analysis of polynomial sequences on general nilmanifolds $G/\Gamma$ (cf. the proof of \cite[Theorem 2.17]{leibman-single-poly}). Though we will not be discussing general sequences for quite a while, this does seem to be an appropriate place to state and prove a lemma which generalizes the observations just made.  For this, we will be working primarily on the vertical torus:

\begin{definition}[Vertical torus]\label{vert-def} Suppose that $G/\Gamma$ is a nilmanifold and that $G_{\bullet}$ is a filtration of degree $d$. Note that $G_d$ then lies in the centre of $G$. We define the \emph{vertical torus} to be $G_d/(\Gamma \cap G_d)$, and the \emph{vertical dimension} $m_d$ to be $m_d := \dim G_d$; the last $m_d$ coordinates of the Mal'cev coordinate map $\psi$ may be used to canonically identify $G_d$ and $G_d/(\Gamma \cap G_d)$ with $\R^{m_d}$ and $\R^{m_d}/\Z^{m_d}$ respectively.  Also observe that the vertical torus acts canonically on the nilmanifold $G/\Gamma$, thus we can define\footnote{Here we have a slight clash between the additive notation for the torus $\R^{m_d}/\Z^{m_d}$ and the multiplicative notation for the group $G$.  We hope this will not confuse the reader.} $\theta y \in G/\Gamma$ for all $\theta \in \R^{m_d}/\Z^{m_d}$ and $y \in G/\Gamma$.
\end{definition}

\begin{definition}[Vertical characters]\label{vert-char-def}
A \emph{vertical character} is a continuous homomorphism $\xi : G_d \rightarrow \R/\Z$ such that $\Gamma \cap G_d \subseteq \ker \xi$ (in particular, $\xi$ can also be meaningfully defined on $G_d/\Gamma_d \cong \R^{m_d}/\Z^{m_d}$). Any such character has the form $\xi(x) = k \cdot x$ for a unique $k \in \Z^{m_d}$, where we identify $G_d$ with $\R^{m_d}$.  We refer to $k$ as the \emph{frequency} of the character $\xi$, and $|\xi| :=|k|$ as the \emph{frequency magnitude}.  For instance the trivial character $\xi \equiv 0$ has frequency $0$.
\end{definition}

\begin{definition}[Vertical oscillation]\label{vert-freq-def}
Let $F : G/\Gamma \rightarrow \C$ be a Lipschitz function and suppose that $\xi$ is a vertical character. We say that $F$ has \emph{vertical oscillation $\xi$} if we have $F(g_d\cdot x) = e(\xi(g_d)) F(x)$ for all $g_d \in G_d$ and $x \in G/\Gamma$. 
\end{definition} 

The next definition is a repetition of Definition \ref{almost-equidistribution}, except that we specialize to functions with a fixed vertical oscillation $\xi$. 

\begin{definition}[Equidistribution along a vertical character]\label{vert-freq-equi}
Let $g : \Z \rightarrow G$ be a polynomial sequence. We say that $(g(n)\Gamma)_{n \in [N]}$ is \emph{$\delta$-equidistributed} along a vertical character $\xi$ if
\[ \left|\E_{n \in [N]} F(g(n)\Gamma) - \int_{G/\Gamma} F\right| \leq \delta \Vert F \Vert_{\Lip}\]
for all Lipschitz functions $F : G/\Gamma \rightarrow \C$ with vertical oscillation $\xi$.
\end{definition}

The next lemma states that in order to check whether a sequence is equidistributed, it suffices to test that sequence against functions possessing a vertical oscillation.

\begin{lemma}[Vertical oscillation reduction]\label{vert-freq-red}  Let $G/\Gamma$ be a nilmanifold together with a filtration $G_{\bullet}$ of degree $d$. Let $m_d$ be as above, and let $0 < \delta < 1/2$.
Suppose that $g : \Z \rightarrow G$ is a polynomial sequence and that $(g(n)\Gamma)_{n \in [N]}$ is not $\delta$-equidistributed. Then there is a vertical character $\xi$ with $|\xi| \ll \delta^{-O_{m_d}(1)}$ such that $(g(n)\Gamma)_{n \in [N]}$ is not $\delta^{O_{m_d}(1)}$-equidistributed along the vertical oscillation $\xi$.
\end{lemma}

\proof We merely sketch this, for the argument is little more than a repetition of that used to prove Proposition \ref{quant-kron}. We begin with the same reductions. That is, assuming the existence of an $F : G/\Gamma \rightarrow \C$ such that 
\begin{equation}\label{equi-failure-gen}
\left|\E_{n \in [N]} F(g(n)\Gamma) - \int_{G/\Gamma} F\right| \geq \delta \Vert F \Vert_{\Lip},
\end{equation}
we weaken $\delta$ to $\delta/2$ and assume that $\int_{G/\Gamma} F = 0$, that $\Vert F \Vert_{\Lip} = 1$ and that $F$ is smooth.

Let $K$ be the same Fej\'er-type kernel as before, and now take $F_1: G \to \C$ to be the function obtained by convolving with $K$ in each $G_d/(\Gamma \cap G_d) \cong \R^{m_d}/\Z^{m_d}$-fibre,  that is to say
\[ F_1(y) := \int_{\R^{m_d}/\Z^{m_d}} F(\theta y)K(\theta)d\theta.\]
Fourier expansion on $\R^{m_d}/\Z^{m_d}$ gives
\[ F_1(y) = \sum_{k \in \Z^{m_d}} F^{\wedge}(y;k)\widehat{K}(k), \]
where
\[ F^{\wedge}(y;k) := \int_{\R^{m_d}/\Z^{m_d}} F(\theta y) e(-k \cdot \theta)d\theta.\]
Now for $g_d \in G_d \cong \R^{m_d}$ we have
\[ F^{\wedge}(g_d y;k) = \int F((\theta + g_d)y)  e(-k \cdot \theta)\, d\theta = e(k \cdot g_d)F^{\wedge}(y;f),\]
thus each function $F^{\wedge}(y;k)$ has vertical oscillation $\xi$, where $\xi(x) := k \cdot x$ is the vertical character with frequency $k$.

Using exactly the same estimates as in the proof of Proposition \ref{quant-kron}, we have $\Vert F - F_2 \Vert_{\infty} \leq \delta/4$, where
\[ F_2(y) := \sum_{k \in \Z^{m_d} : |k| \leq Q} F^{\wedge}(y;k) \widehat{K}(k)\] for some $Q = C_{m_d}\delta^{-2m_d - 1}$.
The rest of the argument proceeds exactly as before, and we see that if we take $\tilde F(y) := F^{\wedge}(y;k)$ for suitable $k \in \Z^{m_d}$, $|k| \ll \delta^{-O_{m_d}(1)}$,
we have
\[ \left|\E_{n \in [N]} \tilde F(g(n) \Gamma) - \int_{G/\Gamma} \tilde F\right| \gg \delta^{O_{m_d}(1)}\Vert \tilde F \Vert_{\Lip}.\]
Thus  $(g(n)\Gamma)_{n \in [N]}$ is not $\delta^{O_{m_d}(1)}$-equidistributed along the vertical character $\xi$, as desired.
\endproof

\section{The van der Corput trick and polynomial flows on tori}\label{sec4}

In the last section we introduced one important trick - the idea of decomposing a Lipschitz function into phases using Fourier analysis. In this section we introduce a second trick - namely, the use of van der Corput's inequality - and use this trick to study polynomial sequences on tori $\R^m/\Z^m$. Although our language is somewhat different, this is really just a \emph{reprise} of the standard theory of Weyl sums as used for instance in the study of Waring's problem (see, for example, \cite{vaughan}).

\begin{lemma}[van der Corput inequality]\label{vdc-lemma}
Let $N,H$ be positive integers and suppose that $(a_n)_{n \in [N]}$ is a sequence of complex numbers. Extend $(a_n)$ to all of $\Z$ by defining $a_n := 0$ when $n \notin [N]$. Then
\[ |\E_{n \in [N]} a_n |^2 \leq \frac{N + H}{H N}\sum_{|h| \leq H}\left(1 - \frac{|h|}{H}\right)\E_{n \in [N]} a_n \overline{a_{n+h}}.\]
\end{lemma}
\proof We have
\[ \sum_n a_n = \frac{1}{H} \sum_{-H < n \leq N} \sum_{h = 0}^{H-1} a_{n+h}.\] Thus, applying the Cauchy-Schwarz inequality, we have
\begin{align*}
\big|\sum_n a_n \big|^2 &= \frac{1}{H^2} \big| \sum_{-H < n \leq N} \sum_{h = 0}^{H-1} a_{n+h}\big|^2 \\ &\leq \frac{N+H}{H^2} \sum_{-H < n \leq N} \big| \sum_{h = 0}^{H-1} a_{n+h} \big|^2 \\ & = \frac{N+H}{H^2} \sum_{-H < n \leq N} \sum_{h = 0}^{H-1} \sum_{h' = 0}^{H-1} a_{n+h} \overline{a}_{n+h'},
\end{align*}
which is equivalent to the right hand side of the claimed inequality.\endproof

We will use the following simple (and rather crude) corollary of this, which we phrase in the contrapositive.

\begin{corollary}[van der Corput]\label{vdc}
Let $N$ be a positive integer and suppose that $(a_n)_{n \in [N]}$ is a sequence of complex numbers with $|a_n| \leq 1$. Extend $(a_n)$ to all of $\Z$ by defining $a_n := 0$ when $n \notin [N]$. Suppose that $0 < \delta < 1$ and that 
\[ |\E_{n \in [N]} a_n | \geq \delta.\]
Then for at least $\delta^2 N/8$ values of $h \in [N]$ we have
\[ |\E_{n \in [N]} a_{n+h}\overline{a_n} | \geq \delta^2/8.\]
\end{corollary}
\proof The result is vacuous if $N \leq 4/\delta^2$, so assume this is not the case. Suppose for a contradiction that the result is false. Apply Lemma \ref{vdc-lemma} with $H = N$. Then it is easy to see that we have
\[ \delta^2  \leq |\E_{n \in [N]} a_n|  \leq \frac{2}{N}\sum_{|h| \leq N}|\E_{n \in [N]} a_n \overline{a_{n+h}}|  \leq \frac{2}{N}\left( 1 + 2\left(\frac{\delta^2N}{8} + \frac{\delta^2 N}{8}\right)\right),\] where we have used the trivial estimate $|\E_{n \in [N]} a_n \overline{a_{n+h}}| \leq 1$ for those $h \in [N]$ such that $|\E_{n \in [N]} a_n \overline{a_{n+h}}| \geq \delta^2/8$, of which there are no more than $\delta^2 N/8$. Rearranging and using the fact that $N > 4/\delta^2$ we see that this is a contradiction.\endproof

The next proposition is the main result of this section, and is Theorem \ref{main-theorem} in the case $G = \R$, $\Gamma = \Z$ and with $g : \Z \rightarrow G$ an arbitrary polynomial.

\begin{proposition}[Weyl]\label{weyl-dichotomy}
Suppose that $g : \Z \rightarrow \R$ is a polynomial of degree $d$, and let $0 < \delta < 1/2$. Then either $(g(n)\md{\Z})_{n \in [N]}$ is $\delta$-equidistributed, or else there is an integer $k$, $1 \leq k \ll \delta^{-O_d(1)}$, such that $\Vert kg \md{\Z}\Vert_{C^{\infty}[N]} \ll \delta^{-O_d(1)}$.
\end{proposition}

We will deduce this from the following, which is nothing but a reformulation of Weyl's exponential sum estimate (see e.g. \cite{vaughan}).

\begin{lemma}[Weyl's exponential sum estimate]
\label{weyl}
Suppose that $g : \Z \rightarrow \R$ is a polynomial of degree $d$ with leading coefficient $\alpha_d$ and that 
\[ |\E_{n \in [N]} e(g(n))| \geq \delta\]
for some $0 < \delta < 1/2$. Then there is $k \in \Z$, $|k| \ll \delta^{-O_d(1)}$, such that 
\[ \Vert k\alpha_d \Vert_{\R/\Z} \ll \delta^{-O_d(1)}/N^d.\]
\end{lemma}
\proof We proceed by induction on $d$, the result having been established in \S \ref{sec3} in the case $d = 1$. We may assume that $N > \delta^{-C'_d}$ for some large $C'_d$ since the result is trivial otherwise. Applying van der Corput's estimate in the form of Corollary \ref{vdc} we deduce that there are $\gg \delta^2 N$ values of $h \in [N]$ such that
\[ |\E_{n \in [N]} e(g(n+h) - g(n))| \gg \delta^2.\]
For each such $h$, $g(n+h) - g(n)$ is a polynomial with degree $d-1$ and leading coefficient $hd\alpha_d$. Thus by the induction hypothesis there is, for $\gg \delta^2$ values of $h \in [N]$, some $1 \leq q_h \ll \delta^{-O_d(1)}$ such that we have
\[ \Vert h q_hd \alpha_d \Vert_{\R/\Z} \ll \delta^{-O_d(1)}/N^{d-1}\] for each of these values of $h$. Pigeonholing in the $q_h$, this implies that there is $q$, $1 \leq q \ll \delta^{-O_d(1)}$, such that
\[ \Vert hq \alpha_d \Vert_{\R/\Z, \delta^{-O_d(1)}} \ll \delta^{-O_d(1)}/N^{d-1}\] for $\gg \delta^{O_d(1)} N$ values of $h \in [N]$. Since $N$ is so large, Lemma \ref{strong-linear} may applied to conclude that there is $q' \ll \delta^{-O_d(1)}$ such that
\[ \Vert qq' \alpha_d \Vert_{\R/\Z} \ll \delta^{-O_d(1)}/N^d.\]
Redefining $q := qq'$, the result follows.\endproof

\emph{Proof of Proposition \ref{weyl-dichotomy}.} In this proof we allow all implied constants to depend on $d$.
Suppose that $g : \Z \rightarrow \R$
is a polynomial sequence of degree $d$ such that the orbit $(g(n)\Z)_{n \in [N]}$ on $\R/\Z$ is not $\delta$-equidistributed. Expand $g$ as a Taylor series
\begin{equation}\label{taylor-1} g(n) = \binom{n}{d} \alpha_d + \dots + \binom{n}{1} \alpha_1  + \alpha_0\end{equation}
and suppose as a hypothesis for induction on $r$, $0\leq r < d$, that we have shown that each of the coefficients $\alpha_{d},\alpha_{d-1},\dots,\alpha_{d-r}$ is nearly rational in the sense that $\Vert q\alpha_{d-i} \Vert_{\R/\Z} \ll \delta^{-O(1)}/N^{d-i}$ for some $q \ll \delta^{-O(1)}$ for $i = 0,\dots,r$. (The implied constants in the $O()$ notation may increase with each induction step, but there are only $d$ such steps, and we are allowing these constants to depend on $d$, so this is harmless.)  The statement we are trying to prove, Proposition \ref{weyl-dichotomy}, is the case $r = d-1$.

Now by the argument used in proving Proposition \ref{quant-kron} (or indeed by simply quoting Lemma \ref{vert-freq-red}), there is $k \in \Z$, $0 < |k| \ll \delta^{-O(1)}$, such that 
\begin{equation}\label{g-large} |\E_{n \in [N]} e(k g(n))| \gg \delta^{O(1)}.\end{equation}
The base case $r = 0$ of the induction follows immediately from Lemma \ref{weyl}. Suppose now that we have established the result for some $r$, and wish to establish it for $r+1$.
Set 
\[ g'(n) := g(n) - \binom{n}{d} \alpha_d - \dots - \binom{n}{d-r} \alpha_{d-r} = \binom{n}{d-r-1} \alpha_{d-r-1} + \dots + \alpha_0.\]
Set $Q := qd!$, and write $\alpha_{d-i} = a_{d-i}/q + O(\delta^{-O(1)}/N^{d-i})$, $i = 0,\dots,r$ for some integers $a_{d-i}$.
For any $n_0 \in \Z$ for any $n' \in \Z$ we have
\begin{align*}
g'(n_0 + Qn') - g'(n_0) &  = g(n_0 + Qn') - g(n_0) - \frac{1}{q}\sum_{i=0}^{r} a_{d-i} \left[ \binom{n_0 + Qn'}{d-i} - \binom{n_0}{d-i} \right] \\ & + O(\delta^{-O(1)}) \sum_{i = 0}^r \frac{1}{N^{d-i}}  \left[ \binom{n_0 + Qn'}{d-i} - \binom{n_0}{d-i} \right].\end{align*}
Set $N' := \lfloor \delta^{C'_d}N\rfloor$ for some suitably large $C'_d$ and suppose that $n' \in [N']$ and also that $|n_0| \leq 2N$. Then the last term here is $O(\delta^{C'_d - O(1)})$. The first term is an integer, since 
\[ \binom{n_0 + Q}{j} - \binom{n_0}{j} = \sum_{i=1}^j \binom{Q}{i} \binom{n_0}{j-i} \equiv 0 \md{q}\] for all $j \leq d$.
Thus we see that if $n' \in [N']$ and $|n_0| \leq 2N$ then
\begin{equation}\label{approx-8} g'(n_0 + Qn') - g'(n_0) = g(n_0 + Qn') - g(n_0) + O(\delta^{C'_d-O(1)}) \md{\Z}.\end{equation}
Splitting $[N]$ into progressions of common difference $Q$ and length $[N']$ plus a negligible error we see from \eqref{g-large} that there is $n_0$, $|n_0| \leq 2N$, such that 
\[ |\E_{n' \in [N']} e(k g(n_0 + Qn'))| \gg \delta^{O(1)}.\]
It follows from \eqref{approx-8} that
\[ |\E_{n' \in [N']} e(k g'(n_0 + Qn'))| \gg \delta^{O(1)}.\]
By Lemma \ref{weyl} we see that the leading coefficient $\alpha' := kQ^{d-r-1}\alpha_{d-r-1}/(d-r-1)!$ of this polynomial is nearly rational in the sense that there is $1 \leq q' \ll \delta^{-O(1)}$ such that $\Vert q'\alpha'\Vert_{\R/\Z} \ll \delta^{-O(1)}/N^{d-r-1}$. It follows that there is $1 \leq q'' \ll \delta^{-O(1)}$ such that $\Vert  q'' \alpha_{d-r-1} \Vert_{\R/\Z} \ll \delta^{-O(1)}/N^{d-r-1}$. Setting $\tilde q := qq''$ we now clearly have $1 \leq \tilde q \ll \delta^{-O(1)}$ and also $\Vert \tilde q \alpha_{d-i}\Vert_{\R/\Z} \ll \delta^{-O(1)}/N^{d-i}$ for $i = 0,\dots,r+1$.

This concludes the proof of the inductive step and hence of the proposition.

We will also need a ``strong recurrence'' result for polynomials $g : \Z \rightarrow \R$, generalizing the linear result, Lemma \ref{strong-linear}, that we obtained in the last section. This is in fact an easy deduction from the Proposition \ref{weyl-dichotomy} and Lemma \ref{strong-linear}.

\begin{lemma}[Strongly recurrent polynomials are highly non-diophantine]\label{strong-polynomial}
Let $d \geq 0$, and suppose that $g : \Z \rightarrow \R$ is a polynomial sequence of degree $d$. Suppose that $0 < \delta < 1/2$ and $\epsilon \leq \delta/2$, that $I \subseteq \R/\Z$ is an interval of length $\epsilon$, and that $g(n)\md{\Z} \in I$ for at least $\delta N$ values of $n \in [N]$. Then there is a $k \in \Z$, $0 < |k| \ll \delta^{-O_d(1)}$, such that $\Vert kg \md{\Z}\Vert_{C^{\infty}[N]} \ll \epsilon \delta^{-O_d(1)}$. 
\end{lemma}

\proof In this proof we allow all implied constants to depend on $d$. If $\epsilon \gg \delta^{C_d}$ for some large $C_d$ depending only on $d$ then the result follows immediately from Proposition \ref{weyl-dichotomy}, so assume this is not the case. Expand $g$ in a Taylor series as in \eqref{taylor-1}, with coefficients $\alpha_0,\dots,\alpha_d$. It follows from the assumption that none of the polynomials $\lambda g$, $\lambda \leq \delta/2\epsilon$, is $\delta^{O(1)}$-equidistributed on $[N]$. Thus by Proposition \ref{weyl-dichotomy} we have see that for each $\lambda \leq \delta/2\epsilon$ there is $q_{\lambda} \ll \delta^{-O(1)}$ such that $\Vert q_{\lambda} \lambda \alpha_i \Vert_{\R/\Z} \ll \delta^{-O(1)}/N^i$ for $i = 0,\dots,d$. Pigeonholing in the possible values of $q_{\lambda}$ we see that there is $q \ll \delta^{-O(1)}$ such that for $\gg \delta^{O(1)}/\epsilon$ values of $\lambda \leq \delta/2\epsilon$ we have $\Vert \lambda q \alpha_i\Vert_{\R/\Z} \ll \delta^{-O(1)}/N^i$ for each $i = 0,\dots,d$. It follows from Lemma \ref{strong-linear} that for each $i$ there is $q_i \ll \delta^{-O(1)}$ such that $\Vert q_i\alpha_i \Vert_{\R/\Z} \ll \epsilon \delta^{-C_d}/N^i$. Writing $\tilde q := q_1 \dots q_d$ we see that $\tilde q \ll \delta^{-O(1)}$ and that $\Vert q \alpha_i \Vert_{\R/\Z} \ll \epsilon \delta^{-O(1)}/N^i$ for all $i$. This concludes the proof of the proposition.\endproof 

\section{The Heisenberg example}\label{heisenberg-sec}

In this section we discuss the first example which is not just a rephrasing of classical work on equidistribution, establishing Theorem \ref{main-theorem} for a linear sequence on the Heisenberg nilmanifold \eqref{heisen}, thus $s=d = 2$, and $m=3$.
Strictly speaking, this section is not necessary in order to prove Theorem \ref{main-theorem} in the general case, however we present this ``worked example'' here in order to illustrate the key ideas of the main argument in a simplified model setting.  (Also, a key computation in this setting, namely Proposition \ref{bracket-poly-lem}, will be reused in the main argument.) As in the preceding section, the idea is to use van der Corput's inequality to reduce the problem to a simpler problem, and in particular to reduce to a ``$1$-step'' or ``abelian'' problem that can be treated by the tools of the previous section.  This turns out to work, but it will take a certain amount of algebraic manipulation to see the $1$-step structure emerge from van der Corput's inequality applied to the $2$-step Heisenberg situation.

Let us begin with a brief tour of the Heisenberg example \eqref{heisen}.
We have $\g = \left(\begin{smallmatrix} 0 & \R & \R \\ 0 & 0 & \R \\ 0 & 0 & 0\end{smallmatrix}\right)$, with the exponential map being given by
\[ \exp\left(\begin{smallmatrix} 0 & x & y \\ 0 & 0 & z \\ 0 & 0 & 0\end{smallmatrix}\right) = \left(\begin{smallmatrix} 1 & x & y + \frac{1}{2}xz \\ 0 & 1 & z \\ 0 & 0 & 1\end{smallmatrix}\right)\] and the logarithm map by
\[ \log\left(\begin{smallmatrix} 1 & x & y \\ 0 & 1 & z \\ 0 & 0 & 1\end{smallmatrix}\right) = \left(\begin{smallmatrix} 0 & x & y - \frac{1}{2}xz \\ 0 & 0 & z \\ 0 & 0 & 0\end{smallmatrix}\right).\]
Observe that $\log \Gamma$ is not quite a lattice in $\R^3$, although it is a finite union of lattices.

Consider the elements $X_1,X_2,X_3 \in \g$ defined by $X_1 := \left(\begin{smallmatrix} 0 & 1 & 0 \\ 0 & 0 & 0 \\ 0 & 0 & 0\end{smallmatrix}\right)$, $X_2 := \left(\begin{smallmatrix} 0 & 0 & 0 \\ 0 & 0 & 1 \\ 0 & 0 & 0\end{smallmatrix}\right)$ and $X_3 := \left(\begin{smallmatrix} 0 & 0 & 1 \\ 0 & 0 & 0 \\ 0 & 0 & 0\end{smallmatrix}\right)$. It is easy to see that $\mathcal{X} = \{X_1,X_2,X_3\}$ is a Mal'cev basis adapted to the lower central series filtration $G_{\bullet}$. A simple computation confirms that
\[ \exp(t_1 X_1)\exp(t_2 X_2)\exp(t_3X_3) = \left(\begin{smallmatrix} 1 & t_1 & t_1t_2 + t_3 \\ 0 & 1 & t_2 \\ 0 & 0 & 1\end{smallmatrix}\right),\] and so the Mal'cev coordinate map $\psi_\X: G \to \R^3$ is given by
\[ \psi_\X \left(\begin{smallmatrix} 1 & x & y \\ 0 & 1 & z \\ 0 & 0 & 1\end{smallmatrix}\right) = (x,z,y - xz).\]
The horizontal torus is isomorphic to $(\R/\Z)^2$, and the projection $\pi: G \to (\R/\Z)^2$ is given by $\pi \left(\begin{smallmatrix} 1 & x & y \\ 0 & 1 & z \\ 0 & 0 & 1\end{smallmatrix}\right) = (x,z)$.

We shall be working through the special case of Theorem \ref{main-theorem} in the case when $g : \Z \rightarrow G$ is a linear sequence. To simplify the exposition very slightly we will assume that this sequence has no constant term, thus $g(n) = a^n$ for some $a \in G$. Note that $g \in \poly(\Z,G_{\bullet})$, where $G_{\bullet}$ is the lower central series filtration. Thus the sequence $g$ has degree $2$.

\begin{proposition}[Main theorem, Heisenberg case]\label{heisenberg-case}
Let $G/\Gamma$ be the $2$-step Heisenberg nilmanifold with the Mal'cev basis $\X$ described above, and let $g : \Z \rightarrow G$ be a linear sequence of the form $g(n) = a^n$. Let $\delta > 0$ be a parameter and let $N \geq 1$ be an integer. Then either $(g(n)\Gamma)_{n \in [N]}$ is $\delta$-equidistributed, or else there is a horizontal character $\eta$ with $0 < |\eta| \ll \delta^{-O(1)}$ such that $\Vert \eta(a)\Vert_{\R/\Z} \ll \delta^{-O(1)}/N$.
\end{proposition}

\remark 
Note that, since $g(n)$ is linear, the last condition here is equivalent to the statement that $\Vert \eta \circ g \Vert_{C^{\infty}[N]} \ll \delta^{-O(1)}$.

\proof By Lemma \ref{vert-freq-red} we may assume that there is a function $F : G/\Gamma \rightarrow \C$ with a vertical oscillation $\xi$ with $\|\xi\| \ll \delta^{-O(1)}$, and $\Vert F \Vert_{\Lip} =1$, such that 
\begin{equation}\label{assumption} \left|\E_{n \in [N]} F(a^n \Gamma) - \int_{G/\Gamma} F\right| \gg \delta^{O(1)}.\end{equation}
We split into two cases: $\xi \equiv 0$ and $\xi \not \equiv 0$.

If $\xi \equiv 0$, then $F$ is $G_2$-invariant, which means we may factor through $\pi$ to get a function $\tilde F : \R^2/\Z^2 \rightarrow \C$ defined by
\[ F(x) = \tilde F(\pi (x)).\]
It is clear that $\Vert \tilde F\Vert_{\Lip} \leq 1$. Equation \eqref{assumption} implies that
\[ |\E_{n \in [N]} \tilde{F}(n \pi(a)) - \int_{\R^2/\Z^2} \tilde{F}| \gg \delta^{O(1)}\Vert \tilde F \Vert_{\Lip}.\] Proposition \ref{heisenberg-case} in this case now follows immediately from Proposition \ref{quant-kron}.  Note how the $G_2$-invariance allowed us to reduce a $2$-step problem into a $1$-step one.

Suppose then that $\xi \not \equiv 0$. The integral of $F$ over every translate of $G_2/(\Gamma \cap G_2)$ is then zero, and hence $\int_{G/\Gamma} F = 0$. Thus \eqref{assumption} becomes
\[ |\E_{n \in [N]} F(a^n \Gamma) | \geq \delta^{O(1)}.\]
We now come to one of the key ideas of the proof, which is to apply the van der Corput lemma, Corollary \ref{vdc}. This tells us that there are $\gg \delta^{O(1)}N$ values of $h \in [N]$ such that 
\begin{equation}\label{to-use-10}|\E_{n \in [N]} F(a^{n+h} \Gamma)\overline{F(a^n \Gamma)} | \gg \delta^{O(1)}.\end{equation}
It is very natural to try and interpret this in terms of a nilsequence on the product nilmanifold $G^2/\Gamma^2$. To do this we first observe by direct computation that any $x \in G$ may be factored uniquely as $\{x\}[x]$, where $\psi(\{x\}) \in [0,1)^3$ and $[x] \in \Gamma$. 

Let us, then, factor $a^h = \{a^h\}[a^h]$. The inequality \eqref{to-use-10} implies that
\[ |\E_{n \in [N]} F(a^n \{a^h\} \Gamma)\overline{F(a^n \Gamma)}| \gg \delta^{O(1)}\] for $\gg \delta^{O(1)}N$ values of $h$. This can be rewritten as
 \begin{equation}\label{eq55} |\E_{n \in [N]} \tilde F_h(\tilde a_h^{n} \Gamma^2)| \gg \delta^{O(1)}\end{equation} for $\gg \delta^{O(1)}N$ values of $h$,
 where $\tilde F_h : G^2/\Gamma^2 \rightarrow \C$ is given by
 \[ \tilde F_h(x,y) := F(\{a^h\} x) \overline{F(y)}\] and the element $\tilde a_h$ is given by
 \[ \tilde a_h := (\{a^h\}^{-1} a \{a^h\},a).\]
 At first sight, the estimates \eqref{eq55} do not appear much better than our original estimate \eqref{assumption}; indeed, it seems ``worse'' since we are now working on a $6$-dimensional $2$-step nilmanifold rather than a $3$-dimensional $2$-step one. 

The crucial observation, however, is that all the elements $\tilde a_h$ in fact lie not just in $G^2$, but in the smaller group
\[ G^{\Box} = G \times_{G_2} G := \{(g,g') : g^{-1} g' \in G_2\}.\] This is also a $2$-step nilpotent, connected, simply connected Lie group (of dimension 4). It is not hard to check that $[G^{\Box},G^{\Box}]$ is the \emph{diagonal group} $G_2^{\Delta} := \{(g_2,g_2) : g_2 \in G_2\}$, and that one can take for a Mal'cev basis of $G^{\Box}/\Gamma^{\Box}$ the collection $\X^{\Box} = \{X_1^{\Box},X_2^{\Box},X_3^{\Box},X_4^{\Box}\}$ given by
\[ X_1^{\Box} = \left(\begin{smallmatrix} 0 & 1 & \{0,0\} \\ 0 & 0 & 0 \\ 0 & 0 & 0 \end{smallmatrix}\right), X_2^{\Box} = \left(\begin{smallmatrix} 0 & 0 & \{0,0\} \\ 0 & 0 & 1 \\ 0 & 0 & 0 \end{smallmatrix}\right), X_3^{\Box} = \left(\begin{smallmatrix} 0 & 0 & \{1,0\} \\ 0 & 0 & 0 \\ 0 & 0 & 0 \end{smallmatrix}\right) \quad \mbox{and} \quad X_4^{\Box} = \left(\begin{smallmatrix} 0 & 0 & \{1,1\} \\ 0 & 0 & 0 \\ 0 & 0 & 0 \end{smallmatrix}\right), \] where we have written
\[ \left(\begin{smallmatrix} 0 & x & \{y,y'\} \\ 0 & 0 & z\\ 0 & 0 & 0 \end{smallmatrix}\right) := \left( \left(\begin{smallmatrix} 0 & x & y \\ 0 & 0 & z \\ 0 & 0 & 0 \end{smallmatrix}\right),\left(\begin{smallmatrix} 0 & x & y' \\ 0 & 0 & z \\ 0 & 0 & 0 \end{smallmatrix}\right)  \right).\]
This allows us to identify the horizontal torus of $G^{\Box}/\Gamma^{\Box}$ with $\R^3/\Z^3$ by projecting onto the first three coordinates.

Now \eqref{eq55} implies that for $\gg \delta^{O(1)}N$ values of $h$ we have
\begin{equation}\label{eq72} |\E_{n \in [N]} F_h^{\Box}((a_h^\Box)^n\Gamma^{\Box})| \gg \delta^{O(1)},\end{equation}
where $F_h^{\Box}$ and $a_h^{\Box}$ are the restrictions of $\widetilde{F}_h$ and $\tilde a_h$ to $G^{\Box}$, and $\Gamma^{\Box} := \Gamma \times_{\Gamma \cap G_2} \Gamma$.  By inspecting the action of $G_2^2$ on $F_h^\Box$ (and the hypothesis $\xi \not \equiv 0$) we also conclude that $\int_{G^\Box/\Gamma^\Box} F_h^\Box = 0$.

Now, the group $G^\Box$ is still $2$-step nilpotent, so we do not appear to have reduced to a $1$-step situation yet.  However, recall that $F$ has vertical oscillation $\xi$.  Using this and the fact that $g_2$ is central in $G$, we obtain
\[ F_h^{\Box}((g_2,g_2) \cdot (g,g'))  = F(\{a^h\} g_2 g) \overline{F(g_2 g')} = \xi(g_2) \overline{\xi(g_2)} F(\{a^h\} g) \overline{F(g')} = F_h^{\Box}((g,g')).\]
Thus $F_h^{\Box}$ is $[G^{\Box},G^{\Box}]$-invariant. In \eqref{eq72} we may therefore factor through the projection $\pi^{\Box}$ to obtain
\[ |\E_{n \in [N]} \tilde {F}_h(n \pi^{\Box}(\tilde a_h))| \gg \delta^{O(1)}\] for $\gg \delta^{O(1)}N$ values of $h$, where the function $\tilde F_h : \R^3/\Z^3 \rightarrow \C$ is defined by
\[ \tilde{F}_h (\pi^{\Box}(x)) = F_h^{\Box}(x \Gamma^{\Box}).\]
We leave it to the reader to check that $\Vert \tilde{F}_h\Vert_{\Lip} = O(1)$ (in the general case to follow this computation is given in more detail).  Since $F_h^\Box$ has mean zero, we see that $\tilde F_h$ has mean zero also.

We are now finally in a situation in which we may apply ``$1$-step'' tools.  Indeed, from Proposition \ref{quant-kron} we see that for each $h$ there is some $k_h^{\Box} \in \Z^3$, $|k^{\Box}_h| \ll \delta^{-O(1)}$ such that 
\[ \Vert k_h^{\Box} \cdot \pi^{\Box}(\tilde a_h) \Vert_{\R/\Z} \ll \delta^{-O(1)}/N.\]
Pigeonholing in $h$, we may assume that $k_h^{\Box} = k^{\Box}$ is independent of $h$.
Define $\eta : G^{\Box} \rightarrow \R/\Z$ by
\[ \eta(x) := k^{\Box} \cdot \pi^{\Box}(x).\] Then $\eta$ is an additive homomorphism which annihilates $[G^{\Box},G^{\Box}]$ and $\Gamma^{\Box}$, and we have
\begin{equation}\label{star-5} \Vert \eta (\tilde a_h) \Vert_{\R/\Z} \ll \delta^{-O(1)}/N\end{equation} for $\gg \delta^{O(1)} N$ values of $h \in[N]$.

Our task now is to ``piece together'' these pieces of information for many different $h$ to deduce Proposition \ref{heisenberg-case}. We begin by factoring the character $\eta$ on $G^\Box$ into two simpler components, which originate from $G$ (or $G_2$) rather than $G^\Box$.

\begin{lemma}[Decomposition of $\eta$]\label{g-box-struct}
There exist horizontal characters $\eta_1 : G \rightarrow \R/\Z$ and $\eta_2 : G_2 \rightarrow \R/\Z$ on on $G$ and $G_2$ respectively \textup{(}thus $\eta_1$ annihilates $\Gamma$ and $\eta_2$ annihilates $\Gamma \cap G_2$\textup{)} such that 
\begin{equation}\label{decomposition} \eta(g',g) = \eta_1(g) + \eta_2(g'g^{-1})\end{equation}
for all $(g,g') \in G^{\Box}$. 
Furthermore we have $| \eta_1 |, |\eta_2| \ll \delta^{-O(1)}$.
\end{lemma}
\proof Since $\eta$ is an additive homomorphism we have $\eta(g',g) = \eta((g'g^{-1},1) \cdot (g,g)) = \eta(g,g) + \eta(g'g^{-1},1)$. Thus if we define $\eta_1(g) := \eta(g,g)$ and $\eta_2(g_2) := \eta(g_2,\id_G)$ then \eqref{decomposition} is immediately seen to hold. 
Now $\eta_1$ is a horizontal character because $\eta$ annihilates $\Gamma^{\Box}$, which contains $\Gamma^{\Delta}$. Furthermore $\Gamma^{\Box}$ also contains $(\Gamma \cap G_2) \times \id_G$, and hence $\eta_2$ annihilates $\Gamma \cap G_2$ as claimed. The bounds on $|\eta_1|$ and $|\eta_2|$ are left as an exercise to the reader; one may compute explicitly with the Mal'cev bases $\X^{\Box}$ and $\X$ on $G^{\Box}/\Gamma^{\Box}$ and $G/\Gamma$ respectively. \endproof

Using this decomposition and the fact that, in the Heisenberg group, we have the identity $x^{-1}yx y^{-1} = [x,y]$ since $[x,y]$ is central, we see that 
\[ \eta(\tilde a_h) = \eta_1(a) + \eta_2 ([a, \{a^h\}]).\]
Now a straightforward computation with matrices confirms that if $\psi(x) = (t_1,t_2,t_3)$ and $\psi(y) = (u_1,u_2,u_3)$ then $\psi([x,y]) = (0,0,t_1u_2- t_2u_1)$, and also that if $\psi(a) = (\gamma_1,\gamma_2,\ast)$ then $\psi(\{a^h\}) = (\{\gamma_1 h\}, \{\gamma_2 h\},\ast)$, where we do not care about the values of the coordinates marked with an asterisk $\ast$. 
Thus if we write $\gamma := (\gamma_1,\gamma_2) = \pi(a)$ and $\zeta := (-\gamma_2,\gamma_1)$ then
\[ \eta (\tilde a_h) = k_1 \cdot \gamma + k_2 \zeta \cdot \{\gamma h\},\]
where $k_1,k_2 = O(\delta^{-O(1)})$ are the frequencies of $\eta_1, \eta_2$ respectively.
Thus if \eqref{star-5} holds then 
\begin{equation}\label{eq89} \Vert  k_1 \cdot \gamma + k_2 \zeta \cdot \{\gamma h\} \Vert_{\R/\Z} \ll \delta^{-O(1)}/N\end{equation} for $\gg \delta^{O(1)} N$ values of $h$. 
The next proposition derives diophantine information concerning $\gamma$ and $\zeta$ from a hypothesis such as this. In fact we handle a slightly more general situation, since this will be useful when we come to handle the general case of Theorem \ref{main-theorem}. In the following proposition we shall take $\alpha = 0$ and $m = 2$; the proof when $\alpha = 0$ is actually considerably shorter and the reader may care to work through that case to better understand the argument.

\begin{proposition}[Bracket polynomial lemma]\label{bracket-poly-lem} Let $\delta \in (0,1)$ and let $N \geq 1$ be an integer. Suppose that $\alpha, \beta \in \R$ and that $|\alpha| \leq 1/\delta N$. Suppose that $\gamma \in \R^m/\Z^m$ and that $\zeta \in \R^m$ satisfies $|\zeta| \leq 1/\delta$. Suppose that for at least $\delta N$ values of $h \in [N]$ we have
\begin{equation}\label{assump-10} \Vert \beta + \alpha h + \zeta \cdot \{\gamma h\} \Vert_{\R/\Z} \leq 1/\delta N.\end{equation}
Then either $|\zeta_i| \ll_m \delta^{-O_m(1)}/N$ for all $1 \leq i \leq m$, or else there is some $k \in \Z^m$, $|k| \ll_m \delta^{-O_m(1)}$, such that $\Vert k \cdot \gamma \Vert_{\R/\Z} \ll_m \delta^{-O_m(1)}/N$. 
\end{proposition}
\proof If $\sup_i |\zeta_i| \leq 1/\delta N$ then we are done, so assume this is not the case. Then the assumption implies that $\Vert \beta + \alpha h \Vert_{\R/\Z} \leq (1 + m)\sup_i|\zeta_i|$ for $\geq \delta N$ values of $h \in [N]$. Then Lemma \ref{strong-linear} implies that there is $q \ll \delta^{-C}$ such that $\Vert q \alpha \Vert_{\R/\Z} \ll_m \sup_i|\zeta_i|\delta^{-C}/N$ for some absolute constant $C > 0$. Since we are assuming that $|\alpha| \leq 1/\delta N$ this forces us to conclude that in fact $|\alpha| \ll_m \sup_i |\zeta_i|\delta^{-C}/N$ unless $N \ll_m \delta^{-O(1)}$, in which case the result is trivial in any case.

Split $[N]$ into intervals of length between $N'$ and $2N'$, where $N' := c_m \delta^{C + 1} N$ and $c_m > 0$ is a small number to be chosen later. By the pigeonhole principle, we can find one of these intervals $I$ in which there are $\geq \delta |I|$ values of $h$ such that \eqref{assump-10} holds. If $c_m$ is chosen sufficiently small then $\alpha h$ does not vary by more than $\frac{\delta}{20} \sup_i|\zeta_i|$ on such an interval, and we conclude that there is $\theta$ such that 
\[ \Vert \theta + \zeta \cdot \{\gamma h\}\Vert_{\R/\Z} \leq \frac{\delta}{20} \sup_i|\zeta_i| + \frac{1}{\delta N}\] for at least $\delta |I|$ values of $h \in I$.
Now if $\sup_i |\zeta_i| \leq \frac{20}{\delta^2 N}$ then the proposition holds, so we may assume that this is not the case, in which eventuality we have
\begin{equation}\label{to-handle} \Vert \theta + \zeta \cdot \{\gamma h\}\Vert_{\R/\Z} \leq \frac{\delta}{10} |\zeta_i| \end{equation} for some $i \in [m]$ and for at least $\delta |I|$ values of $h \in I$.  We then set
\[ \Omega := \left\{t \in \R^m/\Z^m : \Vert \theta + \zeta \cdot \{t\} \Vert_{\R/\Z} \leq \frac{\delta}{10} |\zeta_i|\right\}\]
and 
\[ \tilde\Omega := \{ x \in \R^m/\Z^m: \dist(x,\Omega) < \delta/10 \}.\]
For fixed $u \in \R^m/\Z^m$ the slice
\[ \{ t \in \tilde\Omega : t_j = u_j \;\;\mbox{for $j \neq i$}\}\]
is a union of intervals of length less than $\delta/2$, and so $\vol(\tilde\Omega) \leq \delta/2$. 
Let $F: \R^m/\Z^m \to \R^+$ be the function
\[ F(x) := \max\left(1 - \frac{10\dist(x,\Omega)}{\delta},0\right).\]
Then $F = 1$ on $\Omega$ and so our assumption implies that
\begin{equation}\label{fact-11} \E_{n \in I} F(\gamma n) \geq \delta .\end{equation}
On the other hand $F$ is supported on $\widetilde{\Omega}$ and so
\begin{equation}\label{fact-12} \int_{\R^m/\Z^m} F(x) \, dx \leq \vol(\tilde{\Omega}) \leq \frac{\delta}{2}.\end{equation}
Thus of course
\[|\E_{n \in I} F(\gamma n) - \int_{\R^m/\Z^m} F(x)\, dx| \leq \frac{\delta}{2}.\]
However $F$ has been constructed so that $\Vert F \Vert_{\Lip} \ll 1/\delta$ (we leave this as an exercise) and so we conclude that $(\gamma n)_{n \in I}$ is not $c\delta^2$-equidistributed. 
Applying Proposition \ref{quant-kron} we conclude that there is $1 \leq k \ll \delta^{-O_m(1)}$ such that $\Vert k \cdot \gamma \Vert_{\R/\Z} \ll \delta^{-O_m(1)}/N' \ll \delta^{-O_m(1)}/N$, and the claim follows.
\endproof

Recall that in our efforts to prove Proposition \ref{heisenberg-case} had established the condition \eqref{eq89}. Applying Proposition \ref{bracket-poly-lem} and recalling that $\gamma = (\gamma_1,\gamma_2)$ and $\zeta = (-\gamma_2,\gamma_1)$ we see that in all cases there is some nonzero $k' \in \Z^2$ with $|k'| \ll \delta^{-O(1)}$ such that $\Vert k' \cdot \gamma\Vert_{\R/\Z} \ll \delta^{-O(1)}/N$, that is to say $\Vert k' \cdot \pi(a) \Vert_{\R/\Z} \ll \delta^{-O(1)}/N$. This concludes the proof of Proposition \ref{heisenberg-case}.\endproof

Let us pause for a moment to consider the form of the argument just presented. There were two places where we reduced matters to a simpler situation. First of all in the case $\xi \equiv 0$ we were able to consider $F$ as a function on a $1$-step nilmanifold. Secondly when we applied the van der Corput trick we found ourselves with a function $F_h^{\Box}$ which had 0 as a vertical frequency, and so we were again able to reduce to the $1$-step case, although we had to restrict the ambient nilmanifold (from $G^2/\Gamma^2$ to $G^\Box/\Gamma^\Box$) and also quotient out by a commutator group $[G^\Box,G^{\Box}]$ before the $1$-step structure became manifest. This already makes it clear that some kind of induction is going on, and in the general case we will see this quite clearly.

\section{Polynomial sequences in nilpotent groups}
\label{poly-sequences-sec}

Our analysis of linear sequences on the Heisenberg example captured much of the essence of the proof of Theorem \ref{main-theorem} in general. What it did not reveal, however, was the rather subtle structure of the space of polynomial sequences $g : \Z \rightarrow G$. In this section we begin by establishing a remarkable result of Lazard \cite{lazard}, which asserts that $\poly(\Z,G_{\bullet})$ is a group for any filtration $G_{\bullet}$. Lazard's proof uses the Lie algebra $\g$ and it works if $G$ is a connected and simply-connected Lie group (as in the present paper). However it turns out that the result is true with no topological assumptions on $G$, and indeed in the greater generality of so-called polynomial mappings from $H$ to $G$, where $H$ is an arbitrary group. This result is due to Leibman \cite{leibman-group-2} (see also \cite{leibman-group-1} for a proof of the special case $H = \Z$).

We will then use the Lazard-Leibman results to derive sundry further results concerning the representation of elements of $\poly(\Z,G_{\bullet})$ in coordinates. In fact, keeping in mind our intention to prove multiparameter results in \S \ref{sec12}, we develop the theory of polynomial maps $\poly(\Z^t,G_{\bullet})$. 

\begin{definition}[Polynomial maps]
Let $H$ be a group and let $G$ be a nilpotent group with a filtration $G_{\bullet}$. If $g : H \rightarrow G$ is a map and if $h \in H$ we write $\partial_h g$ for the map defined by $\partial_h g(x) = g(xh)g(x)^{-1}$. We say that $g$ is a polynomial map with coefficients in $G_{\bullet}$ if we have $\partial_{h_i}\dots \partial_{h_1}g(x) \in G_i$ for all choices of $i$ and for all $h_1,\dots,h_i \in H$ and $x \in G$. We write $\poly(H,G_{\bullet})$ for the collection of all such mappings. If $g : H \rightarrow G$ is a map we say that $g$ is a \emph{polynomial sequence of degree at most $d$} if there exists a filtration $G_{\bullet}$ of degree at most $d$ such that $g$ has coefficients in $G_{\bullet}$.
\end{definition}

\begin{proposition}[Lazard-Leibman theorem \cite{leibman-group-2}]\label{poly-precise}  Let $H$ be a group, let $G$ be a nilpotent group, and let $G_{\bullet}$ be a filtration. Then $\poly(H,G_{\bullet})$, the space of polynomial maps $g: H \to G$ having coefficients in $G_{\bullet}$, is a group.
\end{proposition}

\remarks  This result is contained in \cite{leibman-group-2} (although the result is only stated in the case that $G_{\bullet}$ is the lower central series filtration, the proof does not use this fact). Our proof is a little different, relying on the machinery of Host-Kra cube groups. These featured for the first time in \cite[\S 5, \S 11]{host-kra} and were discussed subsequently in \cite[Appendix E]{green-tao-linearprimes}. See also the recent preprint \cite{host-kra-linfty}. We thank Sasha Leibman for helpful conversations concerning these methods. 

One should mention at this point the Hall-Petresco theorem \cite{hall,petresco}, which established a special case of the Lazard-Leibman theorem. This theorem states that if $G_{\bullet}$ is the lower central series filtration then the sequence $n \mapsto a^n b^n$ lies in $\poly(\Z,G_{\bullet})$ for any $a,b \in G$.

In this section it is convenient to generalise the notion of a filtration somewhat. By a \emph{prefiltration} $G_{\bullet}$ on a nilpotent group $G$ we mean a sequence
\[ G \supseteq G_0 \supseteq G_1 \supseteq \dots \supseteq G_d \supseteq \{\id_G\}\] of subgroups with the property that $[G_i,G_j] \subseteq G_{i+j}$ for all $i,j \geq 0$. The only difference between a prefiltration and a filtration (cf. Definition \ref{nil-def}) is that we no longer require that $G = G_0 = G_1$. The definition of $\poly(H,G_{\bullet})$ extends in a completely obvious way to prefiltrations.

For each integer $k \geq 0$ we are going to define the \emph{Host-Kra cube group} $\HK^k(G_{\bullet})$ associated to the prefiltration $G_{\bullet}$. This will be a subgroup of $G^{\{0,1\}^k}$, the product of $2^k$ copies of $G$ indexed by the cube $\{0,1\}^k$. Before giving the definition, we need to set up some nomenclature concerning these cubes.

Each element $\omega \in \{0,1\}^k$ corresponds in an obvious way to a subset of $[k]$, and we write $\omega \subseteq \omega'$ when the corresponding sets are nested. An \emph{upper face} $F$ is a subset of $\{0,1\}^k$ of the form $F(\omega_0) := \{\omega \in \{0,1\}^k : \omega  \supseteq \omega_0\}$. There are, of course, $2^k$ upper faces, one for each $\omega_0 \in \{0,1\}^k$. The \emph{codimension} $\codim(F)$ of $F$ is simply the number of ones in $\omega_0$. Note that if $F,F'$ are two upper faces then $F \cap F'$ is also an upper face, and $\codim(F \cap F') \leq \codim(F) + \codim(F')$.

Given an upper face $F$ and an element $x \in G$ we write $x^F$ for the element of $G^{\{0,1\}^k}$ defined by
\[ (x^F)_{\omega} = \left\{ \begin{array}{ll}  x & \mbox{if $\omega \in F$} \\ \id_G & \mbox{otherwise}.  \end{array}\right.\]
Write $G_{(F)}$ for the subgroup of $G^{\{0,1\}^k}$ consisting of all elements $x^F$ with $x \in G_{\codim(F)}$, where $G_i$ is the $i$th group in the prefiltration $G_{\bullet}$; we call such a group an \emph{upper face group}.

\begin{definition}[Host-Kra cube group]\label{hk-cube-def} Let $G_{\bullet}$ be a prefiltration on a nilpotent group $G$, and let $k \geq 0$ be an integer. Then the \emph{Host-Kra cube group} $\HK^{k}(G_{\bullet})$ is the subgroup of $G^{\{0,1\}^k}$ generated by the upper face groups $G_{(F)}$.
\end{definition}

The Host-Kra cube group can, it turns out, be described in a rather explicit way. Write $\prec$ for the reverse lexicographic ordering on $\{0,1\}^k$, thus $\omega \prec \omega'$ if an only if there is some $j$ such that $\omega_j < \omega'_j$ and $\omega_i = \omega'_i$ for $i = j+1 ,\dots,k$. This induces an ordering on the upper faces $F$. We write $F(\omega) \succ F(\omega')$ if and only if $\omega \prec \omega'$. Let $F_0 \prec F_1 \prec \dots \prec F_{2^k-1}$ be the complete list of upper faces in this order; thus $F_0 = \{1^{k}\}$ and $F_{2^k-1} = \{0,1\}^{k}$.

\begin{lemma}[Description of Host-Kra cube group] \label{hk-cube-struct}
We have
\[ \HK^k(G_{\bullet}) = G_{(F_0)} \cdot G_{(F_1)} \cdot \dots G_{(F_{{2^k-1}})}.\]
That is, every element of $\HK^k(G)$ may be written as $\gamma^{F_0}_0 \dots \gamma^{F^{{2^k-1}}}_{2^k - 1}$ where $\gamma_i \in G_{\codim(F_i)}$. The representation is in fact unique.
\end{lemma}
\proof The key point here is the inclusion \begin{equation}\label{commutator}[G_{(F)}, G_{(F')}] \subseteq G_{(F \cap F')}.\end{equation} This follows immediately from the fact that \[ [G_{\codim(F)}, G_{\codim(F')}] \subseteq G_{\codim(F) + \codim(F')} \subseteq G_{\codim(F \cap F')}.\]
Using this fact repeatedly, we shift all elements coming from $G_{(F_0)}$ to the left. We then shift all elements coming from $G_{(F_1)}$ to the left, and so on. We leave the routine details and the proof that the representation is unique (which we do not actually need) to the reader.\endproof

\textsc{Host-Kra cube groups and polynomial maps.} It is now time to develop the link between Host-Kra cube groups $\HK^k(G_{\bullet})$ and polynomial maps $g \in \poly(H,G_{\bullet})$. To do this we introduce the notion of a \emph{parallelepiped} on $H$. This is an element in $H^{\{0,1\}^k}$ of the form $(xh^{\omega})_{\omega \in \{0,1\}^k}$, where $x \in H$, $h = (h_1,\dots,h_k)$ is a $k$-tuple of elements of $H$, and $h^{\omega} := h_1^{\omega_1} \dots h_k^{\omega_k}$. For example the tuple $(x,xh_1,xh_2,xh_1h_2)$ is a parallelepiped in $H^{\{0,1\}^2}$, and $(x,xh_1,xh_2,xh_1h_2,xh_3,xh_1h_2, xh_2h_3, xh_1h_2h_3)$ is a parallelepiped in $H^{\{0,1\}^3}$. Write $H^{[k]}$ for the set of parallelepipeds in $H^{\{0,1\}^k}$ (if $H$ is abelian $H^{[k]}$ is actually a group, but this need not be the case in general and in any case is not important here).

Suppose that $g : H \rightarrow G$ is a map. Then for any $k \geq 0$ there is an obvious induced map $g^{\{0,1\}^k} : H^{\{0,1\}^k} \rightarrow G^{\{0,1\}^k}$. 

\begin{proposition}[Characterization of polynomial maps]\label{poly-char}
Suppose that $H$ is a group, that $G$ is a nilpotent group together with a prefiltration $G_{\bullet}$, and that $g : H \rightarrow G$. Then $g$ lies in $\poly(H,G_{\bullet})$ if and only if $g^{\{0,1\}^k}$ maps $H^{[k]}$ to $\HK^k(G_{\bullet})$ for all $k \geq 0$.\end{proposition}

\emph{Remark.} The reader might find it useful, as an exercise to get to grips with the notation, to verify this in the case $H = G$ and $g$ being the identity mapping.

We note that Proposition \ref{poly-precise} is an immediate consequence of Proposition \ref{poly-char}. Indeed if $g^{\{0,1\}^k}$ and $\tilde{g}^{\{0,1\}^k}$ both map $H^{[k]}$ to $\HK^k(G_{\bullet})$ then so does $(g\tilde g)^{\{0,1\}^k}$, since $\HK^k(G_{\bullet})$ is a group.

\emph{Proof of Proposition \ref{poly-char}.} We start by establishing the \emph{only if} direction of the proposition, proving by induction on $k$ that $g^{\{0,1\}^k}$ does indeed map $H^{[k]}$ to $\HK^k(G_{\bullet})$ when $g \in \poly(H,G_{\bullet})$. This is clear when $k = 0$. Suppose it is known for a given value of $k \geq 0$. If $X$ is a set, we may regard $X^{\{0,1\}^{k+1}}$ as a product of two copies of $X^{\{0,1\}^k}$, the first factor corresponding to those $\omega$ with $\omega_{k+1} = 0$ and the second to those $\omega$ with $\omega_{k+1} = 1$. With this notation, every $\tilde z \in H^{[k+1]}$ may be written $\tilde z = (z,zh_{k+1})$, where $z :=  (xh^{\omega})_{\omega \in \{0,1\}^{k}}$. We may factor $g^{\{0,1\}^{k+1}}(\tilde z)$ as a product of two elements, namely
\begin{equation}\label{product} g^{\{0,1\}^{k+1}}(\tilde z) =  (\id_G^{\{0,1\}^k}, (\partial_{h_{k+1}}g)^{\{0,1\}^k}(z)) \cdot (g^{\{0,1\}^k}(z), g^{\{0,1\}^k}(z)).\end{equation} By the inductive hypothesis we have $g^{\{0,1\}^k}(z) \in \HK^k(G_{\bullet})$. The derivative $\partial_{h_{k+1}}g : H \rightarrow G$ is a polynomial map with coefficients in the prefiltration $\overleftarrow{G}_{\bullet}$ defined by $\overleftarrow{G}_i := G_{i+1}$ (note that this \emph{is} a prefiltration, since \[ [\overleftarrow{G}_i,\overleftarrow{G}_j] = [G_{i+1},G_{j+1}] \subseteq G_{i+j+2} \subseteq G_{i+j+1} = \overleftarrow{G}_{i+j}).\] By a second application of the inductive hypothesis we therefore have $(\partial_{h_{k+1}}g)^{\{0,1\}^k}(z) \in \HK^k(\overleftarrow{G}_{\bullet})$. In view of \eqref{product} it therefore suffices to show the inclusions 
\[ \HK^k(G_{\bullet})^{\Delta} \subseteq \HK^{k+1}(G_{\bullet})\] 
(where $\HK^{k}(G_{\bullet})^{\Delta}$ is the diagonal subgroup $\{ (t,t) : t \in \HK^k(G_{\bullet})\}$) and
\[ \id_G^{\{0,1\}^k} \times \HK^k(\overleftarrow{G}_{\bullet}) \subseteq \HK^{k+1}(G_{\bullet}).\]
To check the first inclusion it suffices to check elements $(\gamma^F,\gamma^F)$ where $\gamma \in G_{\codim(F)}$. But it is easy to see that $(\gamma^F,\gamma^F) = \gamma^{\tilde{F}}$ inside $G^{\{0,1\}^{k+1}}$, where the codimension of the face $\tilde F$ inside $\{0,1\}^{k+1}$ equals $\codim(F)$, and the inclusion follows. To check the second inclusion it suffices to check elements $(\id_G^{\{0,1\}^k},\gamma_F)$ where $\gamma \in \overleftarrow{G}_{\codim(F)} = G_{\codim(F) + 1}$. But it is again easy to see that $(\id_G^{\{0,1\}^k},\gamma^F) = \gamma^{\tilde F}$, where now the codimension of $\tilde F$ inside $\{0,1\}^{k+1}$ is $\codim(F) +1$. This concludes the proof of the \emph{only if} part of Proposition \ref{poly-char}; the perceptive reader will have noticed that we have not yet made any essential use of the main property of prefiltrations, namely the nesting property that $[G_i,G_j] \subseteq G_{i+j}$.

We turn now to the proof of the \emph{if} direction of the proposition. We are to show that if $g^{\{0,1\}^k}$ maps $H^{[k]}$ to $\HK^k(G_{\bullet})$ for all $k$, then $g \in \poly(H,G_{\bullet})$. Pick an element $z = (xh^{\omega})_{\omega \in \{0,1\}^{k}}$ in $H^{[k]}$. By Lemma \ref{hk-cube-struct} (which does use the nesting property of $G_{\bullet}$) we may write 
\[ g^{\{0,1\}^k}(z) = \gamma^{F_0}_0 \dots \gamma^{F_{2^k-1}}_{2^k-1}\] where $\gamma_i \in G_{\codim(F_i)}$. Write $\eta_j := \gamma^{F_{2^{j-1}}}_{2^{j-1}} \dots \gamma^{F_{2^j-1}}_{2^{j}-1}$, $j = 1,\dots,k$, so that
\begin{equation}\label{equality} g^{\{0,1\}^k}(z) = \gamma^{1^k}_0 \eta_1 \dots \eta_k.\end{equation}
One may check that the $\eta_i$ enjoy the following support properties: $(\eta_i)_{\omega} = \id_G$ unless $\omega_{i+1},\dots,\omega_k$ are all nonzero, and $(\eta_i)_{\omega} = (\eta_i)_{\omega'}$ if $\omega,\omega'$ differ only in the $\omega_i$ coordinate. One may now examine \eqref{equality} coordinatewise, peeling off $\eta_k,\eta_{k-1},\dots$ in turn, to eventually conclude that 
\[ \gamma_0 = \partial_{h_1} \dots \partial_{h_k}g(x).\]
Now we know that $\gamma_0 \in G_{\codim(F_0)} = G_k$, and thus we have proved that $\partial_{h_1} \dots \partial_{h_k}g$ takes values in $G_k$, as required.\endproof

\textsc{Polynomial maps in coordinates.} From now on we specialise to the case of polynomial maps from $\Z^t$ to $G$ and revert to dealing with filtrations as opposed to prefiltrations. Our aim in this section is to describe the elements of $\poly(\Z^t,G_{\bullet})$ using the Mal'cev coordinate map $\psi : G \rightarrow \R^m$ relative to some Mal'cev basis $\X$ for $G/\Gamma$ adapted to the filtration $G_{\bullet}$.

\begin{definition}[Multi-binomial coefficients]
Let $t \geq 1$ be an integer. Suppose that $\vec{n} = (n_1,\dots,n_t)$ and that $\vec{j} = (j_1,\dots,j_t) \in \Z_{\geq 0}^t$ is a set of indices. Then we write 
\[ \binom{\vec{n}}{\vec{j}} := \prod_{i=1}^t \binom{n_i}{j_i}.\]
\end{definition}

A version of the following lemma may be found in \cite[\S 4]{leibman-orb-diag}.

\begin{lemma}[Description of $\poly(\Z^t,G_{\bullet})$ in bases]\label{basis-description}
Suppose that $G/\Gamma$ is a nilmanifold of dimension $m$ and that $\X$ is a Mal'cev basis for $G/\Gamma$ adapted to some filtration $G_{\bullet}$. Then $g \in \poly(\Z^t,G_{\bullet})$ if and only if the coordinates $\psi(g(\vec{n}))$ have the form
\[ \psi(g(\vec{n})) = \sum_{\vec{j}}t_{\vec{j}} \binom{\vec{n}}{\vec{j}},\]
where each $t_{\vec{j}}$ lies in $\R^m$ and is such that $(t_{\vec{j}})_i = 0$ if $i \leq m - m_{|\vec{j}|}$, where $|\vec{j}| := j_1 + \dots + j_t$. 
\end{lemma}

\remark The presence of the discrete subgroup $\Gamma$ is not at all relevant to this lemma; however we have only defined Mal'cev bases in this context.

\proof We start with the \emph{if} direction. If $g(n)$ has the form stated then it is a product of sequences of the form $\vec{n} \mapsto a^{\binom{\vec{n}}{\vec{j}}}$, where $a \in G_{|\vec{j}|}$. By the group property of $\poly(\Z^t,G_{\bullet})$ it therefore suffices to establish the result in the case that $g(\vec{n})$ is actually \emph{equal} to such a sequence. By induction one sees that the derivative $\partial_{h_1} \dots \partial_{h_k} g(\vec{n})$ equals $a^{p(h_1,\dots,h_k;\vec{n})}$, where the maximal degree $\alpha_1 + \dots + \alpha_t$ of a monomial $n_1^{\alpha_1} \dots n_t^{\alpha_t}$ appearing in $p$ is at most $\max(|\vec{j}| - k,0)$. Thus we see that this derivative lies in $G_{|\vec{j}|}$ if $k \leq |\vec{j}|$, and is zero otherwise. It follows that $g \in \poly(\Z^t,G_{\bullet})$.

To prove the \emph{only if} direction, let $\h_j \subset \g$ be the subspace
\[ \h_j := \Span( X_{j+1},\dots,X_m)\] and set $H_j := \exp(\h_j)$. By the nesting property of the Mal'cev basis $\X$ (see \eqref{nest}) we see that $H_j \lhd G$.

Suppose as a hypothesis for downward induction on $k$ that the statement has been proved for all $g \in \poly(\Z^t,G_{\bullet})$ with $g(\vec{n}) \in H_k$ for all $\vec{n}$, for a certain value of $k$. This is trivial for $k = m$, in which case $g(\vec{n}) = \id_G$. Suppose that $g(\vec{n}) \in H_{k-1}$ for all $\vec{n}$. Let $\pi : H_{k-1} \rightarrow H_{k-1}/H_k \cong \R$ be the natural projection. Then $p_{k-1}(\vec{n}) := \pi(g(\vec{n})\Gamma)$ is a polynomial map from $\R^t$ to $\R$. Suppose that $k-1 < m-m_i$, and that $i$ is minimal subject to this property. Then for any $h_1,\dots,h_i \in \Z^t$ we have $\partial_{h_1}\dots \partial_{h_i} g \in G_{i} = H_{m-m_i}$, and therefore $\partial_{h_1}\dots \partial_{h_i} p_{k-1}(\vec{n}) = 0$. Thus the total degree of any monomial in $p_{k-1}$ is at most $i-1$.  Therefore we may write the sequence $h(\vec{n})$ defined by
\[ h(\vec{n}) := \exp(X_{k-1})^{p_{k-1}(\vec{n})}\] as a product of sequences $\exp(X_{k-1})^{t_{\vec{j}}\binom{\vec{n}}{\vec{j}}}$ with $|\vec{j}| \leq i-1$. By the minimality of $i$ we have $X_{k-1} \in \g_{i-1}$, and so each of these sequences lies in $\poly(\Z^t,G_{\bullet})$, and hence so does $h$. It follows that the sequence $\tilde g(n) := g(n) h(n)^{-1}$ lies in $\poly(\Z^t,G_{\bullet})$. But this new sequence $\tilde g$ has $\tilde g(n) \in H_k$, and hence we may proceed by induction.\endproof

A useful and easily-derived corollary of Lemma \ref{basis-description} is that $\poly(\Z^t,G_{\bullet})$ is closed under dilations.

\begin{corollary}[Dilation of polynomial sequences]\label{poly-restrict}  Suppose that $g \in \poly(\Z^t,G_{\bullet})$ and that $a_1,\dots,a_t,b_1,\dots,b_t \in \Z$. Then the sequence $\vec{n} \mapsto g(a_1 + b_1n_1,\dots,a_t + b_t n_t)$ also lies in $G_{\bullet}$.\endproof
\end{corollary}

We remarked in the introduction that a sequence $g : \Z \rightarrow G$ is polynomial with coefficients in some filtration $G_{\bullet}$ if and only if $g$ has the form \begin{equation}\label{old-poly} g(n) = a_1^{p_1(n)} \dots a_k^{p_k(n)}\end{equation} for polynomials $p_1,\dots,p_k$ with integer coefficients. Although this result is not required in the paper it is certainly conceivable that one might wish to apply the main theorems of the paper to a sequence which is presented in an explicit form such as \eqref{old-poly}, and does not obviously satisfy the more abstract condition of Definition \ref{poly-def}. 

The fact that every polynomial sequence has the form \eqref{old-poly} is an easy consequence of Lemma \ref{basis-description}. To establish the converse, consider first the lower central series filtration $G_{\bullet}$ which has degree $s$, the step of the nilpotent Lie group $G$. Let $d$ be the maximum degree occurring amongst the polynomials $p_i$ and define a finer filtration $G'_{\bullet}$ of degree $sd$ by setting $G'_{i} := G_{\lceil i/d\rceil}$. This \emph{is} a filtration since \[ [G'_{i},G'_{j}] = [G_{\lceil i/d\rceil},G_{\lceil j/d\rceil}] \subseteq G_{\lceil i/d\rceil + \lceil j/d\rceil} \subseteq G_{\lceil (i+j)/d\rceil} = G'_{i+j}.\] Any sequence of the form $n \mapsto a^{\binom{n}{j}}$, $j \leq d$ has coefficients in $G'_{\bullet}$ since $G'_{i} = G$ for $i = 0,1,\dots,d$ and the $(d+1)$st derivative of such a sequence is trivial. Since $g$ is a product of such sequences and $\poly(\Z,G'_{\bullet})$ is a group we see that $g \in \poly(\Z,G'_{\bullet})$.

We note that if $G/\Gamma$ has a $Q$-rational Mal'cev basis adapted to the lower central series then, by the results of the appendix, there is a $Q^{O_{d,s}(1)}$-rational Mal'cev basis for $G/\Gamma$ adapted to $G'_{\bullet}$.

We leave it to the reader to formulate and prove an analogous result for polynomial mappings from $\Z^t$ to $G$.

\section{The general case of the main theorem}\label{sec6}

We are now in a position to attack the general case of Theorem \ref{main-theorem}. Our analysis of the Heisenberg example in \S \ref{heisenberg-sec} suggested that the argument will involve an induction on the degree $d$ of $G_{\bullet}$. In that case there were two different scenarios in which we reduced from the case $d = 2$ to the case $d = 1$. Whilst the same is true in general, the introduction of genuinely polynomial sequences (rather than just linear ones) necessitates a further inductive loop on the quantity $m_* := m_{\ab} - m_{\lin}$, which we call the \emph{nonlinearity degree}. To see why, consider the following slightly informal example.

\example \label{example5}
Let $G/\Gamma$ be the Heisenberg example, and let $g(n) = \left(\begin{smallmatrix} 1 & \alpha_1 & \alpha_3  \\ 0 & 1 & \alpha_2  \\ 0 & 0 & 1\end{smallmatrix}\right)^n$, where $\alpha_1,\alpha_2$ and $\alpha_3$ are highly independent over $\Q$. Then there is no horizontal character $\eta$ of low frequency such that $\Vert \eta \circ g \Vert_{C^{\infty}[N]}$ is small.

Now we have $\partial g = g$ and $\partial^i g = \id_G$ for $i \geq 2$, and so $g$ has coefficients in the subgroup sequence $G_{\bullet}$ defined by $G_{(0)} := G_{(1)} := G_{(2)} := G$, $G_{(3)} := G_{(4)} := G_2$, and $G_{(i)} := \{\id_G\}$ for $i \geq 5$. With this choice we have $G^{\Box} = G \times G$. However $g_h^{\Box}$ takes values in $G \times_{G_2} G$, and hence 
$\eta^{\Box} \circ g_h^\Box = 0$ for any horizontal character $\eta^\Box$ with frequency of the form $(a,b,-a,-b) \in \Z^4$.  Thus, a lack of uniform distribution for $g_h^\Box$ does not imply lack of uniform distribution for $g$.
\hfill 

The problem in the above example is that the filtration $G_{\bullet}$ was far too ``coarse'' to accurately capture the differential structure of the sequence $g$. Indeed $g$ also takes values in the minimal (lower central series) filtration, as we saw in \S \ref{heisenberg-sec}.

In the light of the above example we can expect that it will sometimes be necessary to pass to a ``finer'' filtration of the same degree $d$, in order to properly capture the differential structure of $g$. This finer filtration will have a smaller value of the nonlinearity degree $m_*$, and thus we introduce an extra inductive loop to incorporate this parameter.
To be precise we shall prove, by induction on $d$ and $m_*$, the following slight variant of Theorem \ref{main-theorem}.

\begin{theorem}[Variant of Main Theorem]\label{mainthm-ind}
Let $m,d \geq 0$ be integers with $m_* \leq m$. Let $0 < \delta < 1/2$ and suppose that $N \geq 1$. Suppose that $G/\Gamma$ is a nilmanifold and that $G_{\bullet}$ is a filtration of degree $d$ and with nonlinearity degree $m_*$. Suppose that $\mathcal{X}$ is a $1/\delta$-rational Mal'cev basis adapted to $G_{\bullet}$ and suppose that $g \in \poly(\Z,G_{\bullet})$. If $(g(n)\Gamma)_{n \in [N]}$ is not $\delta$-equidistributed then there is a horizontal character $\eta$ with $0 < |\eta| \ll \delta^{-O_{m,m_*,d}(1)}$ such that 
\[ \Vert \eta \circ g \Vert_{C^{\infty}[N]} \ll \delta^{-O_{m,m_*,d}(1)}.\] 
\end{theorem}

It is clear that this does imply Theorem \ref{main-theorem}, since the dependence of the $O(1)$ exponents on $m_*$ may be suppressed once Theorem \ref{mainthm-ind} has been proven by induction. In our proof there will be an outer inductive loop over $d$ and an inner one over $m_*$. In other words we shall assume that Theorem \ref{mainthm-ind} holds for all pairs $(d',m'_*)$ in which either $d' < d$ or for which $d' = d$ and $m'_* < m_*$, and deduce the case $(d,m_*)$. 

Henceforth we allow all constants implicit in the $\ll$ or $O$-notation to depend on $d,m$ and $m_*$.

We begin with some simple reductions. By Lemma \ref{vert-freq-red} we may assume that the orbit $(g(n)\Gamma)_{n \in [N]}$ is not $\delta^{O(1)}$-equidistributed along some vertical frequency $\xi \in \Z^{m_d}$ with $|\xi| \ll \delta^{-O(1)}$. Thus there is some function $F : G/\Gamma \rightarrow \C$ with $\Vert F \Vert_{\Lip} \leq 1$ and vertical frequency $\xi$ such that 
\begin{equation}\label{non-equi} |\E_{n \in [N]} F(g(n)\Gamma) - \int_{G/\Gamma} F| \gg \delta^{O(1)}.\end{equation}
If $\xi = 0$ then $F$ is $G_d$-invariant and we may descend to $G/G_d$, together with the filtration $G_{\bullet}/G_d$ which has length $d-1$, and invoke our inductive hypothesis. We pause to give the rather straightforward details.

Write $\overline{G} := G/G_d$ and $\overline{\Gamma} := \Gamma/(\Gamma \cap G_d)$. Then $\overline{G}/\overline{\Gamma}$ is a nilmanifold togther with a filtration $\overline{G}_{\bullet}$ of length $d-1$, where $\overline{G}_i := G_i/G_d$. The Mal'cev basis $\X = \{X_1,\dots,X_m\}$ may be reduced to give a $\frac{1}{\delta}$-rational Mal'cev basis $\overline{\X} = \{\overline{X}_1,\dots,\overline{X}_{\overline{m}}\}$ for $\overline{G}/\overline{\Gamma}$ adapted to $\overline{G}_{\bullet}$, where $\overline{m} := m - m_d$. 

Write $\overline{g} : \Z \rightarrow \overline{G}$ for the reduction of $g \md{G_d}$
By the $G_d$-invariance the function $F$ descends to a Lipschitz function $\overline{F} : \overline{G}/\overline{\Gamma} \rightarrow \C$ with $\Vert \overline{F} \Vert_{\Lip} \leq \Vert F \Vert_{\Lip}$, and so \eqref{non-equi} implies that
\[ \left|\E_{n \in [N]} \overline{F}(\overline{g}(n)\overline{\Gamma}) - \int_{\overline{G}/\overline{\Gamma}} \overline{F}\right| \geq \delta \Vert \overline{F} \Vert_{\Lip}.\]
(Here we have used the fact that normalised Haar measure on $\overline{G}/\overline{\Gamma}$ is obtained by quotienting that on $G/\Gamma$ by $G_d$.)

We may now apply the inductive hypothesis to obtain a horizontal character $\overline{\eta}: \overline{G} \to \C$ on $\overline{G}$ of frequency magnitude $0 < |\overline \eta| \ll \delta^{-O(1)}$ such that
\[ \Vert \overline \eta \circ \overline{g} \Vert_{C^{\infty}[N]} \ll \delta^{-O(1)}.\]
If we let $\eta: G \to \C$ be the horizontal character on $G$ defined by $\eta(x) = \overline{\eta}(\overline{x})$ then we have $\overline{\eta} \circ \overline{g} = \eta \circ g$ and $|\eta| = |\overline{\eta}|$. This concludes the proof in the case $\xi = 0$.

Suppose henceforth that $\xi \neq 0$. Since $F$ has $\xi$ as a vertical frequency, \eqref{non-equi} becomes
\begin{equation}\label{non-equi-2} |\E_{n \in [N]} F(g(n)\Gamma)| \gg \delta^{O(1)}.\end{equation}
We proceed initially with two additional reductions. The first is to the case $g(0) = \id_G$. Factorize $g(0) = \{g(0)\}[g(0)]$ as in Lemma \ref{fund-dom-def}. Set $\tilde g(n) := \{g(0)\}^{-1}g(n)g(0)^{-1}\{g(0)\}$. Then we have $|\E_{n \in [N]} \tilde F(\tilde g(n)\Gamma)| \geq \delta$, where $\tilde F(x) := F(\{g(0)\}x)$. But $\tilde F$ still has vertical oscillation $\xi$ and, by Lemma \ref{approx-left}, it has Lipschitz constant $O(1)$. Noting that $\Vert \eta \circ g\Vert_{C^{\infty}[N]} = \Vert \eta \circ \tilde g \Vert_{C^{\infty}[N]}$ we see that if we have Theorem \ref{mainthm-ind} for $\tilde g$ then we also have it for $g$.

The second reduction is to the case when $|\psi(g(1))| \leq 1$ (this is needed in the lead up to \eqref{to-work-with-2}). To do this, factorize $g(1) = \{g(1)\}[g(1)]$ as in Lemma \ref{fund-dom-def}. Set $\tilde g(n) := g(n)[g(1)]^{-n}$. Then $\tilde g(n)\Gamma = g(n)\Gamma$, $\tilde g(0) = \id_G$, $\tilde g \in \poly(\Z,G_{\bullet})$ and $\pi(\tilde g(n)\Gamma) = \pi(g(n)\Gamma)$, so proving Theorem \ref{mainthm-ind} for $g$ is equivalent to proving it for $\tilde g$. 

Henceforth we assume $g(0) = \id_G$ and $|\psi(g(1))| \leq 1$. 

As in \S \ref{heisenberg-sec} we apply Van der Corput's Lemma (Corollary \ref{vdc}) to \eqref{non-equi-2} to deduce that for $\gg \delta^{O(1)}N$ values of $h$, we have
\begin{equation}\label{eq801} |\E_{n \in [N]} F(g(n+h)\Gamma) \overline{F(g(n) \Gamma)}| \gg \delta^{O(1)}.\end{equation}
For each fixed $h$ this may be interpreted as a statement about the polynomial sequence $(g(n+h),g(n))$ on the product group $G^2$. However, guided by our experience with the Heisenberg group, it is natural to try and interpret it as a sequence on a somewhat smaller group. To this end, we define the \emph{nonlinear part} $g_{\nonlin}$ of $g$ by
\begin{equation}\label{g2-def}
g_{\nonlin}(n) := g(n) g(1)^{-n}.
\end{equation}
Motivated by what we did in \S \ref{heisenberg-sec}, we may then rewrite \eqref{eq801} in the form
\begin{equation}\label{use-soon-1} |\E_{n \in [N]} \tilde F_h(\tilde g_h(n)\Gamma^2) | \gg \delta^{O(1)},\end{equation}
where
\[ \tilde F_h(x,y) := F(\{g(1)^h\}x) \overline{F(y)}\] and
\begin{equation}\label{gh-def} \tilde g_h(n) := (\{g(1)^h\}^{-1} g_{\nonlin}(n+h) g(1)^n \{g(1)^h\}, g_{\nonlin}(n)g(1)^n).\end{equation}
It turns out that $g_h$ takes values in $G^{\Box} := G \times_{G_2} G$, just as we found in our analysis of the Heisenberg case. To prove this note that have $G_2 \supseteq [G,G]$, and so $G$ becomes abelian after quotienting out by the normal subgroup $G_2$. Thus we need only prove that $g_{\nonlin}(n) \in G_2$ for all $n$. We have $\partial^2 g(n) = \id_G$ modulo $G_2$.  Since $g(0)=\id_G$, this implies by an easy induction that $g(n) = g(1)^n$ modulo $G_2$, and so $g_{\nonlin}$ does indeed take values in $G_2$.

We may therefore replace \eqref{use-soon-1} by 
\begin{equation}\label{use-soon-1point1}
| \E_{n \in [N]} F^{\Box}_h (g^{\Box}_h(n)\Gamma^{\Box})| \gg \delta^{O(1)}
\end{equation}
by restricting everything in that equation to an object on $G^{\Box}$.

Note that, exactly as in the Heisenberg case, $F_h^{\Box}$ is invariant under $G_d^{\Delta} = \{(g_d,g_d) : g_d \in G_d\}$. Indeed, since $G_d$ is central in $G$, we have
\begin{align*} F_h^{\Box}((g_d,g_d) \cdot x^{\Box}) & = F(\{g(1)^h\} g_d x) \overline{F(g_d x')} \\ & = e(\xi(g_d))e(-\xi(g'_d))F(\{g(1)^h\} x) \overline{F(x')} \\ & = F_h^{\Box}(x^{\Box}).\end{align*}
Thus $F_h^{\Box}$ descends to a function $\overline{F^{\Box}_h}$ on $\overline{G^{\Box}} := G^{\Box}/G_d^{\Delta}$ and we may write \eqref{use-soon-1point1} as
\begin{equation}\label{use-soon-1point2}
| \E_{n \in [N]} \overline{F^{\Box}_h} (\overline{g^{\Box}_h}(n)\overline{\Gamma^{\Box}})| \gg \delta^{O(1)},
\end{equation}
where $\overline{\Gamma^{\Box}} := \Gamma^{\Box}/(\Gamma \cap G_d^{\Delta})$.

The next proposition is central to our whole argument in that it clarifies the sense in which $\overline{G^{\Box}}$ is ``less complex'' than $G$.

\begin{proposition}[Reduction in degree]
Define $(G^{\Box})_i := G_i \times_{G_{i+1}} G_i$ for $i = 1,\dots,d$. Then $(G^{\Box})_{\bullet}$ is a filtration on $G^{\Box}$ of degree $d$. Since $(G^{\Box})_d = G_d^{\Delta}$, it descends under quotienting by $G_d^{\Delta}$ to a filtration $\overline{(G^{\Box})}_{\bullet}$ of degree $d-1$ on $\overline{G^{\Box}}$. Each polynomial sequence $g_h^{\Box}$ lies in $\poly(\Z,(G^{\Box})_{\bullet})$, and hence each reduced polynomial sequence $\overline{g_h^{\Box}}$ lies in $\poly(\Z,(\overline{G^{\Box}})_{\bullet})$.
\end{proposition}
\proof We start with a lemma.
\begin{lemma}
Suppose that $H_1,H_2$ and $K_1, K_2$ are normal subgroups of a group $G$, that $H_1, H_2$ generate a group $H$ and that $K_1,K_2$ generate a group $K$. Then $[H,K]$ is generated by the groups $[H_i,K_j]$, $1 \leq i,j \leq 2$.
\end{lemma}
\proof The groups $[H_i,K_j]$ are all normal, and thus the group they generate is also normal.  If we quotient by that group, then $H_1, H_2$ commute with $K_1, K_2$, and thus $H$ commutes with $K$.  The claim follows.\endproof

Now observe that $(G^{\Box})_i$ is generated by $G_{i+1}^2$ and $G_{i}^{\triangle}$. In view of the lemma it therefore suffices to establish that all four of the quantities
\[ [G^{\triangle}_{i}, G^{\triangle}_{j}], [G^{\triangle}_{i}, G_{j+1}^2], [G_{i+1}^2, G^{\triangle}_{j}], [G_{i+1}^2, G_{j+1}^2]\]
lie in $G_{i+j}^{\Box}$. Using the fact that $G_{\bullet}$ is a filtration, the first quantity is manifestly contained in $G^{\triangle}_{i+j}$ and the last three lie in $G_{i+j+1}^2$. It follows immediately that $(G^{\Box})_{\bullet}$ is indeed a filtration.

Next we show that $g_h^{\Box} \in \poly(\Z,(G^{\Box})_{\bullet})$. Here we make serious use of the fact that $\poly(\Z,G_{\bullet})$ is a group for the first time. Recall that 
\begin{equation}\label{gh-box-def} g_h^{\Box}(n) := \left(\{g(1)^h\}^{-1} g_{\nonlin}(n+h) g(1)^n \{g(1)^h\}, g_{\nonlin}(n)g(1)^n\right).\end{equation}
Now $\poly(\Z,(G^{\Box})_{\bullet})$ is a group, and it is also closed under conjugation by elements of $G^2$. Since $(g(1)^n,g(1)^n)$ is obviously in $\poly(\Z,(G^{\Box})_{\bullet})$, it suffices to check that $(g_2(n+h), g_2(n)) \in \poly(\Z,(G^{\Box})_{\bullet})$. Of course, $g_2 \in \poly(\Z,G_{\bullet})$ and hence, by Lemma \ref{basis-description}, it is a product of elements $g_i^{\binom{n}{i}}$ with $g_i \in G_{i}$. It therefore suffices to show that $(g_i^{\binom{n+h}{i}}, g_i^{\binom{n}{i}}) \in (G^{\Box})_{\bullet}$. Taking $j$th derivatives, it suffices to check that $g_i^{\binom{n+h}{i-j}} \equiv g_i^{\binom{n}{i-j}} \md{G_{j+1}}$. For $j < i$ this follows from the fact that $g_i \in G_{i}$, whilst for $j \geq i$ it is trivial.\endproof

In order to apply the inductive hypothesis, we must specify a Mal'cev basis $\overline{\mathcal{X}^{\Box}}$ for $\overline{G^{\Box}}/\overline{\Gamma^{\Box}}$ adapted to the sequence $\overline{(G^{\Box})}_{\bullet}$, and it must then be checked that $\overline{F_h^{\Box}}$ is Lipschitz with respect to the metric $d_{\overline{\mathcal{X}^{\Box}}}$. These are rather tedious matters and we recommend that the reader take the following lemma on trust on a first reading of the paper.

\begin{lemma}[Rationality bounds for the relative square]\label{rat-bounds}
There is an $O(\delta^{-O(1)})$-rational Mal'cev basis $\mathcal{X}^{\Box} = \{X^{\Box}_1,\dots,X^{\Box}_{m^{\Box}}\}$ for $G^{\Box}/\Gamma^{\Box}$ adapted to the filtration $(G^{\Box})_{\bullet}$ with the property that $\psi_{\X^{\Box}}(x,x')$ is a polynomial of degree $O(1)$ with rational coefficients of height $\delta^{-O(1)}$ in the coordinates $\psi(x), \psi(x')$. With respect to the metric $d_{\mathcal{X}^{\Box}}$ we have $\Vert F^{\Box}_h \Vert_{\Lip} \ll \delta^{-O(1)}$ uniformly in $h$.
\end{lemma}
\proof We consider $G^{\Box}$ as a subgroup of $G \times G$. Recall (cf. Definition \ref{weak-basis-def}) the definition of a \emph{weak basis}. It is clear that $\X \times \X = \{(X_1,0),(0,X_1),\dots,(X_m,0),(0,X_m)\}$ is a $\delta^{-O(1)}$-rational weak basis for $G/\Gamma \times G/\Gamma$ and that each of the groups $(G^{\Box})_i := G_i \times_{G_{i+1}} G_i$ is $\delta^{-O(1)}$-rational with respect to this basis. By Proposition \ref{sub-nil-basis} it follows that there is a Mal'cev basis $\X^{\Box} = \{X^{\Box}_1,\dots,X^{\Box}_{m^{\Box}}\}$ for $G^{\Box}/\Gamma^{\Box}$, adapted to the filtration $(G^{\Box})_{\bullet}$, with the property that each $X^{\Box}_i$ is a $\delta^{-O(1)}$-rational combination of the elements of $\X \times \X$. By adding the elements $(X_1,0),\dots,(X_{m_{\lin}},0)$ to $\X^{\Box}$ we obtain a weak basis $\mathcal{Y}$ for $G/\Gamma \times G/\Gamma$ which enjoys the nesting property \eqref{nest}. From Lemma \ref{1-2-lem} it follows that each coordinate of $\psi_{\mathcal{Y}}(x,x')$ is a polynomial of degree $O(1)$ and with coefficients $\delta^{-O(1)}$ in the coordinates $\psi_{\X \times \X}(x,x')$. Restricting to those pairs $(x,x')$ which lie in $G^{\Box}$, we obtain the stated property.

Recall that $F_h^{\Box}(x^{\Box}) = F(\{g(1)^h\}x) \overline{F(x')}$. Now by definition we have $|\psi_{\X}(\{g(1)^h)\}| \leq 1$. By Lemma \ref{approx-left} (and Lemma \ref{fund-dom-def}, which guarantees that every $x \in G/\Gamma$ has a representative with coordinates bounded by $O(1)$) we see that $(x,x') \mapsto F(\{g(1)^h\} x) \overline{F(x')}$ defines a function on $G \times G$ whose Lipschitz constant with respect to the product metric $d \times d$ is $\ll \delta^{-O(1)}$. Now by Lemma \ref{comparison-lemma} and the construction of $\X^{\Box}$ we therefore have $\Vert F_h^{\Box} \Vert_{\Lip} \ll \delta^{-O(1)}$ where, remember, the Lipschitz constant is being computed with respect to the metric $d_{\X^{\Box}}$.\endproof

Let us now resume the discussion starting from \eqref{use-soon-1point2}. We begin by reprising some of the straightforward arguments at the start of the section (where we dealt with the case $\xi = 0$). By reducing the first $\overline{m^{\Box}} := m^{\Box} - m_d$ elements of $\mathcal{X}^{\Box}$ we obtain an $O(\delta^{-O(1)})$-rational Mal'cev basis $\overline{\mathcal{X}^{\Box}} = \{\overline{X^{\Box}_1},\dots,\overline{X^{\Box}_{\overline{m^{\Box}}}}\}$ for $\overline{G^{\Box}}/\overline{\Gamma^{\Box}}$ adapted to the filtration $\overline{(G^{\Box})}_{\bullet}$. With respect to the metric $d_{\overline{\mathcal{X}^{\Box}}}$ we have $\Vert \overline{F_h^{\Box}} \Vert_{\Lip} \ll \delta^{-O(1)}$.

Since $\overline{(G^{\Box})}_{\bullet}$ has degree $d-1$ our inductive hypothesis is applicable and we conclude that for $\gg \delta^{O(1)}$ values of $h \in [N]$ there is some horizontal character $\overline{\eta}_h : \overline{G^{\Box}} \rightarrow \R/\Z$ with $0 < |\overline{\eta}_h| \ll \delta^{-O_m(1)}$ and
\[ \Vert \overline{\eta}_h \circ \overline{g^{\Box}_h} \Vert_{C^{\infty}[N]} \ll \delta^{-O_m(1)}.\] By pigeonholing in $h$ we may assume that $\overline{\eta} = \overline{\eta_h}$ is independent of $h$. Writing $\eta : G^{\Box} \rightarrow \R/\Z$ for the horizontal character defined by $\eta(x) = \overline{\eta}(\overline{x})$, we see that $0 < |\eta| \ll \delta^{-O_m(1)}$ and that 
\begin{equation}\label{h-biases-2} \Vert \eta \circ g_h^{\Box} \Vert_{C^{\infty}[N]} \ll \delta^{-O_m(1)}.\end{equation}

The next lemma, which is almost identical to Lemma \ref{g-box-struct}, allows us to write $\eta$ in terms of maps defined on $G$ rather than $G^{\Box}$.

\begin{lemma}\label{eta-decomp} We have a decomposition $\eta(g', g) = \eta_1(g) + \eta_2(g' g^{-1})$ for all $(g',g) \in G^\Box$, where $\eta_1 : G \rightarrow \R/\Z$ is a horizontal character on $G$, and $\eta_2 : G_{2} \rightarrow \R/\Z$ is a horizontal character on $G_{2}$ which also annihilates $[G,G_{2}]$. Furthermore we have $| \eta_1 |, |\eta_2 | \ll \delta^{-O(1)}$.
\end{lemma}
\proof If we define $\eta_1(g) := \eta(g,g)$ and $\eta_2(g_2) := \eta(g_2,\id_G)$ for $g \in G$ and $g_2 \in G_{2}$ then the decomposition follows since $\eta$ is an additive homomorphism.  Since $\eta$ annihilates $[G^\Box,G^\Box]$, which contains $[G^\Delta,G_{2} \times \id_G] = [G,G_{2}] \times \id_G$, we see that $\eta_2$ annihilates $[G,G_{2}]$; since $\eta$ annihilates $\Gamma^\Box$, which contains both $\Gamma^\Delta$ and $(\Gamma \cap G_{2}) \times \id_G$, we see that $\eta_1$ and $\eta_2$ annihilate $\Gamma$ and $\Gamma \cap G_{2}$ respectively.

It remains to check the boundedness properties. Writing 
\[ \eta(x,x') = k^{\Box} \cdot \psi_{\mathcal{X}^{\Box}}(x,x'),\] where $k^{\Box} \in \Z^{m_{\Box}}$, we have by definition that $|k^{\Box}| \ll \delta^{-O(1)}$. The integer vectors $k_1$ and $k_2$ used to define $|\eta_1|$ and $|\eta_2|$ are then given by
\[ k_1 \cdot \psi(x) = \eta_1(x) = \eta(x,x) = k^{\Box} \cdot \psi_{\mathcal{X}^{\Box}}(x,x)\]
and 
\[ k_2 \cdot \psi(x) = \eta_2(x) = \eta(x,\id_G) = k^{\Box} \cdot \psi_{\mathcal{X}^{\Box}}(x,\id_G).\]
That $|k_1|,|k_2| \ll \delta^{-O(1)}$ now follows immediately from the fact, established in Lemma \ref{rat-bounds}, that $\psi_{\mathcal{X}^{\Box}}(x,x')$ is a polynomial of degree $O(1)$ with rational coefficients of height $O(\delta^{-O(1)})$ in the coordinates $\psi(x),\psi(x')$.
\endproof

Now let us return to \eqref{h-biases-2}, and reinterpret this in terms of the decomposition of $\eta$ just given.
Recalling the formula \eqref{gh-box-def} for $g_h^{\Box}(n)$ we therefore have
\[ \eta(g_h^{\Box}(n)) = \eta_1(g(n)) + \eta_2(\{g(1)^h\}^{-1} g_{\nonlin}(n+h) g(1)^n \{g(1)^h\} g(1)^{-n} g_{\nonlin}(n)^{-1})\] which, since $\eta_2$ vanishes on $[G,G_{2}]$, is equal to
\begin{align*}  &  \eta_1(g(n)) + \eta_2(g_{\nonlin}(n+h) \{g(1)^h\}^{-1}  g(1)^n \{g(1)^h\} g(1)^{-n} g_{\nonlin}(n)^{-1})\\  = & \eta_1(g(n)) + \eta_2(g_{\nonlin}(n+h)) - \eta_2(g_{\nonlin}(n)) + \eta_2( \{g(1)^h\}^{-1}  g(1)^n \{g(1)^h\} g(1)^{-n}).\end{align*} 
Now one easily verifies by induction on $n$ that $y^{-1}x^n y x^{-n} \equiv [x,y]^n \md{[G,[G,G]]}$. Since $\eta_2$ annihilates $[G,G_{2}]$, which contains $[G,[G,G]]$, we can therefore simplify the above a little further to
\begin{align}\nonumber \eta(g_h^{\Box}(n)) & = \eta_1(g(n)) + \eta_2(g_{\nonlin}(n+h)) - \eta_2(g_{\nonlin}(n)) + n\eta_2( [g(1),\{g(1)^h\}])\\ & := P(n) + Q(n+h) - Q(n) + \sigma(h) n,\label{eta-form}\end{align}
where $P,Q : \Z \rightarrow \R/\Z$ are polynomial sequences of degree at most $d$.

The next lemma is specifically designed to handle the situation that has arisen here. In this lemma it is convenient to reprise a notation from earlier papers of ours (such as \cite{green-tao-u3mobius}): if $\alpha \in \R/\Z$ and $Q > 1$ we write $\Vert \alpha \Vert_{\R/\Z,Q} := \inf_{1 \leq q \leq Q} \Vert q \alpha \Vert_{\R/\Z}$.  In a similar spirit, for any $f: \Z \to \R/\Z$ define
$$ \Vert f \Vert_{C^\infty[N], Q} := \inf_{1 \leq q \leq Q} \Vert qf \|_{C^\infty[N]}.$$

\begin{lemma}[Polynomials lemma]\label{polynomials}
Suppose that $P,Q : \Z \rightarrow \R/\Z$ are polynomial sequences of degree at most $d$ with $P(0) = 0$ and $Q(0) = \partial Q(0) = 0$ and that $\sigma : [N] \rightarrow \R/\Z$ is an arbitrary map. Suppose that there are $\gg \delta^{O(1)} N$ values of $h \in [N]$ such that 
\[ \Vert P(n) + Q(n+h) - Q(n) + \sigma(h)n \Vert_{C^{\infty}[N]} \ll \delta^{-O(1)}.\]
Then $\Vert \partial^i Q \Vert_{\R/\Z,\delta^{-O(1)}} \ll \delta^{-O(1)}/N^i$ for $i \geq 3$, and
\[ \Vert P(1) + \alpha h + \sigma(h) \Vert_{\R/\Z,\delta^{-O(1)}} \ll \delta^{-O(1)}/N\] for $\gg \delta^{O(1)} N$ values of $h \in [N]$, where 
\begin{equation}\label{alphadef}
\alpha := \partial^2 Q(0).
\end{equation}
\end{lemma}
\begin{proof} The assumption implies, looking at the second derivative at $n = 0$, that 
\[ \Vert  \partial^2 (P - Q)(0) + \partial^2 Q(h)\Vert_{\R/\Z} \ll \delta^{-O(1)}/N^2\] for $\gg \delta^{O(1)} N$ values of $h \in [N]$.
 Applying Lemma \ref{strong-polynomial} then implies that 
\[ \Vert \partial^2 (P-Q)(0) + \partial^2 Q \Vert_{C^{\infty}[N],\delta^{-O(1)}} \ll \delta^{-O(1)}/N^2.\]
Thus, as stated, we have
\[ \Vert \partial^i Q \Vert_{\R/\Z, \delta^{-O(1)}} \ll \delta^{-O(1)}/N^i\] for $i \geq 3$, which means 
 in view of the Taylor expansion of $Q$ that we can write
\[ Q(n) = \alpha \binom{n}{2} + R(n),\]
 where $R(0) = R(1) = R(2) = 0$ and $\Vert R \Vert_{C^{\infty}[N],\delta^{-O(1)}} \ll \delta^{-O(1)}$. Substituting back into our assumption yields that
 \[ \left\Vert P(n) + (\alpha h + \sigma(h))n + R(n+h) - R(h) + \alpha\binom{h}{2} \right\Vert_{C^{\infty}[N]} \ll \delta^{-O(1)}\] for $\gg \delta^{O(1)} N$ values of $h \in [N]$.
Differentiating at zero and recalling that $P(0) = 0$ we obtain
\[ \Vert P(1) + \sigma(h) + \alpha h + \partial R(h)\Vert_{\R/\Z} \ll \delta^{-O(1)}/N,\]
which implies in view of the properties of $R$ that
\[ \Vert P(1) + \sigma(h) + \alpha h \Vert_{\R/\Z, \delta^{-O(1)}} \ll \delta^{-O(1)}/N.\]
This completes the proof.\end{proof}

Now let us recall \eqref{eta-form}. We know that $\Vert \eta \circ g_h^{\Box} \Vert_{C^{\infty}[N]} \ll \delta^{-O(1)}$ for $\gg\delta^{O(1)} N$ values of $h$, so let us apply the lemma with $P := \eta_1 \circ g$, 
\begin{equation}\label{qn-def}
Q := \eta_2 \circ g_{\nonlin},
\end{equation}
and $\sigma(h) := \eta_2 ([g(1),\{g(1)\}^h])$. By pigeonholing in $h$ we see that there is some $q \leq \delta^{-O(1)}$ for which
\[ \Vert q \eta_1(g(1)) + q\eta_2([g(1), \{g(1)^h\}]) + q\alpha h \Vert_{\R/\Z} \ll \delta^{-O(1)}/N.\] By redefining  $\eta_1$ and $\eta_2$ (none of the boundedness properties of Lemma \ref{eta-form} are lost by doing this) we may write this as
\begin{equation}\label{to-work-with} \Vert \eta_1(g(1)) + \eta_2([g(1), \{g(1)^h\}]) + q\alpha h \Vert_{\R/\Z} \ll \delta^{-O(1)}/N.
\end{equation}

We now proceed as in \S \ref{heisenberg-sec}, using Mal'cev bases to work with explicit bracket polynomials. 

Since $\eta_2$ annihilates $[G,[G,G]] \subseteq [G,G_{2}]$, we see that the map $x \mapsto \eta_2([g(1),x])$ is a homomorphism.  Thus there exists $\zeta \in \R^m$ such that
\begin{equation}\label{zetadef}
\eta_2([g(1),x]) = \zeta \cdot \psi(x) \md{\Z}
\end{equation}
for all $x \in G$. 
Since $\eta_2$ annihilates $[G,G_{2}]$, all but the first $m_{\lin}$ coordinates of $\zeta$ are zero. Since we have reduced to the case $|\psi(g(1))|  \leq 1$ and the basis $\mathcal{X}$ is $\frac{1}{\delta}$-rational it follows that $|\zeta| \ll \delta^{-O(1)}$.

We now define $\beta := \eta_1(g(1))$ and $\gamma := \psi(g(1))$.
Now since $[G,G] \subseteq G_2$ the map $\psi_{\lin} : G \rightarrow \R^{m_{\lin}}$ which picks out the first $m_{\lin}$ Mal'cev coordinates is a homomorphism, and therefore the first $m_{\lin}$ coordinates of $\psi(g(1)^h)$ are just $\gamma h$. We may now rewrite \eqref{to-work-with} as 
\begin{equation}\label{to-work-with-2}
\Vert \beta + q\alpha h + \zeta \cdot \{\gamma h\} \Vert_{\R/\Z} \ll \delta^{-O(1)}/N
\end{equation}
for $\gg \delta^{O(1)} N$ values of $h \in [N]$.

This assumption is the same as in Proposition \ref{bracket-poly-lem}, except that we do not have a bound on $|q\alpha|$. However, we have

\begin{claim}\label{cl} At least one of the following statements holds:
\begin{enumerate}
\item There is $r \ll \delta^{-O(1)}$ such that $\Vert r\zeta_i \md{\Z} \Vert_{\R/\Z} \ll \delta^{-O(1)}/N$ for $i = 1,\dots, m_{\lin}$; 
\item There exists $k \in \Z^{m_{\lin}}$, $0 < |k| \ll \delta^{-O(1)}$ such that $\Vert k \cdot \gamma\Vert_{\R/\Z} \ll \delta^{-O(1)}/N$.
\end{enumerate}
\end{claim}

\begin{proof} We apply Proposition \ref{bracket-poly-lem} with $\zeta' := (\zeta, 1) \in \R^{m_{\lin}} \times \R$, $\gamma' := (\gamma, q\alpha) \in \R^{m_{\lin}} \times \R$ and $\alpha' := 0$, deducing that either $|\zeta'_i| \ll \delta^{-O(1)}/N$ for all $i = 1,\dots,m_{\lin}$ (in which case (i) holds) or else there exist $k \in \Z^{m_{\lin}}$ and $r \in \Z$, not both zero and with $|k|, |r| \ll \delta^{-O(1)}$, such that $\Vert k \cdot \gamma + qr\alpha \Vert_{\R/\Z} \ll \delta^{-O(1)}/N$. If $r = 0$ then (ii) holds, so assume that $r \neq 0$. Multiplying \eqref{to-work-with-2} through by $r$ we see that for $\geq \delta N$ values of $h \in [N]$ we have
\[ \Vert \tilde\beta + \tilde\alpha h + \tilde \zeta \cdot \{\gamma h\} \Vert_{\R/\Z} \ll \delta^{-O(1)}/N,\] where $\tilde \beta := r \beta$, $\tilde \alpha := \{k\cdot \gamma + qr\alpha\}$ satisfies $|\tilde \alpha| \leq \delta^{-O(1)}/N$ and $\tilde \zeta := r \zeta - k$. Thus we may apply Proposition \ref{bracket-poly-lem} once more to conclude that either $|\tilde \zeta_i| \ll \delta^{-O(1)}/N$ for $i = 1,\dots,m_{\lin}$, which implies (i), or else there is a nonzero $\tilde k \in \Z^{m_{\lin}}$ such that $\Vert \tilde k \cdot \gamma \Vert_{\R/\Z} \ll \delta^{-O(1)}/N$, which implies (ii). This establishes the claim.
\end{proof}

If Claim \ref{cl}(ii) holds then consider the map $\eta : G \rightarrow \R/\Z$ defined by
\[ \eta(x) := k \cdot \psi(x) \md{\Z}.\]
Since $k \in \Z^{m_{\lin}}$, $\eta$ is a horizontal character and we have $|\eta| = |k| \ll \delta^{-O(1)}$. Finally we have
\[ \eta \circ g(n) = \eta(g(1)^n) = n k \cdot \gamma \md{\Z},\]
and so $\Vert \eta \circ g \Vert_{C^{\infty}[N]} \ll \delta^{-O(1)}$. This completes the proof of Theorem \ref{mainthm-ind} in this case.

Suppose then Claim \ref{cl}(i) of the claim holds. For each $i = 1,\dots,m$ consider the map $\tau_i : G \rightarrow \R/\Z$ defined by
\[ \tau_i(x) := r\eta_2([x,\exp(X_i)]).\] Since $[\Gamma,\Gamma] \subseteq \Gamma$ and $[G,G] \subseteq G_2$ we see from the properties established in Lemma \ref{eta-decomp} that $\tau_i$ is a horizontal character which annihilates $G_2$. It is not hard to establish that $|\tau_i| \ll \delta^{-O(1)}$. To do this we write (as usual)
\[ \tau_i(x) = k_i \cdot \psi(x) \md{\Z},\]  where $k_i \in \Z^{m}$ (and in fact $k_i \in \Z^{m_{\lin}}$ since $\tau_i$ annihilates $G_2$). From the definition of $\tau_i$, the bound $r \ll \delta^{-O(1)}$, the $\frac{1}{\delta}$-rationality of the basis $\mathcal{X}$ and Lemma \ref{mult-basis-lem} we have \[ (k_i)_j = \tau_i(\exp(X_j)) = r \eta_2([\exp(X_j),\exp(X_i)]) \ll \delta^{-O(1)}, \] and so indeed $|\tau_i| = |k_i| \ll \delta^{-O(1)}$. 
Now we have
\[ \tau_i \circ g(n) = n\tau_i(g(1)) = rn\zeta_i \md{\Z}\]
where the last equality follows from \eqref{zetadef}. By property (i), this implies that 
\[ \Vert \tau_i \circ g \Vert_{C^{\infty}[N]} \ll \delta^{-O(1)},\]
and so once again we have proved Theorem \ref{mainthm-ind} unless $\tau_i = 0$ for all $i = 1,\dots,m$.

So far we have been successful in deducing Theorem \ref{mainthm-ind} by induction on the degree $d$, but we know from the example at the start of this section that it is not always possible to make such a deduction as $G_{\bullet}$ may be ``reducible'' for $g$. It turns out that the case we have not yet covered corresponds to this situation.

Suppose then that $\tau_i = 0$ for all $i$, so that $\eta_2([x,\exp(X_i)]) = 0$ for all $x \in G$ and all $i \in [m]$. Since the homomorphism $\eta_2$ annihilates $[G,[G,G]] \subseteq [G, G_{2}]$, we see using the identity $[x,yz] = [x,z][z^{-1},[x,y]][x,y]$ that the map $y \mapsto \eta_2([x,y])$ is a homomorphism for any fixed $x$.  It follows that $\eta_2([x,y]) = 0$ for all $x,y \in G$, or in other words that $\eta$ annihilates $[G,G]$.  Thus $\zeta = 0$ (cf. \eqref{zetadef}) and \eqref{to-work-with-2} degenerates to
\[ \Vert \beta + q\alpha h \Vert_{\R/\Z} \ll \delta^{-O(1)}/N\]
for $\gg \delta^{O(1)} N$ values of $h \in [N]$. By Lemma \ref{strong-linear} this implies that 
\[ \Vert \alpha \Vert_{\R/\Z,\delta^{-O(1)}} \ll \delta^{-O(1)}/N^2,\]
and thus by \eqref{alphadef}
\[ \Vert \partial^2 Q \Vert_{\R/\Z,\delta^{-O(1)}} \ll \delta^{-O(1)}/N^2.\]
 where $Q$ was defined in \eqref{qn-def}. We have $Q(0) = Q(1) = 0$ and, by Lemma \ref{polynomials}, $\Vert \partial^i Q \Vert_{\R/\Z,\delta^{-O(1)}} \ll \delta^{-O(1)}/N^i$ for $i \geq 3$. Thus 
\[ \Vert \eta_2 \circ g_2 \Vert_{C^{\infty}[N],\delta^{-O(1)}} \ll \delta^{-O(1)}.\]
Thus there exists $q$, $1 \leq q \leq \delta^{-O(1)}$, such that
\[ \Vert q\eta_2 \circ g_2 \Vert_{C^{\infty}[N]} \ll \delta^{-O(1)}.\] For notational simplicity we rename $q\eta_2$ as $\eta_2$, thus
\begin{equation}\label{nonlinreduction}  \Vert \eta_2 \circ g_2 \Vert_{C^{\infty}[N]} \ll \delta^{-O(1)}.\end{equation}
Roughly speaking, this statement means that $g$ exhibits some essentially \emph{linear} behaviour (in the ``direction'' orthogonal to $\eta_2$) inside $G_2$. For our purposes this means that $G_2$ was too large to accurately capture the quadratic and higher order terms of $g$, and we must pass to a finer filtration $G'_{\bullet}$ which does not have this drawback. This is the point in the proof where we induct on the nonlinearity degree $m_*$.

Now $\eta_2 : G_2 \rightarrow \R/\Z$ has the form 
\[ \eta_2(x) = k \cdot \psi(x) \md{\Z},\] where $k \in \Z^{m_2} \subseteq \Z^m$ satisfies $|k| \ll \delta^{-O(1)}$. In the ensuing discussion we will also need the lift $\tilde \eta_2 : G_2 \rightarrow \R$ defined by
\[ \tilde \eta_2(x) := k \cdot \psi(x).\]
Now the map $\theta : G_2 \times G_2 \rightarrow \R$ defined by $\theta(x,y) := \tilde \eta_2(xy) - \tilde \eta_2(x) - \tilde \eta_2(y)$ is continuous, $\Z$-valued and vanishes when $x = y = \id_G$. Since $G_2 \times G_2$ is connected it follows that $\theta = 0$ identically, and hence the lift $\tilde \eta_2$ is a homomorphism.
\begin{lemma}[A finer subgroup sequence]
Define $G'_{0} = G'_{1} = G$ and $G'_{i} = G_{i} \cap \ker \tilde\eta_2$ for $i \geq 2$. Then $G'_{\bullet} = (G'_{i})_{i=0}^{\infty}$ is a filtration with degree at most $d$ and nonlinearity degree $m'_* \leq m_* - 1$. Each $G'_i$ is closed, connected and $\delta^{-O(1)}$-rational \textup{(}with respect to our Mal'cev basis $\X$ on $G/\Gamma$ adapted to $G_{\bullet}$\textup{)}.
\end{lemma}
\proof Let $\pi : G_2 \rightarrow G_2/[G_2,G_2]$ be the natural projection. It follows from the Baker-Campbell-Hausdorff formula $\exp(X)\exp(Y) = \exp(X + Y + \frac{1}{2}[X,Y] + \dots)$ that $\pi \circ \exp : \g_2 \rightarrow G_2/[G_2,G_2]$ is a linear map. Since $\tilde \eta_2 : G_2 \rightarrow \R$ factors through $G_2/[G_2,G_2]$ it follows that $\tilde \eta_2 \circ \exp : \g_2 \rightarrow \R$ is also a linear map. For $i = m_{\lin} + 1,\dots,m$ we have $\tilde \eta_2 \circ \exp(X_i) = k_i$, an integer of magnitude $O(\delta^{-O(1)})$. Thus by simple linear algebra we see that each Lie algebra $\g'_i = \g_i \cap \ker(\tilde \eta_2 \circ \exp)$ is spanned by $O(\delta^{-O(1)})$-rational combinations of the $X_i$. Thus the $G'_i$ are $O(\delta^{-O(1)})$-rational closed connected subgroups as claimed.
 
If $i,j \geq 2$ then it is clear that $[G'_{i}, G'_{j}] \subseteq G'_{i+j}$ since $\eta_2 : G_{2} \rightarrow \R$ is a homomorphism. We must also check that $[G,G'_{i}] \subseteq G'_{i+1}$ for $i \geq 2$, which follows from the fact that $[G,G_{i}] \subseteq [G,G_{2}] \subseteq \ker \eta_2$. The statement about $m'_*$ is immediate from the fact that $\eta$ is nontrivial, and it is obvious that the degree of $G'_{\bullet}$ is at most $d$.\endproof

We now come to the main result of this section, which allows us to pass to a new sequence $g' \in \poly(\Z,G'_{\bullet})$ with smaller nonlinearity degree than $g$. 

\begin{lemma}[Factorization lemma]\label{factorization} Suppose that \eqref{nonlinreduction} holds. Then we may factor $g = \varepsilon g'\gamma$, where
\begin{enumerate}
\item $\varepsilon \in \poly(\Z,G_{\bullet})$, $\varepsilon(0) = \id_G$, $\varepsilon$ is $(\delta^{-O(1)},N)$-smooth \textup{(}cf. Definition \ref{smooth-seq-def}\textup{)} and $\Vert \eta \circ \varepsilon \Vert_{C^{\infty}[N]} \ll \delta^{-O(1)}$ for all horizontal characters $\eta : G \rightarrow \R/\Z$ with $0 < \|\eta\| \ll \delta^{-O(1)}$;
\item $g' \in \poly(\Z,G'_{\bullet})$;
\item $\gamma \in \poly(\Z,G_{\bullet})$ and $\gamma(n)\Gamma$ is periodic with period $Q \ll \delta^{-O(1)}$.
\end{enumerate}
\end{lemma}

We remark that this lemma is strikingly similar in form to Proposition \ref{prop2.17} below. The proof of the latter result will, in fact, be closely modelled on the proof of this one, but will be rather easier. 

\proof By Lemma \ref{basis-description} and the fact that $g_2(0) = g_2(1) = \id_G$ we have
\[\psi(g_2(n)) = \binom{n}{2} t_2 + \binom{n}{3} t_3 + \dots + \binom{n}{d} t_d,\]
where $t_i \in \R^m$ and the coordinate $(t_i)_j$ is equal to $0$ if $j \leq m-m_{i}$. Thus
\[ \tilde\eta_2 \circ g_2(n) = \sum_{i = 2}^d k \cdot t_i \binom{n}{i}\] From \eqref{nonlinreduction} we thus have
\[ \Vert k \cdot t_i \Vert_{\R/\Z} \ll \delta^{-O(1)}/N^i,\]
$i = 2,\dots,d$. Since $|k| \ll \delta^{-O(1)}$ we may choose vectors $u_i \in \R^{m}$ with $(u_i)_j = 0$ if $j \leq m-m_{i}$ such that $|t_i - u_i| \ll \delta^{-O(1)}/N^i$ and $k \cdot u_i \in \Z$ for $i = 2,\dots,d$.

We may now pick vectors $v_i$ in $\R^{m}$ with $(v_i)_j = 0$ if $j \leq m-m_{i}$, all of whose coordinates are rationals over some denominator $q \ll \delta^{-O(1)}$, such that $k \cdot u_i = k \cdot v_i$ for $i = 2,\dots,d$.

Define sequences $\varepsilon,\gamma : \Z \rightarrow G$ by
\begin{equation}\label{seq-def} \psi(\varepsilon(n)) := \sum_{i=2}^d \binom{n}{i} (t_i - u_i) \quad \mbox{and} \quad \psi(\gamma(n)) := \sum_{i = 2}^d \binom{n}{i} v_i,\end{equation} and set
\[ g'(n) := \varepsilon(n)^{-1} g(n) \gamma(n)^{-1}.\]
Observe from Lemma \ref{basis-description} that $\varepsilon,\gamma$ lie in $\poly(\Z,G_{\bullet})$ and take values in $G_{2}$. We verify the properties of $\varepsilon, g'$ and $\gamma$ in turn.

That $\varepsilon(0) = \id_G$ is obvious. To see that $\varepsilon$ is $(\delta^{-O(1)},N)$-smooth we must confirm that $d(\varepsilon(n),\varepsilon(n-1)) \ll \delta^{-O(1)}/N$ for all $n \in [N]$. Now as a fairly immediate consequence of the definition of $\varepsilon$ we have that
\[ |\psi(\varepsilon(n)) - \psi(\varepsilon(n-1))| \ll \delta^{-O(1)}/N\] and 
\[ |\psi(\varepsilon(n))| \ll \delta^{-O(1)}\] for all $n \in [N]$.
The smoothness therefore follows from Lemma \ref{dx-bounds}.
Finally we must establish the statement about $\eta \circ \varepsilon$, where $\eta : G \rightarrow \R/\Z$ is a horizontal character. It is clear that any horizontal character $\eta : G \rightarrow \R/\Z$ is represented in coordinates as
\[ \eta(g) = k \cdot \psi(g) \md{\Z},\] where $k_i = \eta(\exp(X'_i))$ and so in particular $|k| \ll \delta^{-O(1)}$ if $\Vert \eta \Vert \ll \delta^{-O(1)}$. It follows immediately from the definition of $\varepsilon$ that $\Vert \eta \circ \varepsilon \Vert_{C^{\infty}[N]} \ll \delta^{-O(1)}$, as required.

Next we show that $g' \in \poly(\Z,G'_{\bullet})$.  Now we have
\[ g'(n) = \varepsilon^{-1}(n) g(n) \gamma(n)^{-1} = \varepsilon(n)^{-1} g_2(n) \gamma(n)^{-1} \cdot g(1)^n \cdot [g(1)^{-n},\gamma(n)].\] 
The first derivative of the sequence $n \mapsto g(1)^n$ is $g(1)$ and all higher derivatives are just $\id_G$, so this sequence has coefficients in \emph{any} subgroup sequence. Also the sequence $[g(1)^{-n},\gamma(n)]$ lies in $\poly(\Z,G'_{\bullet})$ since it is in $\poly(\Z,G_{\bullet})$ and takes values in $[G,G_{2}]$, which is annihilated by $\eta$.

By the group property of $\poly(\Z,G'_{\bullet})$ it therefore suffices to check that $\varepsilon^{-1} g_2 \gamma^{-1} \in \poly(\Z,G'_{\bullet})$. Since this sequence lies in $\poly(\Z,G_{\bullet})$, we need only check that it is annihilated by $\tilde\eta_2$, that is to say that
\[ -\eta(\gamma(n))- \eta(\varepsilon(n)) + \eta(g_2(n)) = 0.\]
Computing using coordinates we see that the left-hand side here is
\[ \sum_{i = 2}^d k \cdot (-v_i + u_i - t_i + t_i) \binom{n}{i},\]
which does indeed vanish by our construction of $u_i$ and $v_i$.

Finally we must check that $\gamma(n)\Gamma$ is periodic. By definition and Lemma \ref{ratpoint} we see that $\gamma$ is $\delta^{-O(1)}$-rational (cf. Definition \ref{rat-def-quant}), and then the result follows instantly from Lemma \ref{rat-poly-lem} (ii).\endproof

We will shortly be completing the proof of Theorem \ref{mainthm-ind} in the case that \eqref{nonlinreduction} holds, which is the only case left to handle. We isolate a technical lemma which allows us to deduce $C^{\infty}[N]$-properties of polynomials $p(n)$ from properties of $p(an + b)$. 

\begin{lemma}[Single-parameter extrapolation]\label{lem10.4a} Suppose that $Q,N \geq 1$ are integers and $a, b$ are rationals with height at most $Q$ such that $b \neq 0$. Let $p : \Z \rightarrow \R/\Z$ be a polynomial sequence of degree $d$ and write $\tilde p(n) := p(a + bn)$. Then there is some $q \in \Z$, $1 \leq |q| \ll_{d} Q^{O_{d}(1)}$, such that 
\[ \Vert qp \Vert_{C^{\infty}[N]} \ll_{d} Q^{O_{d}(1)} \Vert \tilde p \Vert_{C^{\infty}[N]}.\]
\end{lemma}

We will defer the proof of this lemma to the next section, in which we prove a more general multiparameter version of it (see Lemma \ref{lem10.4}).

Recall now that in our efforts to prove Theorem \ref{mainthm-ind} by induction we had reduced to the following situation: $g : \Z \rightarrow G$ is a polynomial sequence with $g(0) = \id_G$ and $|\psi(g(1))| \leq 1$, and there is a function $F : G/\Gamma \rightarrow \C$ with nontrivial vertical oscillation $\xi$ and $\Vert F \Vert_{\Lip} \leq 1$ such that
\[ |\E_{n \in [N]} F(g(n)\Gamma)| \geq \delta.\]
Furthermore we reduced to the case when $g$ is ``reducible'' in the sense that \eqref{nonlinreduction} holds. This allows us to
factor $g$ as in Lemma \ref{factorization}, obtaining
\[ |\E_{n \in [N]} F(\varepsilon(n) g'(n) \gamma(n)\Gamma)| \geq \delta.\]
Choose a $Q \ll \delta^{-O(1)}$ such that $\gamma(n)\Gamma$ is periodic with period $Q$, and split $[N]$ up into progressions of length between $N'$ and $2N'$, where $N' := \lfloor \delta^C N\rfloor$, and common difference $Q$. By the pigeonhole principle, there is some such progression $\{n_0 + nQ : n \in [N']\}$ such that 
\[ \left|\E_{n \in [N']} F\left(\varepsilon(n_0 + nQ) g'(n_0 + nQ) \{\gamma(n_0)\} \Gamma\right)\right| \geq \delta/2.\]
Now since $\varepsilon$ is $(\delta^{-O(1)},N)$-smooth we see, using the right-invariance of $d$, that if $C$ is sufficiently large then
\begin{equation}\label{waiting-for-F-dash} \left|\E_{n \in [N']} F\left(\varepsilon(n_0) g'(n_0 + nQ) \{\gamma(n_0)\} \Gamma\right)\right| \geq \delta/4.\end{equation}
Now $g' \in \poly(\Z,G'_{\bullet})$ and hence, by Lemma \ref{poly-restrict}, the sequence \[ \tilde g(n) := \{g(n_0)\}^{-1}\varepsilon(n_0)g'(n_0 + nQ)\{\gamma(n_0)\}\] is also in $\poly(\Z,G'_{\bullet})$. The inequality \eqref{waiting-for-F-dash} may be rewritten as
\begin{equation}\label{F-dash} |\E_{n \in [N']} \tilde F (\tilde g(n)\Gamma)| \geq \delta/4,\end{equation}
where $\tilde F(x) := F(\{g(n_0)\} x)$. By Lemma \ref{approx-left} we have $\Vert \tilde F \Vert_{\Lip} \ll \delta^{-O(1)}$. Noting that $\tilde g(0) = \id_G$, we may thus apply the inductive hypothesis that Theorem \ref{mainthm-ind} holds with parameters $(d,m_*-1)$, deducing that there is some horizontal character $\tilde \eta$ with $0 < \| \tilde \eta \| \ll \delta^{-O(1)}$ such that 
\[ \Vert \tilde\eta \circ \tilde g \Vert_{C^{\infty}[N]} \ll \delta^{-O(1)}.\]
From Lemma \ref{lem10.4a} and the definition of $\tilde g$ it follows that there is a horizontal character $\eta$ with $0 < \|\eta\| \ll \delta^{-O(1)}$, such that 
\[ \Vert \eta \circ g'' \Vert_{C^{\infty}[N]} \ll \delta^{-O(1)},\] where
\[ g''(n) := \{g(n_0)\}^{-1}\varepsilon(n_0)g'(n)\{\gamma(n_0)\}.\]
Since $g'(0) = \id_G$, it follows that 
\[ \Vert \eta \circ g' \Vert_{C^{\infty}[N]} \ll \delta^{-O(1)}.\]
To complete the proof of the result we must, of course, replace $g'$ by $g := \varepsilon g' \gamma$. To do this, note first that by multiplying $\eta$ by an integer of size $O(\delta^{-O(1)})$ if necessary we in fact have
\[ \Vert \eta \circ \gamma\Vert_{C^{\infty}[N]} = 0,\] since the Mal'cev coordinates $\psi(\gamma(n)\Gamma)$ are always rationals over some denominator $\ll \delta^{-O(1)}$.
From the property (i) of Lemma \ref{factorization} we have that $\Vert \eta \circ \varepsilon \Vert_{C^{\infty}[N]} \ll \delta^{-O(1)}$. Putting all this together, we obtain
\[ \Vert \eta \circ g \Vert_{C^{\infty}[N]} \leq \Vert \eta \circ \varepsilon \Vert_{C^{\infty}[N]} + \Vert \eta \circ g' \Vert_{C^{\infty}[N]} + \Vert \eta \circ \gamma \Vert_{C^{\infty}[N]} \ll \delta^{-O(1)},\]
completing (at last!) the proof of Theorem \ref{mainthm-ind}.\endproof

Let us remind the reader that, by remarks immediately following the statement of Theorem \ref{mainthm-ind}, we have also completed the proof of Theorem \ref{main-theorem}.

\section{The multiparameter Leibman theorem}\label{multi-deduction}\label{sec12}

We have proved one of our main results, Theorem \ref{main-theorem}.   In this section we bootstrap this result into a multiparameter version of itself. Strictly speaking, this step is not necessary in order to establish any of the results stated in the introduction, however the arguments here are not terribly difficult, and will be needed in order to obtain multiparameter analogues of the those results.

Recall from \ref{poly-sequences-sec} the definition of $\poly(\Z^t,G_{\bullet})$, the group of polynomial sequences $g : \Z^t \rightarrow G$ with coefficients in $G_{\bullet}$. Recall also the definition of, and notation for, multibinomial coefficients $\binom{\vec{n}}{\vec{j}}$.

We need an analogue of the smoothness norms $C^{\infty}[N]$ in the multiparameter setting. To set these up, we introduce the Taylor coefficients of a polynomial map $g : \Z^t \rightarrow \R/\Z$.

\begin{definition}[Taylor expansion]
Suppose that $g : \Z^t \rightarrow \R/\Z$ is a polynomial map. Then we define the \emph{Taylor coefficients} $\alpha_{\vec j} \in \R/\Z$ for $\vec j \in \Z^t$ to be the unique elements of $\R/\Z$ such that
\[ g(\vec{n}) = \sum_{\vec{j}} \binom{\vec{n}}{\vec{j}} \alpha_{\vec{j}}\]
for all $\vec n$; it is not difficult to verify the existence and uniqueness of these coefficients, and to check that if $g$ has degree at most $d$ then $\alpha_{\vec{j}} = 0$ unless $|\vec{j}| \leq d$, where $|\vec{j}| := j_1 + \dots + j_t$.
\end{definition}

\begin{definition}[Smoothness norms]
Suppose that $g : \Z^t \rightarrow \R/\Z$ is a polynomial map with Taylor expansion
\[ g(\vec{n}) = \sum_{\vec{j}} \alpha_{\vec{j}} \binom{\vec{n}}{\vec{j}}.\]
 Then for any $t$-tuple $\vec N = (N_1,\ldots,N_t)$ for $N_1,\ldots,N_t \geq 1$ we write $[\vec N] := [N_1] \times \ldots \times [N_t]$ and
\[ \Vert g \Vert_{C^{\infty}[\vec N]} := \sup_{\vec{j} \neq 0} \vec N^{\vec{j}}\Vert \alpha_{\vec{j}} \Vert_{\R/\Z},\] where $\vec N^{\vec{j}} := N_1^{j_1} \dots N_t^{j_t}$.
\end{definition}

We have the following generalisation of Lemma \ref{smooth-slow-1}:

\begin{lemma}[Smooth polynomials vary slowly]\label{smooth-slow} 
Let $g : \Z^t \rightarrow \R/\Z$ be a polynomial sequence of degree at most $d$ and suppose that $\vec{n} \in [\vec N]$. Then for any $i \in [t]$ we have  
\[ |g(\vec{n}) - g(\vec{n}-\vec e_i)| \ll_{t,d} \frac{1}{N_i}\Vert g \Vert_{C^{\infty}[\vec N]},\] where $\vec e_i = (0,\dots,0,1,0,\dots,0)$ is the $i^{\operatorname{th}}$ basis vector of $\Z^t$.
\end{lemma}
\proof From the Taylor expansion and binomial identities we have
\[ g(\vec{n}) - g(\vec{n}-\vec e_i) = \sum_{|\vec{j}| \leq d} \binom{\vec{n}-\vec e_i}{\vec{j} - \vec e_i}  \alpha_{\vec{j}}.\]
Thus
\[ |g(\vec{n}) - g(\vec{n}-e_i)| \leq \frac{1}{N_i}\Vert g \Vert_{C^{\infty}[\vec N]} \sum_{\substack{|\vec{j}| \leq d\\ \vec{j} \neq 0}} \frac{1}{\vec N^{\vec{j} - \vec e_i}}\binom{\vec{n}-\vec e_i}{\vec{j}-\vec e_i} \ll_{t, d} \frac{1}{N_i}\Vert g \Vert_{C^{\infty}[\vec N]},\]
as required.\endproof

We now give a multiparameter version of Lemma \ref{lem10.4a}, which implies that lemma as the $t=1$ special case.

\begin{lemma}[Multiparameter extrapolation]\label{lem10.4} Suppose that $t,Q, N_1,\ldots,N_t,d \geq 1$ are integer parameters
and that $a_i,b_i \in \Q$, $i = 1,\dots,t$ are rationals of height at most $Q$ with $b_i \neq 0$. Let $p : \Z^t \rightarrow \R/\Z$ be a polynomial map of degree at most $d$ and write $\tilde p(\vec{n}) := p(a_1 + b_1n_1,\dots, a_t + b_tn_t)$. Then there is some $q \in \Z$, $|q| \ll_{d,t} Q^{O_{d,t}(1)}$, such that 
\[ \Vert qp \Vert_{C^{\infty}[\vec N]} \ll_{d,t} Q^{O_{d,t}(1)} \Vert \tilde p \Vert_{C^{\infty}[\vec N]}.\]
\end{lemma}

\proof First of all observe that, if $a,b \in \Q$ are rationals with height at most $Q$ and $b \neq 0$, we may expand
\[ \binom{(n-a)/b}{j} = \sum_{j' \leq j} c(a,b,j',j) \binom{n}{j'},\] where $c(a,b,j',j)$ is a rational number with height $O_j(Q^{O_{j}(1)})$. Indeed we clearly have $c(a,b,j,j) = b^{-j}$, and we may then compute $c(a,b,j-1,j), c(a,b,j-2,j),\dots$ in turn.

Multiplying such relations together we obtain a multiparameter version, viz.
\[ \prod_{i=1}^t \binom{(n_i-a_i)/b_i}{j_i} = \sum_{\vec{j}' \leq \vec{j}} c(\vec{a},\vec{b},\vec{j}',\vec{j}) \binom{\vec{n}}{\vec{j}},\] where $\vec{j}' \leq \vec{j}$ means that each component of $\vec{j}'$ is at most the corresponding component of $\vec{j}$.

Applying this allows us to give the Taylor coefficients $\alpha_{\vec{j}}$ of $p$ in terms of those of $\tilde p$. Indeed we have
\[ p(\vec{n}) = \tilde p(\frac{n_1 - a_1}{b_1},\dots,\frac{n_t - a_t}{b_t}) = \sum_{\vec{j}} \prod_{i=1}^t \binom{(n_i - a_i)/b_i}{\vec{j}} \tilde{\alpha}_{\vec{j}} =  \sum_{\vec{j}}\sum_{\vec{j}'\leq \vec{j}} \binom{\vec{n}}{\vec{j}'} c(\vec{a},\vec{b},\vec{j}',\vec{j}) \tilde \alpha_{\vec{j}},\] and so
\[ \alpha_{\vec{j}} = \sum_{\vec{j}' \geq \vec{j}} c(\vec{a},\vec{b},\vec{j},\vec{j}') \tilde{\alpha}_{\vec{j}'}.\]
To obtain the lemma, we simply need to take $q$ to be the product of all the denominators of the rationals $c(\vec{a},\vec{b},\vec{j},\vec{j'})$, which is clearly $\ll_{\vec d,t} Q^{O_{d,t}(1)}$.\endproof

\begin{definition}[Multiparameter equidistribution]\label{almost-equidistribution-multi}  Let $G/\Gamma$ be a nilmanifold and let $\delta > 0$. An finite sequence $(g(\vec n)\Gamma)_{\vec n \in P}$ in $G/\Gamma$ indexed by a finite non-empty set $P$ is \emph{$\delta$-equidistributed} if we have
$$ \left|\sum_{\vec n \in P} F(g(\vec n) \Gamma) - \int_{G/\Gamma} F\right| \leq \delta \|F\|_{\Lip}$$
for all Lipschitz functions $F: G/\Gamma \to \C$.  If $\vec N = (N_1,\ldots,N_t)$, we say that a sequence $(g(\vec n)\Gamma)_{\vec n \in [N]}$ is \emph{totally $\delta$-equidistributed} if we have
$$ \left|\sum_{\vec n \in P_1 \times \ldots \times P_t} F(g(\vec n) \Gamma) - \int_{G/\Gamma} F\right| \leq \delta \|F\|_{\Lip}$$
whenever $P_i$ are arithmetic progressions in $[N_i]$ of length at least $\delta N_i$ for each $1 \leq i \leq t$.
\end{definition}

We can now give the multiparameter version\footnote{Note added in 2015: there are some errors in the statement and proof of this theorem. Regarding the statement, one has to exclude the possibility that one of the $N_i$ is small in the sense that $N_i \ll \delta^{-O_{d,m,t}(1)}$, or else restrict to the equal-sides case $N_1=\dots=N_t$.  Furthermore, the proof given here is incorrect in various ways. See the erratum at {\tt arXiv:1311.6170} for a correct argument.  Similar corrections need to be made to the theorems in the next two sections.} of Theorem \ref{main-theorem}.

\begin{theorem}[Multiparameter quantitative Leibman theorem]\label{main-theorem-multi} Let $s,m,t \geq 1$ and $0 < \delta < 1/2$, and let $N_1,\dots,N_t \geq 1$ and $d \geq 1$ be integers. 
Suppose that $G/\Gamma$ is an $m$-dimensional nilmanifold equipped with a $\frac{1}{\delta}$-rational Mal'cev basis $\X$ adapted to some filtration $G_{\bullet}$ of degree $d$, and that $g \in \poly(\Z^t,G_{\bullet})$. Then either $(g(\vec{n})\Gamma)_{\vec{n} \in [\vec N]}$ is $\delta$-equidistributed, or else there is some horizontal character $\eta$ with $0 < \|\eta\| \ll  \delta^{-O_{d,m,t}(1)}$ such that
\[ \Vert \eta \circ g\Vert_{C^{\infty}[\vec N]} \ll \delta^{-O_{d,m,t}(1)}.\]
\end{theorem}

\begin{proof} 
We allow all implied constants to depend on $d,m$ and $t$. Suppose that $(g(\vec{n})\Gamma)_{\vec{n} \in [\vec N]}$ is not $\delta$-equidistributed. Suppose to begin with that $N_1 \geq \delta^{-C}$.

A simple averaging argument confirms that, for $\gg \delta^{O(1)} N_2\dots N_t$ values of $(n_2,\dots,n_t) \in [N_2 \times \dots \times N_t]$, the polynomial sequence $(g_{n_2,\dots,n_t}(n))\Gamma)_{n \in [N_1]}$ is not $\delta^{O(1)}$-equidistributed, where $g_{n_2,\dots,n_t}(n) := g(n,n_2,\dots,n_t)$.

For each such tuple $(n_2,\dots,n_t)$, Theorem \ref{main-theorem} implies that there is some horizontal character $\eta_{n_2,\dots,n_t}$ with $0 < \|\eta\| \ll \delta^{-O(1)}$ such that 
\[ \Vert \eta \circ g_{n_2,\dots,n_t} \Vert_{C^{\infty}[N_1]} \ll \delta^{-O(1)}.\]
By pigeonholing in $\eta$ and passing to a thinner set of tuples $(n_2,\dots,n_t)$ we may assume that $\eta_{n_2,\dots,n_t}$ does not depend on $(n_2,\dots,n_t)$. Writing $p := \eta \circ g$ and expanding
\[ p := \sum_{i_1=0}^d  p_{i_1}(n_2,\dots,n_t) \binom{n_1}{i_i},\] where the $p_{i_1}$ are polynomials, we therefore see that
\begin{equation}\label{condition} \Vert  p_{i_1}(n_2,\dots,n_t)\Vert_{\R/\Z} \ll \delta^{-O(1)}/N_1^{i_1},\end{equation} for $\gg \delta^{O(1)}N_2\dots N_d$ values of $(n_2,\dots,n_t)$, for each $i_1 = 0,\dots,d$. In particular (for each $i_1$) there are $\gg \delta^{O(1)}N_3 \dots N_t$ values of $(n_3,\dots,n_t)$ for which \eqref{condition} holds for $\gg \delta^{O(1)}N_2$ values of $n_2$.

Suppose that $i_1 > 0$. Writing
\[ p_{i_1}(n_2,\dots,n_t) = \sum_{i_2 = 0}^t  p_{i_1,i_2}(n_3,\dots,n_t) \binom{n_2}{i_2}\] and applying Lemma \ref{strong-polynomial}, we see that for $\gg \delta^{O(1)}N_3\dots N_t$ tuples $(n_3,\dots,n_t)$ there is $q_{i_1}(n_3,\dots,n_t) \ll \delta^{-O(1)}$ such that 
\[ \Vert q_{i_1}(n_3,\dots,n_t) p_{i_1,i_2}(n_3,\dots,n_t)\Vert_{\R/\Z} \ll \delta^{-O(1)}/N_1^{i_1} N_2^{i_2}.\] Note that the application of Lemma \ref{strong-polynomial} is valid because $i_1 > 0$ and $N_1 \geq \delta^{-C}$; this guarantees that the parameter $\epsilon$ in that lemma is small enough. Pigeonholing in $(n_3,\dots,n_t)$ and passing to a somewhat smaller set of these tuples we may suppose that $q_{i_1} = q_{i_1}(n_3,\dots,n_t)$ is constant. 

We now continue in this vein, obtaining successively quantities $q_{i_1,i_2,\dots,i_r} \ll \delta^{-O(1)}$. At the final stage we obtain
\[ \Vert q_{i_1,\dots,i_t} p_{i_1,\dots,i_t} \Vert_{\R/\Z} \ll \delta^{-O(1)}/N_1^{i_1} \dots N_t^{i_t}\] or, in our earlier notation,
\begin{equation}\label{star-eq} \Vert q_{\vec{i}} p_{\vec{i}} \Vert_{\R/\Z} \ll \delta^{-O(1)}/\vec N^{\vec{i}}.\end{equation} 
This has been obtained for all $\vec{i}$ with $i_1 > 0$ on the assumption that $N_1 \geq \delta^{-C}$. By switching the indices $i_1,\dots,i_t$ if necessary, we may in fact obtain such a $q_{\vec{i}}$ whenever there is some $r$ with $N_r^{i_r} > \delta^{-C}$. If this is not the case for any $r$ then \eqref{star-eq} holds anyway for trivial reasons (for any $q_{\vec{i}} \ll \delta^{-O(1)}$).

Note that by construction the $p_{\vec{i}}$ are simply the Taylor coefficients of $p$.

Taking $q := \prod_{\vec{i}} q_{\vec{i}}$ we see that $q \ll \delta^{-O(1)}$ and that
\[ \Vert q p_{\vec{i}} \Vert_{\R/\Z} \ll \delta^{-O(1)}/\vec N^{\vec{i}}\] for each index $\vec{i}$ and thus
\[ \Vert q \eta \circ g \Vert_{C^{\infty}[\vec N]} \ll \delta^{-O(1)}.\]
The theorem follows.
\end{proof}

\section{A multiparameter initial factorization theorem}\label{sec13}

Having just established Theorem \ref{main-theorem-multi}, we now use it to obtain an initial factorization theorem for multiparameter polynomial sequences.  We first give a multiparameter version of Definition \ref{smooth-seq-def}, the definition of a smooth sequence (the multiparameter version of a rational sequence is obvious).

\begin{definition}[Multiparameter smooth sequences]\label{smooth-seq-def-multi}  Let $G/\Gamma$ be a nilmanifold with a Mal'cev basis $\X$.  Let $(\varepsilon(n))_{n \in \Z^t}$ be a multiparameter sequence in $G$, let $M\geq 1$ be an integer and let $\vec N = (N_1,\dots,N_t)$ with $N_i \geq 1$ for all $i$.  We say that $(\varepsilon(n))_{n \in \Z^t}$ is \emph{$(M,\vec N)$-smooth} if we have $d(\varepsilon(n),\id_G) \leq M$ and $d(\varepsilon(\vec n),\varepsilon(\vec n-\vec e_i)) \leq M/N_i$ for all $\vec n \in [\vec N]$.
\end{definition}

Here, then, is the main result of this section.

\begin{proposition}[Factorization of poorly-distributed polynomial sequences]\label{prop2.17}  Let $s,m,$ $t \geq 1$, let $0 < \delta < 1/2$, and let $N_1,\ldots,N_t \geq 1$ and $d \geq 0$ be integers.  Write $\vec N := (N_1,\ldots,N_t)$.  Let $G/\Gamma$ be an $m$-dimensional nilmanifold with a $\frac{1}{\delta}$-rational Mal'cev basis $\X$ adapted to a filtration $G_{\bullet}$ of degree $d$, and suppose that $g \in \poly(\Z^t,G)$. Suppose that $(g(\vec{n})\Gamma)_{\vec{n} \in [\vec N]}$ is not totally $\delta$-equidistributed.  Then there is a factorization $g = \varepsilon g' \gamma$, where $\varepsilon, g', \gamma \in \poly(\Z^t,G_{\bullet})$ are polynomial sequences with the following properties: 
\begin{enumerate}
\item $\varepsilon : \Z^t \rightarrow G$ is $(O(\delta^{-O_{d, m, t}(1)}),\vec N)$-smooth;
\item $g' : \Z^t \rightarrow G'$ takes values in a connected proper subgroup $G'$ of $G$ which is $O(\delta^{-O_{d, m, t}(1)})$-rational relative to $\X$;
\item $\gamma: \Z^t \rightarrow G$ is $\delta^{-O_{d, m, t}(1)}$-rational.
\end{enumerate}
\end{proposition}

\begin{proof} We will allow all implied constants to depend on $d, m$ and $t$.

We first reduce to the case $g(0) = \id_G$, by factorizing $g = \{g(0)\} \tilde g [g(0)]$ where $\tilde g$ is the polynomial sequence $\tilde g := \{g(0)\}^{-1}  g [g(0)]^{-1}$, for which $\tilde g(0) = \id_G$.  If $(g(\vec n)\Gamma)_{\vec n \in [\vec N]}$ is not totally $\delta$-equidistributed, then one easily verifies using Lemma \ref{approx-left} that $(\tilde g(\vec n)\Gamma)_{\vec n \in [\vec N]}$ is not totally $\tilde \delta$-equidistributed for some $\tilde \delta \gg \delta^{O(1)}$.  Applying the proposition to $\tilde g$, we obtain a factorization $\tilde g = \tilde \eps g' \tilde \gamma$. 
Setting $\eps := \{g(0)\} \tilde \eps$ and $\gamma := \tilde \gamma [g(0)]$, we certainly have $g = \eps g' \gamma$. The sequence $\gamma$ is $\delta^{-O(1)}$-rational by Lemma \ref{ratpoint} and (the multiparameter version of) Lemma \ref{rat-poly-lem}. The sequence $\eps$ is $(\delta^{-O(1)},\vec{N})$-smooth by Lemma \ref{approx-left}.

Henceforth, then, we assume that $g(0)=\id_G$.  By hypothesis, we can find progressions $P_i := \{ a_i + b_i n_i : n_i \in [N'_i]\}$ in $[N_i]$ with $N'_i \geq \delta N_i$ such that the polynomial sequence $\tilde g : \Z^t \rightarrow G$ defined by $\tilde g(\vec{n}) = g(a_1 + b_1 n_1,\dots, a_t + b_t n_t)$ is such that $(\tilde g(\vec{n})\Gamma)_{\vec{n} \in [\vec N']}$ fails to be $\delta$-equidistributed, where $\vec N' := (N'_1,\ldots,N'_t)$. by Lemma \ref{poly-restrict} we have $\tilde g \in \poly(\Z^t,G_{\bullet})$. Applying Theorem \ref{main-theorem} we conclude the existence of a horizontal character $\tilde \eta: G \to \R/\Z$ with $0 < \|\tilde \eta\| \ll \delta^{-O(1)}$ such that
\[ \Vert \tilde \eta \circ \tilde g \Vert_{C^{\infty}[\vec N']} \ll \delta^{-O(1)}.\] At the expense of worsening the exponent of the $\delta^{-O(1)}$, we may replace $[\vec N']$ here by $[\vec N]$. Applying Lemma \ref{lem10.4}, we deduce that there is a horizontal character $\eta: G \to \R/\Z$ with $0 < \| \eta \| \ll \delta^{-O(1)}$  such that 
\begin{equation}\label{smoothness-cond} \Vert \eta \circ g\Vert_{C^{\infty}[\vec N]} \ll \delta^{-O(1)}.\end{equation}
Take $G'$ to be the connected component of $\ker(\eta)$. Then $G'$ is rather clearly a subgroup of $G$ which is $O(\delta^{-O(1)})$-rational relative to $\X$.

Write
\[ \psi(g(n)) = \sum_{\vec{j}} t_{\vec{j}} \binom{\vec{n}}{\vec{j}},\] where $t_{\vec{j}} \in \R^{m}$. By Lemma \ref{basis-description} we know that the coordinate $(t_{\vec{j}})_i$ is equal to 0 if $i \leq m - m_{|\vec{j}|}$. The horizontal character $\eta$ is given in coordinates by
\[ \eta \circ g(\vec{n}) = \sum_{\vec{j}} k \cdot t_{\vec{j}} \binom{\vec{n}}{\vec{j}},\] where $|k| \ll \delta^{-O(1)}$, and \eqref{smoothness-cond} tells us that $\Vert k \cdot t_{\vec{j}} \Vert_{\R/\Z} \ll \delta^{-O(1)}/\vec N^{\vec{j}}$ for all $\vec{j} \neq 0$. Since $|k| \ll \delta^{-O(1)}$ we may choose vectors $u_{\vec{j}} \in \R^{m_1}$ such that $|t_{\vec{j}} - u_{\vec{j}}| \ll \delta^{-O(1)}/\vec N^{\vec{j}}$ and $k \cdot u_{\vec{j}} \in \Z$ for all $\vec{j} \neq 0$. We then choose vectors $v_{\vec{j}} \in \R^{m_1}$, all of whose coordinates are rationals with complexity at most $O(\delta^{-O(1)})$, such that $k \cdot u_{\vec{j}} = k \cdot v_{\vec{j}}$ for all $\vec{j} \neq 0$. We may insist that the $u_{\vec{j}}$ and $v_{\vec{j}}$ have the same support properties as the $t_{\vec{j}}$, namely that $(u_{\vec{j}})_i = (v_{\vec{j}})_i = 0$ if $i \leq m - m_{|\vec{j}|}$.

Define polynomial sequences $\varepsilon,\gamma : \Z^t \rightarrow G$ in terms of their Mal'cev coordinates by
\[ \psi(\varepsilon(\vec{n})) = \sum_{\vec{j} \neq 0}(t_{\vec{j}} - u_{\vec{j}}) \binom{\vec{n}}{\vec{j}} \qquad \mbox{and} \qquad \psi(\gamma(\vec{n})) = \sum_{\vec{j} \neq 0}v_{\vec{j}} \binom{\vec{n}}{\vec{j}},\] and
\[ g' := \varepsilon^{-1} g \gamma^{-1}.\]
By Lemma \ref{basis-description} and the fact that $\poly(\Z^t,G_{\bullet})$ is a group we see that all three of $\eps,g'$ and $\gamma$ lie in $\poly(\Z^t,G_{\bullet})$. We must check the claims (i), (ii) and (iii).
The claim (ii) is clear. To prove (i), that is to say that $\eps$ is $(\delta^{-O(1)},\vec{N})$-smooth, we need to show that 
\[ d(\eps(\vec{n}), \eps(\vec{n} - \vec{e}_i)) \ll \delta^{-O(1)}/N_i\] for $\vec{n} \in \vec{N}$. But as a fairly immediate consequence of the definition of $\eps$ we have the bound
\[ |\psi(\eps(\vec{n})) - \psi(\eps(\vec{n} - \vec{e}_i))| \ll \delta^{-O(1)}/N_i,\]
and so the desired bound follows from Lemma \ref{dx-bounds}. Finally we note that (iii) follows immediately from the definition of $\gamma$ and the properties of rational points described in Lemma \ref{ratpoint}.
\end{proof}

\section{A multiparameter complete factorization theorem}\label{factor-sec}

The last major task of the paper is to iterate Proposition \ref{prop2.17} to deduce our a multiparameter version of our main result, Theorem \ref{mainthm-prelim}.   We first need a technical lemma.

\begin{lemma}[Product of smooth sequences is smooth]\label{prod-smooth}
Let $G/\Gamma$ be a nilmanifold of dimension $m$ and let $M \geq 2$ and $N_1,\dots,N_t \geq 1$ be parameters. Suppose that $\X$ is an $M$-rational Mal'cev basis for $G/\Gamma$ adapted to some filtration $G_{\bullet}$ of degree $d$, and suppose that the maps $\varepsilon_1,\varepsilon_2 : \Z^t \rightarrow G$ are $(M,\vec{N})$-smooth in the sense of Definition \ref{smooth-seq-def-multi}. Then the product $\eps_1\eps_2$ is $(M^{O_{d,m,t}(1)},\vec{N})$-smooth.
\end{lemma}

\proof First of all we have, for all $\vec{n} \in \vec{N}$,

By the triangle inequality we have
\begin{align*} d(\varepsilon_1\varepsilon_2(\vec{n} - \vec e_i),\varepsilon_1\varepsilon_2(\vec{n})) \leq d(\varepsilon_1(\vec{n} &- \vec e_i)\varepsilon_2(\vec{n} - \vec e_i),\varepsilon_1(\vec{n})\varepsilon_2(\vec{n} - \vec e_i)) \\ & + d(\varepsilon_1(\vec{n})\varepsilon_2(\vec{n} - \vec e_i),\varepsilon_1(\vec{n})\varepsilon_2(\vec{n})).\end{align*}
Using the fact that $d(\varepsilon_1(\vec{n}),\id_G), d(\varepsilon_2(\vec{n}),\id_G) \leq Q$ for all $\vec{n} \in [\vec{N}]$, the result now follows immediately from the right-invariance of $d$, Lemma \ref{approx-left} and Lemma \ref{dx-bounds}.\endproof

We can now state and prove the multiparameter version of Theorem \ref{mainthm-prelim} that we need.

\begin{theorem}[Multiparameter factorization theorem]\label{mainthm-multi}
Let $s,m,t \geq 0$, let $M_0 \geq 2$ and $A > 0$, and let $N_1,\ldots,N_t \geq 1$ and $d \geq 0$. Suppose that $G/\Gamma$ is an $m$-dimensional nilmanifold with a $M_0$-rational Mal'cev basis $\X$ adapted to some filtration $G_{\bullet}$ of degree $d$, and that $g \in \poly(\Z^t, G_{\bullet})$. Then there is a some $M$, $M_0 \leq M \ll M_0^{O_{A,m,d}(1)}$, a subgroup $G' \subseteq G$ which is $M$-rational with respect to $\X$ and a decomposition $g = \varepsilon g' \gamma$ into sequences $\eps,g',\gamma \in \poly(\Z^t,G_{\bullet})$ with the following properties:
\begin{enumerate}
\item $\varepsilon$ is $(M,\vec{N})$-smooth;
\item $g'$ takes values in $G'$ and with respect to the restriction of the metric $d$ the orbit $(g'(\vec{n})\Gamma')_{\vec{n} \in P_1 \times \dots \times P_t}$ is $1/M^A$-equidistributed in $G'/\Gamma'$, for any subprogressions $P_i \subseteq [N_i]$ with $|P_i| \geq N_i/M^A$;
\item $\gamma$ is a $M$-rational.
\end{enumerate}
\end{theorem}

\begin{proof}
Let $1/M_0^A = \delta_1 > \delta_2 > \dots$ be a sequence of parameters to be specified as the proof unfolds. For each $i = 1,\dots,t$ let $P_i \subseteq [N_i]$ be a progression of size at least $\delta_1 N_i$. From Proposition \ref{prop2.17} we know that either $(g(\vec{n}))_{\vec{n} \in P_1 \times \dots \times P_t}$ is $\delta_1$-equidistributed on $G/\Gamma$, or else there is a factorization
\[ g = \varepsilon_1 g_1 \gamma_1\] where $\varepsilon_1,g_1,\gamma_1 \in \poly(\Z^t,G_{\bullet})$, $g_1$ takes values in some $O(\delta_1^{-O(1)})$-rational proper subgroup $G' \subseteq G$, $\varepsilon_1$ is $(O(\delta_1^{-O(1)}),\vec{N})$-smooth and $\gamma_1$ is $O(\delta_1^{-O(1)})$-rational. Set $\Gamma' := G' \cap \Gamma$; we are now going to look at the distribution properties of $(g(\vec{n}))$ inside $G'/\Gamma'$ by applying Proposition \ref{prop2.17} once more.

To do this we choose an $M_0^{O_{A,d,m}(1)}$-rational Mal'cev basis $\X'$ for $G'/\Gamma'$ adapted to the filtration $G'_{\bullet} := G_{\bullet} \cap G'$. This is possible by Lemma \ref{sub-nil-basis}, and we may furthermore ensure that each of the basis elements $\X'_i$ is an $M_0^{O_{A,d,m}(1)}$-rational combination of the $X_i$. In view of Lemma \ref{comparison-lemma} we have
\begin{equation}\label{metric-comparison} d'(x,y) \ll M_0^{O_{A,d,m}(1)} d(x,y)\end{equation} for all $x,y \in G'/\Gamma'$.

Take $\delta_2 := cM_0^{-C}$ for some constants $c,C$ depending on $m,d$ and $A$. If these are chosen suitably, and if $(g_1(\vec{n}))_{\vec{n} \in P_1 \times \dots \times P_t}$ is $\delta_2$-equidistributed on $G'/\Gamma'$ with respect to the metric $d'$ for all progressions $P_i$ with $|P_i| \geq \delta_2N_i$, then by \eqref{metric-comparison} the conclusion of the theorem holds. If this is not the case then we apply Proposition \ref{prop2.17} once again, obtaining a factorization $g_1 = \varepsilon_2 g_2 \gamma_2$ where $g_2$ takes values in some $O(\delta_2^{-O(1)})$-rational proper subgroup $G'' \subseteq G'$, $\varepsilon_2 : \Z^t \rightarrow G'$ is $(O(\delta_2^{-O(1)}),\vec{N})$-smooth and $\gamma_2: \Z^t \rightarrow G'$ is $O(\delta_2^{-O(1)})$-rational. 

This allows us to write 
\[ g = \varepsilon_2 \varepsilon_1 g_2 \gamma_1 \gamma_2.\]

Now it follows from Lemma \ref{comparison-lemma} that $\varepsilon_2 : \Z^t \rightarrow G'$ is in fact $(M_0^{O(1)},\vec{N})$-smooth when regarded as a map into $G$ (smoothness now being measured with respect to the metric $d$). By Lemma \ref{prod-smooth}, $\varepsilon_2\varepsilon_1 : \Z^t \rightarrow G$ is also $(M_0^{O(1)},\vec{N})$-smooth. By Lemma \ref{ratpoint} (v), $\gamma_1\gamma_2 : \Z^t \rightarrow G$ is $O(\delta_2^{-O(1)})$-rational. Thus, taking $\varepsilon := \varepsilon_2\varepsilon_1$, $\gamma := \gamma_1\gamma_2$ and $g' := g_2$, the conclusion of the theorem holds unless $(g_2(\vec{n}))_{\vec{n} \in P_1 \times \dots \times P_t}$ fails to be equidistributed on $G''/\Gamma''$. We now proceed as before, introducing a Mal'cev basis $\X''$ and encoding this lack of equidistribution as the failure of $(g_2(\vec{n}))_{\vec{n} \in P_1 \times \dots \times P_t}$ to be $\delta_3$-equidistributed relative to the metric $d'' = d_{\X''}$ for some $\delta_3 = cM_0^{-C}$ (the constants $c,C$ are, of course, not the same as before). We may then apply Proposition \ref{prop2.17} once more, and so on.

It is clear that the total number of iterations is bounded by $m = \dim G$. The implied constants in the $O()$ notation increase with each iteration, but since the total number of iterations is at most $m=O(1)$, this does not cause a difficulty.
Thus we obtain a proof of our main theorem.
\end{proof}

It follows from Lemma \ref{rat-poly-lem} (or rather the multidimensional version of it) that $(\gamma(\vec{n})\Gamma)_{\vec{n} \in \Z^t}$ is periodic in each direction in the sense that $\gamma(\vec{n} + Q\vec{e_i})\Gamma = \gamma(\vec{n})\Gamma$ for some $Q \ll M^{O_{s,m,\vec{d}}(1)}$. Setting $t = 1$, we recover Theorem \ref{mainthm-prelim}.

We leave the straightforward deduction of Theorem \ref{ratner-result} to the reader.

\appendix

\section{Facts about coordinates and Mal'cev bases}\label{malcev-app}

Let us begin this appendix by discussing coordinate systems on a connected, simply-connected nilpotent Lie group $G$ of dimension $m$. A discrete and cocompact subgroup $\Gamma$, leading to a nilmanifold $G/\Gamma$, will be introduced in a little while. Let $\g$ be the Lie algebra of $G$, and let $\exp : \g \rightarrow G$ and $\log : G \rightarrow \g$ be the exponential and logarithm maps, which are both diffeomorphisms. In this appendix all implied constants are allowed to depend on $m$ and $s$, and for notational brevity this dependence will usually be suppressed. The rationality parameter $Q$ will always be assumed to be at least $2$.

Let us begin by recalling from \S \ref{sec2} the notion of coordinates of the first and second kinds. 

\begin{definition}[Coordinates]
Let $\X = \{X_1,\dots,X_m\}$ be a basis for $\g$. If
\[ g = \exp(t_1X_1 + \dots + t_m X_m)\] then we say that $(t_1,\dots,t_m)$ are the coordinates of the \emph{first kind} or \emph{exponential coordinates} for $g$ relative to the basis $\X$. We write $(t_1,\dots,t_m) = \psi_{\X,\exp}(g)$.
If
\[ g = \exp(u_1X_1)\dots \exp(u_m X_m)\] then we say that $(u_1,\dots,u_m)$ are the coordinates of the \emph{second kind} for $g$ relative to $\X$, and we write $(u_1,\dots,u_m) = \psi_\X(g)$.\end{definition}

From now on in this appendix (as in the main text) we will write $\psi := \psi_{\X}$ and $\psi_{\exp} := \psi_{\X,\exp}$. When another basis $\X'$ for some Lie algebra $\g'$ is present we shall write $\psi' := \psi_{\X'}$ and $\psi'_{\exp} := \psi_{\X,\exp}$.

Recall that $\X$ is said to be $Q$-rational if all the structure constants $c_{ijk}$ in the relations 
\[ [X_i,X_j] = \sum_k c_{ijk} X_k\] are rationals of height at most $Q$.

The effect of a change of basis is easily understood in coordinates of the first kind (indeed, it merely effects a linear transformation of coordinates). Nilmanifolds, however, are best studied using coordinates of the second kind. It is, therefore, no surprise that the following lemma describing the passage between the two types of coordinate system is very useful.

\begin{lemma}[Coordinates of the first and second type]\label{1-2-lem}\textup{(i)} Let $\X$ be a basis for $\g$ with the nesting property that
\begin{equation}\label{nest} [\g, X_i] \subseteq \Span( X_{i+1},\dots,X_m) \end{equation} for $i = 1,\dots,m-1$. Then the compositions $\psi_{\exp}\circ \psi^{-1}$ and $\psi \circ \psi_{\exp}^{-1}$ are both polynomial maps on $\R^m$ with degree $O(1)$. If $\mathcal{X}$ is $Q$-rational then all the coefficients of these polynomials are rational of height at most $Q^{O(1)}$.

\textup{(ii)} Suppose that $G' \subseteq G$ is a closed, connected subgroup of dimension $m'$ with associated Lie algebra $\g' \subseteq \g$. Suppose $\X'$ is a basis for $\g'$ with the nesting property. Then $\psi \circ \psi^{\prime -1}$ is a polynomial map from $\R^{m'}$ to $\R^m$ and $\psi'\circ \psi^{-1}$ is a polynomial map from $\psi(G') \subseteq \R^m$ to $\R^{m'}$. Both of these maps have degree $O(1)$. If $\X$ and $\X'$ are $Q$-rational and if each element $X'_i$ of $\mathcal{X}'$ is a $Q$-linear combination of the $X_i$ then all coefficients of these polynomials are rationals of height $Q^{O(1)}$.
\end{lemma}
\proof (i) Recall the Baker-Campbell-Hausdorff formula, which states that 
\[ \log(\exp(X)\exp(Y)) =X + Y + \frac{1}{2}[X,Y] + \frac{1}{12}[X,[X,Y]] - \frac{1}{12}[Y,[X,Y]] + \dots,\]
this expression being a sum of $O_s(1)$ terms, each of which is a rational number of height $O_s(1)$ times a commutator of order at most $s$ involving $X$s and $Y$s. Repeated use of this allows us to write $\exp(u_1X_1) \dots \exp(u_mX_m)$ in the form $\exp(t_1 X_1 + \dots + t_m X_m)$. Property \eqref{nest} is easily seen to imply that the $t_i$ are polynomials in the $u_i$ with the specific form
\begin{align}\nonumber t_1 &= u_1 \\ \nonumber t_2 &= u_2 + P_2(u_1) \\ \nonumber t_2 &= u_3 + P_3(u_1,u_2) \\ \nonumber & \dots \\ \label{polys} t_m &= u_m + P_m(u_1,\dots,u_{m-1}).\end{align}

This establishes the claim for $\psi_{\exp} \circ \psi^{-1}$. To prove the result for $\psi \circ \psi^{-1}_{\exp}$ we simply note that the relations \eqref{polys} are of an ``upper triangular'' form which is easy to invert. Thus the $u_i$ are given in terms of the $t_i$ by polynomial relations of a similar upper triangular form. The quantitative statements follow by the same arguments, keeping track of the heights of the rational numbers involved. We leave the details to the reader.

(ii) Note the decomposition \[ \psi \circ \psi^{\prime -1} = (\psi \circ \psi_{\exp}^{-1}) \circ (\psi_{\exp} \circ \psi^{\prime -1}_{\exp}) \circ (\psi'_{\exp} \circ \psi^{\prime -1}) .\]
Of the three maps here, the first one is a polynomial map from $\R^m$ to $\R^m$ by (i), and the third is a polynomial map from $\R^{m'}$ to $\R^{m'}$. The middle map is simply a linear transformation from $\R^{m'}$ to $\R^m$.

The composition $\psi^{\prime} \circ \psi^{-1}$ may be dealt with in a very similar manner.

Once again the quantitative claims follow by the same arguments, keeping track of heights. We leave the details to the reader.
\endproof

The upper-triangular form of the relations \eqref{polys} allows us to prove the following key result, which describes group multiplication and inversion in coordinates.

\begin{lemma}[Multiplication and inversion in coordinates]\label{mult-basis-lem}
Let $\X$ be a basis for $\g$ with the nesting property \eqref{nest}. Let $x,y \in G$, and suppose that $\psi(x) = t$ and $\psi(y) = u$. Then
\begin{align*} &\psi(xy) =  \\ & (t_1 + u_1, t_2 + u_2 + P_1(t_1,u_1), ,\dots, t_m + u_m + P_{m-1}(t_1,\dots,t_{m-1},u_1,\dots,u_{m-1})),\end{align*} where, for each $i = 1,\dots,m-1$, $P_i : \R^{i} \times \R^{i} \rightarrow \R$ is a polynomial of degree $O(1)$. Furthermore
\[ \psi(x^{-1}) = (-t_1 , -t_2 + \tilde P_1(t_1),\dots,-t_m + \tilde P_{m-1}(t_1,\dots,t_{m-1}))\] where $\tilde P_i : \R^{i} \rightarrow \R$ is a polynomial of degree $O(1)$.
Let $Q \geq 2$. If $\X$ is $Q$-rational then all the coefficients of the polynomials $P_i,\tilde P_i$ are rationals of height $Q^{O(1)}$. 
\end{lemma}
\proof By \eqref{polys} we know that 
\[ \psi_{\exp}(x) = (t_1 , t_2 + R_1(t_1),\dots,t_m + R_{m-1}(t_1,\dots,t_{m-1}))\] and similarly for $\psi_{\exp}(y)$, where $R_i : \R^{i} \rightarrow \R$ is a polynomial for $i = 1,\dots,m-1$. It follows from the Baker-Campbell-Hausdorff formula and the nesting property \eqref{nest} that 
\begin{align*} & \psi_{\exp}(xy) = \\ & (t_1 + u_1 , t_2 + u_2 + S_1(t_1,u_1), \dots, t_m + u_m + S_{m-1}(t_1,\dots,t_{m-1},u_1,\dots,u_{m-1})),\end{align*} where each $S_i : \R^{i} \times \R^{i} \rightarrow \R$ is again polynomial. The statement about the form of $\psi(xy)$ now follows from a further application of the relations \eqref{polys}, and the statement about $\psi(x^{-1})$ is an immediate corollary of it.

To obtain the quantitative versions of these statements we use the same arguments, keeping track of the heights of the rational numbers involved. We leave the details to the reader.\endproof

Recall at this point Definition \ref{metric-def}, in which a basis $\X$ is used to define metric $d = d_\X$ on $G$. We defined $d$ to be the largest metric such that $d(x,y) \leq |\psi(xy^{-1})|$ for all $x,y \in G$, where $|\cdot |$ denotes the $\ell^{\infty}$-norm on $\R^m$. For practical purposes it is important to have an understanding of such metrics in terms of the coordinates $\psi(x)$ and $\psi(y)$, or even in terms of coordinates $\psi'(x),\psi'(y)$ relative to some other basis $\X'$. The following lemma provides some information in this regard. Here, and in the rest of this appendix, we write $d := d_{\X}$ and $d' := d_{\X'}$.

\begin{lemma}[Bounds for $d$ in terms of coordinates]\label{dx-bounds}
Suppose that $Q \geq 2$. Suppose that $\X,\X'$ are two $Q$-rational bases for $\g$, both satisfying the nesting condition \eqref{nest}. Suppose that each $X'_i$ is given by a $Q$-rational combination of the $X_i$ and vice versa. Then for all $x,y \in G$ with $|\psi'(x)|, |\psi'(y)| \leq Q$ we have the bound
\begin{equation}\label{d-psi-bound-1} d(x,y) \ll Q^{O(1)} |\psi'(x) - \psi'(y)|,\end{equation}
and for all $x,y \in G$ with $d(x,\id_G), d(y,\id_G) \leq Q$ we have the bound
\begin{equation}\label{d-psi-bound-2} |\psi'(x) - \psi'(y)| \ll Q^{O(1)} d(x,y).\end{equation}
\end{lemma}
\proof Inequality \eqref{d-psi-bound-1} is by far the easier of the two inequalities claimed here and we prove it first. By definition we have $d(x,y) \leq |\psi(xy^{-1})|$. Write $\psi'(x) = t$ and $\psi'(y) = u$; by Lemmas \ref{1-2-lem} and \ref{mult-basis-lem} we see that the coordinates $\psi(xy^{-1})$ are \[ (P_1(t,u),\dots,P_m(t,u)),\] where each $P_i : \R^m \times \R^m \rightarrow \R$ is a polynomial of degree $O(1)$ whose coefficients are rationals of height $Q^{O(1)}$. Each of these polynomials of course vanishes when $t = u$, and so we can write (e.g.)
\[ P_1(t,u) = P_1(t,u) - P_1(t,t) = \sum_{i = 1}^m (t_i - u_i)R_{1,i}(t,u),\] where each $R_{1,i} : \R^m \times \R^m \rightarrow \R$ is a polynomial of degree $O(1)$ whose coefficients are rationals of height $Q^{O(1)}$. (One way to see this is to expand $P_1$ as a sum of monomials $t^{\vec{\alpha}}u^{\vec{\beta}}$.) The bound \eqref{d-psi-bound-1} follows immediately.

The second bound, \eqref{d-psi-bound-2}, is significantly more difficult. We begin by proving the special case in which $\mathcal{X} = \mathcal{X}$' and $y = \id_G$, or in other words the following claim: 
\begin{equation} \label{special-claim}\mbox{$|\psi(x)| \ll Q^{O(1)} d(x,\id_G)$ uniformly for all $x$ with $d(x,\id_G) \leq Q$.}\end{equation} Write $\kappa(x,y) := \min(|\psi(xy^{-1})|, |\psi(yx^{-1})|)$.  
We will use the bound
\begin{equation}\label{psi-x-y} |\psi(x) - \psi(y)| \ll Q^{O(1)}\kappa(x,y)(1 + \kappa(x,y) + |\psi(y)|)^{O(1)}.\end{equation}
To prove this when $\kappa(x,y) = |\psi(xy^{-1})|$ we proceed much as in the proof of \eqref{observed-bound}: set $x = zy$ and use Lemma \ref{mult-basis-lem} to expand $\psi(x) - \psi(y) = \psi(zy) - \psi(y)$ as a polynomial in the coordinates of $v = \psi(y)$ and $w = \psi(z)$ which vanishes when $w = 0$. When $\kappa(x,y) = |\psi(yx^{-1})|$ we proceed similarly, setting $x = yz^{-1}$. 

From \eqref{psi-x-y} we see in particular that if $|\psi(y)| \leq 1$ and $\kappa(x,y) \leq 1$, then
$$ |\psi(x)| \leq |\psi(y)| + C Q^C \kappa(x,y)$$
for some constant $C \geq 1$.  Iterating this we see that if $x_0,\ldots,x_n$ are elements of $G$ with $x_0 = \id_G$ and $\kappa(x_0,x_1) + \ldots + \kappa(x_{n-1},x_n) \leq C^{-1} Q^{-C}$ then
$$ |\psi(x_n)| \leq C Q^C ( \kappa(x_0,x_1) + \ldots + \kappa(x_{n-1},x_n) ).$$
Inspecting the definition of $d$, we conclude that
\begin{equation}\label{psix}
 |\psi(x)| \ll Q^{O(1)} d(x,\id_G) \hbox{ whenever } d(x,\id_G) \leq C^{-1} Q^{-C}.
\end{equation}
By right-invariance and symmetry of $d$, we can amplify this to
\begin{equation}\label{kapxy}
 |\kappa(x,y)| \ll Q^{O(1)} d(x,y) \hbox{ whenever } d(x,y) \leq C^{-1} Q^{-C}.
\end{equation}

The estimate \eqref{psix} is almost what we need, except that the bound on $d(x,\id_G)$ is too strict.  To relax it, we argue as follows.  To obtain \eqref{special-claim}, it suffices to show that
$$ |\psi( x_n )| \ll Q^{O(1)} ( \kappa(x_0,x_1) + \ldots + \kappa(x_{n-1},x_n) )$$
whenever $x_0,\ldots,x_n \in G$ with $x_0 = \id_G$ and $\kappa(x_0,x_1) + \ldots + \kappa(x_{n-1},x_n) \leq 2Q$ (say).

Using a greedy algorithm, split the path $(x_0,\ldots,x_n)$ into $O(Q^{O(1)})$ paths $(x_i,\ldots,x_j)$ with $\kappa(x_i,x_{i+1}) + \ldots + \kappa(x_{j-1},x_j) \leq C^{-1}Q^{-C}$, plus $O(Q^{O(1)})$ singleton paths $(x_i, x_{i+1})$ with $C^{-1} Q^{-C} \leq \kappa(x_i, x_{i+1}) \leq 2Q$.  Applying \eqref{kapxy}, we thus see that there exists a path $(y_0,\ldots,y_r)$ with $r = O(Q^{O(1)})$, $y_0 = \id_G$, and $y_r = x_n$, such that $\kappa(y_i,y_{i-1}) \ll Q^{O(1)}$ for all $1 \leq i \leq r$.  In particular (using Lemma \ref{mult-basis-lem}) if we write $g_i := y_i y_{i-1}^{-1}$ for $1 \leq i \leq r$, then we see that $|\psi(g_i)| \ll Q^{O(1)}$.  On the other hand, we have the telescoping product
$$ x_n = g_r \ldots g_1.$$
Now if $g_1,\dots,g_r \in G$ are any elements with $|\psi(g_i)| \leq t$ for all $i$ then \[ |\psi(g_1\dots g_r)| \ll (1 + t)^{O(1)}r^{O(1)}.\] This may be seen by
applying Lemma \ref{mult-basis-lem} repeatedly to expand the product out completely in coordinates. That the first coordinate is polynomially controlled is obvious, and it then follows that the second is also, and so on inductively.
Applying this in the present situation gives $|\psi(x_n)| \ll Q^{O(1)}$, and similar arguments for each $i$ give that in fact $|\psi(x_i)| \ll Q^{O(1)}$ uniformly for $0 \leq i \leq n$.  Applying \eqref{psi-x-y} we have
$$ |\psi(x_i)| \leq |\psi(x_{i-1})| + O( Q^{O(1)} \kappa(x_{i-1},x_i) )$$
and \eqref{special-claim} follows.

We have just established the special case $\mathcal{X} = \mathcal{X}'$, $y = \id_G$ of \eqref{d-psi-bound-2}. We now deal with the case where $\mathcal{X} = \mathcal{X}'$ but $y$ is arbitrary. Suppose then that $d(x,\id_G), d(y,\id_G) \leq Q$. Applying \eqref{special-claim} we see that $|\psi(x)|,|\psi(y)| \ll Q^{O(1)}$. By Lemma \ref{mult-basis-lem} we therefore have $|\psi(xy^{-1})| \ll Q^{O(1)}$, and hence by \eqref{d-psi-bound-1} it follows that $d(xy^{-1},\id_G) \ll Q^{O(1)}$. Applying \eqref{special-claim} once more, we see that \[ |\psi(xy^{-1})| \ll Q^{O(1)}d(xy^{-1},\id_G),\] which, since $d$ is right-invariant, implies that \begin{equation}\label{d-psi-bd-3}|\psi(xy^{-1})| \ll Q^{O(1)} d(x,y).\end{equation} The claimed result now follows immediately using \eqref{psi-x-y}.

Finally we turn to the general case in which $\mathcal{X}$ and $\mathcal{X}'$ may be different. We start with the special case of \eqref{d-psi-bound-2} just proved, namely
\begin{equation}\label{spec} |\psi(x) - \psi(y)| \ll Q^{O(1)}d(x,y).\end{equation}
Applying \eqref{d-psi-bound-1} we obtain
\[ d'(x,y) \ll Q^{O(1)}|\psi(x) - \psi(y)| \ll Q^{O(1)}d(x,y).\]
In particular we have $d'(x,\id_G), d'(y,\id_G) \ll Q^{O(1)}$. A second application of \eqref{spec}, with $\mathcal{X}$ replaced by $\mathcal{X}'$, then gives
\[ |\psi'(x) - \psi'(y)| \ll Q^{O(1)}d'(x,y) \ll Q^{O(1)}d(x,y).\] This concludes the proof of Lemma \ref{dx-bounds}.\endproof

The metric $d$ is right-invariant, that is to say $d(xg,yg) = d(x,y)$ for all $x,y,g \in G$. It is useful to have, in addition, the following approximate left-invariance property.

\begin{lemma}[Approximate left-invariance of $d$]\label{approx-left}
Suppose that $Q \geq 2$ and that $\X$ is a $Q$-rational basis for $\g$ satisfying the nesting condition \eqref{nest}. Suppose that $g,x,y \in G$ are elements with $|\psi(x)|,|\psi(y)|,|\psi(g)| \leq Q$. Then we have the bound
\[ d(gx,gy) \ll Q^{O(1)}d(x,y).\]
\end{lemma}
\proof We start by observing that uniformly in $g,z \in G$ we have the bound
\begin{equation}\label{observed-bound} |\psi(gzg^{-1})| \ll Q^{O(1)} (1 + |\psi(z)| + |\psi(g)|)^{O(1)}|\psi(z)|.\end{equation}
This follows by using Lemma \ref{mult-basis-lem} to conclude that the components of $\psi(gzg^{-1})$ are polynomials of degree $O(1)$ with $Q^{O(1)}$-rational coefficients in the coordinates $v = \psi(g)$ and $w = \psi(z)$, and these polynomials all vanish when $w = 0$. 
Recall from Definition \ref{metric-def} that 
\begin{equation}\label{d-def-repeat} d(x,y) = \inf\{ \sum_{i=0}^{n-1} \min(|\psi(x_{i-1}x_i^{-1})|, |\psi(x_ix_{i-1}^{-1})|): x_0,\dots,x_n \in G; x_0 = x; x_n = y  \}.\end{equation} We see, then, that the lemma will follow from \eqref{observed-bound} (taking $z = x_ix_{i-1}^{-1}$ or $x_{i-1}x_i^{-1}$) if we can show that the infimum may be taken over all those $x_i,x_{i-1}$ which satisfy some bound $\min(|\psi(x_{i-1}x_i^{-1})|,|\psi(x_ix_{i-1}^{-1})|) \ll Q^{O(1)}$.
But this follows from the inequality $d(x,y) \ll Q^{O(1)}$, which is an instant consequence of Lemma \ref{dx-bounds}.\endproof

We conclude this subsection by recording the following result.

\begin{lemma}[Comparison lemma]\label{comparison-lemma}
Suppose that $G' \subseteq G$ is a closed subgroup and that $\X,\X'$ are bases for $\g,\g'$ respectively which have the nesting property \eqref{nest}. Let $Q \geq 2$, and suppose that each $X'_i$ is a $Q$-rational combination of the $X_i$. Then we have the bounds
\[ d'(x,y) \ll Q^{O(1)} d(x,y)\] uniformly for all $x,y \in G'$ with $|\psi(x)|,|\psi(y)|\leq Q$ and
\[ d(x,y) \ll Q^{O(1)} d'(x,y)\] uniformly for all $x,y \in G'$ with $|\psi'(x)|, |\psi'(y)| \leq Q$.
\end{lemma}
\proof We follow essentially the same argument used in the previous lemma. To prove the first bound, for example, replace \eqref{observed-bound} with the bound
\[ |\psi'(z)| \ll Q^{O(1)}(1 + |\psi(z)|)^{O(1)} |\psi(z)|.\] This follows immediately from Lemma \ref{1-2-lem} (ii), which guarantees that $\psi'(z)$ is a polynomial in the coordinates $\psi(z)$ which vanishes when $\psi(z) = 0$.\endproof

\textsc{Mal'cev bases.} Suppose that $G$ is a connected, simply-connected nilpotent Lie group with a filtration $G_{\bullet}$. Let us now introduce a discrete and cocompact subgroup $\Gamma$ to the discussion. Throughout the paper we have assumed that $G/\Gamma$ comes together with a special type of basis $\X$ called a \emph{Mal'cev basis adapted to $G_{\bullet}$}, which is invoked whenever it is necessary to discuss the metric structure of $G/\Gamma$.

Let us recall from \S \ref{sec2} the basic properties of these bases:
\begin{enumerate}
\item For each $j = 0,\dots,m-1$ the subspace $\h_j := \Span(X_{j+1},\dots,X_m)$ is a Lie algebra ideal in $\g$, and hence $H_j := \exp \h_j$ is a normal Lie subgroup of $G$.
\item For every $i$, $0 \leq i \leq s$, we have $G_i = H_{m - \dim(G_i)}$ (or equivalently, $\g_i = \h_{m-\dim(\g_i)}$);
\item Each $g \in G$ can be written uniquely as $\exp(t_1X_1)\dots \exp(t_mX_m)$, for $t_1,\ldots,t_m \in \R$.
\item $\Gamma$ consists precisely of those elements which, when written in the above form, have all $t_1,\ldots,t_m \in \Z$.
\end{enumerate}

Mal'cev bases are not especially flexible in certain ways -- for example it is not at all easy to take a Mal'cev basis on $G/\Gamma$ and use it to construct one on $G^{\Box}/\Gamma^{\Box}$ as we had to do in the proof of Lemma \ref{rat-bounds}. For additional flexibility it is convenient to introduce the notion of a \emph{weak} basis for $G/\Gamma$. These are only ever used in the process of constructing actual Mal'cev bases with desirable properties.

\begin{definition}[Weak bases]\label{weak-basis-def}
Let $\X = \{X_1,\dots,X_m\}$ be a basis for $\g$. Let $Q \geq 2$ be a parameter. We say that $\X$ is a \emph{$Q$-rational weak basis} for $G/\Gamma$ if $\mathcal{X}$ is $Q$-rational (cf. Definition \ref{ratnil-def}) and if we have $\frac{1}{q}\Z^m \supseteq \psi_{\exp}(\Gamma) \supseteq q\Z^m$ for some $q \leq Q$, that is to say the coordinates of $\log \Gamma$ relative to $\X$ are close to being integers.
\end{definition}
Note carefully that $\log \Gamma$ is not necessarily a subgroup of $\g$, as we saw in \S \ref{heisenberg-sec} in connection with the Heisenberg example.

We record some simple facts about weak bases.

\begin{lemma}[Weak bases: simple facts] \label{weak-triv} Weak bases enjoy the following properties.
\begin{enumerate}
\item Suppose that $\X$ is a $Q$-rational weak basis for $G/\Gamma$, and that $\X' = \{X'_1,\dots,X'_m\}$ is another basis for $\g$ with the property that each $X'_i$ is a $Q$-rational combination of the $X_i$. Then $\X'$ is a $Q^{O(1)}$-rational weak basis for $G/\Gamma$. 
\item Suppose that $\X$ is a Mal'cev basis adapted to some subgroup sequence $G_{\bullet}$, that is to say conditions \textup{(i)}, \textup{(ii)}, \textup{(iii)} and \textup{(iv)} from the start of the section are satisfied. Suppose that $\X$ is $Q$-rational. Then $\X$ is an $O(Q^{O(1)})$-rational weak basis for $G/\Gamma$.
\end{enumerate}
\end{lemma}
\proof Part (i) is immediate. Part (ii) follows quickly from Lemma \ref{1-2-lem}.\endproof

The next proposition allows us to construct Mal'cev bases from weak bases. If $\mathcal{X}$ is a Mal'cev basis for $G/\Gamma$ and if $G' \subseteq G$ is a subgroup, we say that $G'$ is $Q$-\emph{rational} if the Lie algebra $\g'$ is generated by $Q$-rational combinations of the basis elements $X_i$.

\begin{proposition}[Construction of Mal'cev bases]\label{malcev-exist}
Suppose that $\X$ is a $Q$-rational weak basis for $G/\Gamma$ and that $G_{\bullet}$ is a filtration in which each subgroup $G_i$ is $Q$-rational. Then there is a Mal'cev basis $\X' = \{X'_1,\dots,X'_m\}$ for $G/\Gamma$ adapted to $G_{\bullet}$ in which each $X'_i$ is a $Q^{O(1)}$-rational combination of the basis elements $X_i$. In particular, the Mal'cev basis $\X'$ is $Q^{O(1)}$-rational.
\end{proposition}
\proof Take a basis for $\g_d$ consisting of $Q$-rational linear combinations of the $X_i$. By straightforward linear algebra this may be extended to a basis of $\g_{d-1}$ consisting of $Q^{O(1)}$-rational combinations of the $X_i$. This in turn may be extended to a basis of $\g_{d-2}$ and so on. In this fashion we obtain a basis $\mathcal{Y} = \{Y_1,\dots,Y_m\}$ for $\g$ as a vector space consisting of $Q^{O(1)}$-rational combinations of the $X_i$ such that each $\g_{i}$ equals $\Span(Y_{j+1},\dots,Y_m)$ where $j = m - m_i$. By Lemma \ref{weak-triv} (i) we see that $\mathcal{Y}$ is a $Q^{O(1)}$-rational weak basis for $G/\Gamma$. 

Since $[\g,\g_{i}] \subseteq \g_{i+1}$ for all $i$ we see that the weak basis $\mathcal{Y}$ enjoys the nesting property, that is to say $[\g,Y_j] \subseteq \Span( Y_{j+1},\dots,Y_m)$ for all $j$. 

We now convert this basis $\mathcal{Y}$ into the desired Mal'cev basis by choosing $X'_m = c_mY_m,\dots,$ $X'_1 = c_1Y_1$ in turn so that 
\begin{equation}\label{to-verify} \Span( Y_{i+1},\dots,Y_m) \cap \Gamma = \{\exp(n_{i+1}X'_{i+1}) \dots \exp(n_mX'_m) : n_{i+1},\dots,n_m \in \Z\}\end{equation} for $i = m-1,\dots,0$. Such a basis $\X'$ has all of the properties (i), (ii), (iii) and (iv) required to qualify as a Mal'cev basis.
Suppose this is done for $i = j$. Since $\mathcal{Y}$ is a $Q^{O(1)}$-rational weak basis for $G/\Gamma$ we see that
\[ \big(\Span( Y_{j},\dots,Y_m) \cap \Gamma\big)/\Span( Y_{j+1},\dots,Y_m)\]
is generated by $\overline{\exp(c_jY_j)}$ for some $c_j \in \Q$ with heights bounded by $Q^{O(1)}$.
Taking $X'_j := c_jY_j$, we see that \eqref{to-verify} holds for $i = j-1$ too.\endproof

For applications (for example in the proof of Lemma \ref{rat-bounds}) it is convenient to have the following variant of the above proposition.

\begin{proposition}[Mal'cev bases of subnilmanifolds]\label{sub-nil-basis}
Suppose that $\X = \{X_1,\dots,X_m\}$ is a $Q$-rational Mal'cev basis for $G/\Gamma$ adapted to a filtration $G_{\bullet}$. Suppose that $G' \subseteq G$ is a $Q$-rational subgroup of $G$, and furthermore that $G'_{\bullet}$ is a filtration on $G'$ in which each of the groups $G'_i$ is $Q$-rational \textup{(}with respect to the basis $\X$\textup{)}. Write $\Gamma' := \Gamma \cap G'$. Then $G'/\Gamma'$ has a Mal'cev basis $\X' = \{X'_1,\dots,X'_{m'}\}$ adapted to $G'_{\bullet}$ in which each $X'_i$ is a $Q^{O(1)}$-rational combination of the $X_i$.
\end{proposition}
\proof One simply observes that by linear algebra there is a basis $\mathcal{Y} = \{Y_1,\dots,Y_{m'}\}$ for $\g'$ together with an extension $\tilde{\mathcal{Y}} = \{Y_1,\dots,Y_m\}$ to a basis for $\g$ such that each of the $Y_i$ is a $Q^{O(1)}$-rational combination of the $X_i$. By Lemma \ref{weak-triv}, $\tilde{\mathcal{Y}}$ is a weak basis for $G/\Gamma$, and therefore $\mathcal{Y}$ is a weak basis for $G'/\Gamma'$. The result now follows from Proposition \ref{malcev-exist} applied to this weak basis.\endproof

\textsc{Rationality.}  We now record some simple results about rational points in nilmanifolds $G/\Gamma$. Recall Definition \ref{rat-def}: $g \in G$ is \emph{rational} if $g^r \in \Gamma$ for some integer $r > 0$. Recall also the quantitative version of this, Definition \ref{rat-def-quant}: $g \in G$ is $Q$-\emph{rational} if $g^r \in \Gamma$ for some integer $r$, $0 < r \leq Q$.

\begin{lemma}[Properties of rational points]\label{ratpoint} Suppose that $\X$ is a $Q$-rational Mal'cev basis adapted to some subgroup sequence $G_{\bullet}$, where $Q \geq 2$.
\begin{enumerate}
\item If $\gamma \in G$, then $\gamma$ is rational if and only if $\psi(\gamma) \in \Q^m$.
\item The set of rational points in $G$ is a group.
\item If $\gamma \in G$ is $Q$-rational, then $\psi(\gamma) \in \frac{1}{Q'} \Z^m$ for some $Q'$, $1 \leq Q' \ll Q^{O(1)}$, which does not depend on $\gamma$.
\item If $\gamma \in G$ is such that $\psi(\gamma) \in \frac{1}{Q} \Z^m$, then $\gamma$ is $O(Q^{O(1)})$-rational.
\item If $\gamma, \gamma'$ are $Q$-rational, then $\gamma\gamma'$ and $\gamma^{-1}$ are $O(Q^{O(1)})$-rational.
\end{enumerate}
\end{lemma}

\proof If $\gamma$ is rational, then by definition there exists $r \geq 1$ such that $\gamma^r \in \Gamma$, and thus $\psi( \gamma^n ) \in \Z^m$ whenever $n$ is a multiple of $r$. Now from Lemma \ref{basis-description} we know that the coordinates $\psi(g^n)$ are all polynomials of degree $O(1)$; these vanish at zero, and take integer values at multiples of $r$. By the Lagrange interpolation formula we conclude that all the coefficients of these polynomials are rational, and so in particular we have $\psi(\gamma) \in \Q^m$. 

Suppose conversely that $\psi(\gamma) \in \Q^m$. Then by Lemma \ref{mult-basis-lem} we see that each of $\psi(\gamma^2)$, $\psi(\gamma^3),\dots$ also lies in $\Q^m$. By another application of Lemma \ref{basis-description} and the Lagrange interpolation formula we conclude that each coordinate of $\psi(\gamma^n)$ is a polynomial with rational coefficients which vanishes at zero. In particular it is easy to see that by choosing $r \in \N$ suitably we may ensure that $\psi(\gamma^r) \in \Z^m$, which of course implies that $\gamma^r \in \Gamma$.

Part (ii) follows immediately from (i) and Lemma \ref{mult-basis-lem}.
 
Claims (iii)-(v) follow by repeating the above arguments, but keeping track of the heights of all the rational numbers involved; the key point is that the group operations, as well as Lagrange interpolation, are all polynomial in nature and so all heights will be $O(Q^{O(1)})$.  We omit the routine details.
\endproof

Let us now recall the notion of a rational sequence, also given in Definition \ref{rat-def}. A sequence $\gamma : \Z \rightarrow G$ is rational if $\gamma(n)\Gamma$ is rational for all $n$, and it is $Q$-rational if $\gamma(n)\Gamma$ is rational for all $n$. The next lemma records some useful properties of rational \emph{polynomial} sequences.

\begin{lemma}[Properties of rational polynomial sequences]\label{rat-poly-lem} 
Suppose that $\gamma : \Z \rightarrow G$ is a polynomial sequence of degree $d$. 
\begin{enumerate}
\item Suppose that $\gamma$ is rational. Then $\gamma(n)\Gamma$ is periodic.
\item Suppose that there is a $Q$-rational Mal'cev basis $\X$ for $G/\Gamma$ and that $\gamma$ is $Q$-rational. Then $\gamma(n)\Gamma$ is periodic with period $\ll Q^{O(1)}$.
\end{enumerate}
\end{lemma}
\proof (i). Let $\X$ be any Mal'cev basis for $G/\Gamma$. By Lemma \ref{basis-description} the coordinates $\psi(\gamma(n))$ are all polynomials of degree $O(1)$, and by the previous lemma and the Lagrange interpolation formula they all have rational coefficients. Clearing denominators, we thus find some $q$ such that $\psi(\gamma(n)) \in \frac{1}{q}\Z^m$ for all integers $n$. By Lemma \ref{mult-basis-lem} we see that there is some $q' \in \N$ such that, for any $r \in \Z$, we have $\psi(\gamma(n+r)\gamma(n)^{-1}) \in \frac{r}{qq'}\Z^m$. Thus $\gamma(n)\Gamma$ is indeed periodic, with period $qq'$. 

Part (ii) is proved in exactly the same way, once again taking care to keep track of the heights of all rationals involved. \endproof

We leave the formulation and proof of the multidimensional version of this lemma (that is, concerning maps $\gamma : \Z^t \rightarrow G$) to the reader; only trivial modifications are required.

The next result, stating that conjugates of rational subgroups by rational elements are rational, is not needed in the present paper. It is required in the companion paper \cite{ukmobius}.

\begin{lemma}[Rational conjugates]
Suppose that $\mathcal{X} = \{X_1,\dots,X_m\}$ is a $Q$-rational Mal'cev basis for $G/\Gamma$ adapted to some filtration. Suppose that $\gamma \in G$ is $Q$-rational and additionally that the coordinates $\psi(\gamma)$ are all bounded in magnitude by $Q$. Suppose that $G' \subseteq G$ is a $Q$-rational subgroup. Then the conjugate $\gamma G'\gamma^{-1}$ is $Q^{O(1)}$-rational.
\end{lemma}
\proof Set $H := \gamma G' \gamma^{-1}$ and let $\h$ be the corresponding Lie algebra. Recall from basic Lie theory the identity
\[ \log (\gamma \exp (X) \gamma^{-1}) = \Ad(\gamma) X,\] where $\Ad(\gamma) : \g \rightarrow \g$ is the adjoint automorphism of $\g$ associated to the element $\gamma \in G$. For the purposes of this argument all we need is the following immediate consequence of this identity: if $X'_1,\dots,X'_{m'}$ is a basis for the Lie algebra $\g'$ then the elements 
\[ \tilde X_i := \log(\gamma \exp(X'_i) \gamma^{-1})\]
are a basis for $\h$.  By assumption we may choose the $X'_i$ to be $Q$-rational combinations of the $X_i$. It then follows from Lemmas \ref{1-2-lem} and \ref{mult-basis-lem} that each $\tilde X_i$ is a $Q^{O(1)}$-rational combination of the $X_i$.

\textsc{Fundamental domain and reduction.} The next lemma provides a description of $G/\Gamma$ in terms of coordinates relative to any Mal'cev basis $\X$.

\begin{lemma}[Reducing to the fundamental domain]\label{fund-dom-def} Let $\X$ be a Mal'cev basis adapted to some subgroup sequence $G_{\bullet}$.
Suppose that $g \in G$. Then we may write $g = \{g\}[g]$ in a unique way, where $\psi(\{g\}) \in [0,1)^m$ and $[g] \in \Gamma$. 
\end{lemma}
\proof Recall Lemma \ref{mult-basis-lem}, which describes the multiplication on $G$ in coordinates relative to $\X$. Using this we may iteratively construct $\gamma_m,\gamma_{m-1},\dots,\gamma_1 \in \Gamma$ in such a way that coordinates $i+1,\dots,m$ of $\psi(g\gamma_m \dots \gamma_{i})$ all lie in the interval $[0,1)$.

The uniqueness also follows easily from Lemma \ref{mult-basis-lem}: if $\psi(x\gamma),\psi(x) \in [0,1)^m$ then we may equate coefficients of $\psi(\gamma)$ starting at the right to deduce that $\gamma = \id_G$.\endproof

\textsc{Metrics on nilmanifolds.} Let $\mathcal{X}$ be a Mal'cev basis for some nilmanifold $G/\Gamma$. Recall from Definition \ref{metric-def} the manner in which we used the metric $d = d_{\mathcal{X}}$ on $G$ to define a ``metric'' on $G/\Gamma$ via
\[ d(x\Gamma, y\Gamma) = \inf_{\gamma,\gamma' \in \Gamma} d(x\gamma,y\gamma').\] We can now prove that $d$ really is a metric on $G/\Gamma$ (and thus the inverted commas above can be dispensed with).

\begin{lemma}[Nondegeneracy of metric]\label{quotient-metric}
Suppose that $\mathcal{X}$ is a rational Mal'cev basis for a nilmanifold $G/\Gamma$, adapted to some filtration. Suppose that $d(x\Gamma,y\Gamma) = 0$. Then $x \equiv y \mdlem{\Gamma}$.
\end{lemma}
\proof Since the metric $d$ on $G$ is right-invariant we have
\[ d(x\Gamma,y\Gamma) = \inf_{\gamma \in \Gamma} d(x,y\gamma).\] It suffices to show that the $\inf$ here is a actually a minimum, to which end we need only show that for any $M$ there are just finitely many $\gamma \in \Gamma$ with $d(x,y\gamma) \leq M$. By Lemma \eqref{approx-left} this assumption implies that $d(y^{-1}x,\gamma) \leq M'$, for some $M'$ depending on $M$, the rationality of the Mal'cev basis $\mathcal{X}$ and the size of the coordinates of $x$ and $y$. This in turn implies that $d(\id_G, \gamma) \leq M''$ which, in view of Lemma \ref{dx-bounds}, implies that $|\psi(\gamma)| \leq M'''$. But if $\gamma \in \Gamma$ then the coordinates $\psi(\gamma)$ are all integers, so the result follows.\endproof

\begin{lemma}[Nilmanifolds are bounded]\label{nil-bounded}
Let $Q \geq 2$, and suppose that $\mathcal{X}$ is a $Q$-rational Mal'cev basis for a nilmanifold $G/\Gamma$ \textup{(}with respect to some filtration\textup{)}. Then $d(x\Gamma,y\Gamma) \ll Q^{O(1)}$ uniformly in $x,y \in G$.
\end{lemma}
\proof By Lemma \ref{fund-dom-def} we may choose $\gamma$ and $\gamma'$ so that $|\psi(x\gamma)|, |\psi(y\gamma')| \leq 1$. The claim now follows immediately from Lemma \ref{dx-bounds}.\endproof

The final result of this appendix is not used in this paper but is required in \S 2 of the companion paper \cite{ukmobius}.

\begin{lemma}[Comparison of metrics on nilmanifolds]
Let $Q \geq 2$. Suppose that $G' \subseteq G$ is a closed subgroup and that $\mathcal{X},\mathcal{X}'$ are $Q$-rational Mal'cev bases for $G/\Gamma$ and $G'/\Gamma'$ respectively such that each $X'_i$ is a $Q$-rational combination of the $X_i$. Let $d,d'$ be the metrics induced on $G/\Gamma$ and $G'/\Gamma'$ respectively. Then for any $x,y \in G'$ we have
\[ d'(x\Gamma',y\Gamma') \ll Q^{O(1)}d(x\Gamma,y\Gamma) \] and \[ d(x\Gamma,y\Gamma) \ll Q^{O(1)} d'(x\Gamma',y\Gamma').\]
\end{lemma}
\proof We prove the second inequality first. By the proof of Lemma \ref{quotient-metric} there is some $\gamma' \in \Gamma'$ such that $d'(x\Gamma',y\Gamma') = d'(x,y\gamma')$. Here we may assume, using Lemma \ref{fund-dom-def}, that $|\psi'(x)|,|\psi'(y)| \leq 1$. By Lemma \ref{nil-bounded} we have $d'(x,y\gamma') \leq Q^{O(1)}$, and therefore by Lemma \ref{dx-bounds} and the triangle inequality we have $d'(\id_{G'},y\gamma') \ll Q^{O(1)}$. By a second application of Lemma \ref{dx-bounds} it follows that $|\psi'(y\gamma')| \ll Q^{O(1)}$. By Lemma \ref{comparison-lemma} we therefore have $d(x,y\gamma') \ll Q^{O(1)}d'(x,y\gamma')$. Since $\Gamma' \subseteq \Gamma$, this implies that
\[ d(x\Gamma,y\Gamma) \leq d(x,y\gamma') \ll Q^{O(1)}d'(x,y\gamma') = Q^{O(1)} d'(x\Gamma',y\Gamma'),\] which is the second inequality claimed.

To prove the first inequality we make the same initial manoeuvres. That is, we may assume that $|\psi(x)|,|\psi(y)| \leq 1$ and that there is some $\gamma \in \Gamma$ such that $d(x\Gamma,y\Gamma) = d(x,y\gamma)$. Let $C$ be a constant to be specified later. If $d(x,y\gamma) \geq Q^{-C}$ then, by Lemma \ref{nil-bounded}, the bound is trivial. Suppose, then, that $d(x,y\gamma) < Q^{-C}$. This is an assertion to the effect that $\gamma$ lies ``near'' $G'$. We will use the rationality properties of the coordinates of $\Gamma$ to conclude from this that $\gamma$ must actually lie in $G'$. 

 By Lemma \ref{approx-left} and Lemma \ref{mult-basis-lem} we obtain $d(z,\gamma) \ll Q^{O(1) - C}$, where $z := y^{-1}x$. Since $d(z,\id_G) \ll Q^{O(1)}$ we have $d(\gamma,\id_G) \ll Q^{O(1)}$, and so by Lemma \ref{dx-bounds} it follows that $|\psi(z) - \psi(\gamma)| \ll Q^{O(1) - C}$. It follows from this and Lemma \ref{1-2-lem} that \begin{equation}\label{near}|\psi_{\exp}(z) - \psi_{\exp}(\gamma)| \ll Q^{O(1) - C}.\end{equation}
Now $G'$ is defined, in exponential or type I coordinates, as the intersection of the kernels of $O(1)$ linear forms with rational coefficients of height $O(Q^{O(1)})$. The coordinates $\psi(\gamma)$ are integers and so the type I coordinates $\psi_{\exp}(\gamma)$ are, by Lemma \ref{1-2-lem}, rationals of height $O(Q^{O(1)})$. The element $z$, of course, lies in $G'$. If $C$ is chosen sufficiently large, it follows from these observations and \eqref{near} that indeed $\gamma$ lies in $G'$ and hence in $\Gamma'$.

We now have that $d(x,y\gamma') \ll Q^{O(1)}$, where $\gamma' = \gamma$ lies in $\Gamma'$. One final application of Lemma \ref{comparison-lemma} implies that $d'(x,y\gamma') \ll Q^{O(1)}d(x,y\gamma')$, from which it of course follows that \[ d'(x\Gamma',y\Gamma') \leq d'(x,y\gamma') \ll Q^{O(1)}d(x,y\gamma') = Q^{O(1)}d(x\Gamma,y\Gamma).\]
This concludes the proof.\endproof

\providecommand{\bysame}{\leavevmode\hbox to3em{\hrulefill}\thinspace}
\providecommand{\MR}{\relax\ifhmode\unskip\space\fi MR }
\providecommand{\MRhref}[2]{%
  \href{http://www.ams.org/mathscinet-getitem?mr=#1}{#2}
}
\providecommand{\href}[2]{#2}

     \end{document}